\documentclass[11pt,letterpaper]{article}
\usepackage[utf8]{inputenc}
\usepackage[margin=1.0in]{geometry}
\usepackage{setspace}
\usepackage{alphalph}

\usepackage{amsmath}
\usepackage{amsthm}
\usepackage{amssymb}
\usepackage{mathrsfs}
\usepackage{graphicx}
\usepackage{dsfont}
\usepackage{enumerate}
\usepackage{mathabx}
\usepackage{mathdots}
\usepackage{authblk}
\usepackage{titlesec}
\usepackage{soul}
\usepackage{tikz-cd}
\usepackage{mathtools}
\usepackage{slashed}

\usepackage{hyperref}
\hypersetup{
    colorlinks=true,
    urlcolor=cyan,
    linkcolor=blue,
    citecolor=blue,
}

\newcounter{bigthm}

\newtheorem{bigtheorem}[bigthm]{Theorem} 

\newtheoremstyle{citedtheorem}%
  {3pt}
  {3pt}
  {\itshape}
  {}
  {\bfseries}
  {.}
  {.5em}
  {\thmname{#1} \thmnumber{#2} \thmnote{\normalfont#3}}

\newtheoremstyle{citeddefinition}%
  {3pt}
  {3pt}
  {\normalfont}
  {}
  {\bfseries}
  {.}
  {.5em}
  {\thmname{#1} \thmnumber{#2} \thmnote{\normalfont#3}}

\newtheoremstyle{citedexample}%
  {3pt}
  {3pt}
  {\normalfont}
  {}
  {\itshape}
  {.}
  {.5em}
  {\thmname{#1} \thmnumber{#2} \thmnote{\normalfont#3}}

\theoremstyle{plain}
\newtheorem{theorem}{Theorem}[section]
\newtheorem{lemma}[theorem]{Lemma}
\newtheorem{corollary}[theorem]{Corollary}
\newtheorem{proposition}[theorem]{Proposition}
\newtheorem{definition}[theorem]{Definition}

\theoremstyle{citedtheorem}

\theoremstyle{definition}

\theoremstyle{citeddefinition}

\theoremstyle{remark}
\newtheorem{remark}[theorem]{Remark}
\newtheorem{example}[theorem]{Example}

\theoremstyle{citedexample}

\newcommand{\bB}{\mathbb{B}}
\newcommand{\bC}{\mathbb{C}}

\newcommand{\bN}{\mathbb{N}}

\newcommand{\bR}{\mathbb{R}}
\newcommand{\bS}{\mathbb{S}}

\newcommand{\bV}{\mathbb{V}}

\newcommand{\bZ}{\mathbb{Z}}

\newcommand{\cB}{\mathcal{B}}

\newcommand{\cD}{\mathcal{D}}
\newcommand{\cE}{\mathcal{E}}

\newcommand{\cM}{\mathcal{M}}

\newcommand{\cQ}{\mathcal{Q}}
\newcommand{\cR}{\mathcal{R}}
\newcommand{\cS}{\mathcal{S}}
\newcommand{\cT}{\mathcal{T}}

\newcommand{\cV}{\mathcal{V}}

\newcommand{\fF}{\mathfrak{F}}

\newcommand{\fo}{\mathfrak{o}}

\newcommand{\fs}{\mathfrak{s}}

\newcommand{\sD}{\mathscr{D}}
\newcommand{\sE}{\mathscr{E}}

\newcommand{\sG}{\mathscr{G}}

\newcommand{\sJ}{\mathscr{J}}

\newcommand{\sN}{\mathscr{N}}
\newcommand{\sO}{\mathscr{O}}

\newcommand{\sR}{\mathscr{R}}
\newcommand{\sS}{\mathscr{S}}

\newcommand{\sU}{\mathscr{U}}
\newcommand{\sV}{\mathscr{V}}

\renewcommand{\a}{\alpha} 
\renewcommand{\b}{\beta} 

\newcommand{\g}{\gamma}

\newcommand{\n}{\nu}
\newcommand{\x}{\xi}

\newcommand{\vp}{\varphi}

\renewcommand{\l}{\ell} 

\newcommand{\Z}{\mathbb{Z}}

\newcommand{\R}{\mathbb{R}}
\newcommand{\C}{\mathbb{C}}

\renewcommand{\d}{\partial} 

\newcommand{\sub}{\subseteq}
\newcommand{\setm}{\, \setminus \,}

\newcommand{\<}{\left\langle}
\renewcommand{\>}{\right\rangle}

\DeclareMathOperator{\Hom}{Hom}
\DeclareMathOperator{\End}{End}
\DeclareMathOperator{\id}{id}

\DeclareMathOperator{\rank}{rank}
\DeclareMathOperator{\Tr}{Tr}

\let\Re\relax
\DeclareMathOperator{\Re}{Re}

\DeclareMathOperator{\Cl}{C\l}

\DeclareMathOperator{\spec}{spec}

\DeclareMathOperator{\Spin}{Spin}

\DeclareMathOperator{\SO}{SO}

\title{\vspace{-20pt}A Calder\'{o}n Problem for the Dirac operator with chiral boundary conditions}
\author{Carlos Valero}

\begin{document}

\maketitle

\begin{abstract}
    We consider on a spin manifold with boundary a Dirac operator $D_A$ with chiral boundary conditions, twisted by a unitary connection $A$. When $m$ is not in the chiral spectrum of $D_A$, we define an analogue of the Dirichlet-to-Neumann map for the Dirac equation $D_A - m$, which we call the boundary conjugation map, and show that it is a pseudodifferential operator of order $0$ on the boundary. We show that in dimension greater than $2$, its symbol determines the Taylor series of the metric and connection modulo gauge on the boundary when $m \neq 0$ and $m^2$ is not in the Dirichlet spectrum of $D_A^2$. We go on to show that a real-analytic Riemannian manifold and twisted spinor bundle with twisted spin connection can be recovered from its boundary conjugation map. Under further hypotheses, one can recover the unitary connection up to global gauge equivalence and the complex spinor bundles. Similar results hold in dimension $2$ when the auxiliary bundle and connection are absent.
\end{abstract}

\tableofcontents

\newpage

\section{Introduction}\label{sec: intro}

An archetypal problem in the field of inverse problems is the Calder{\'o}n problem, which was first introduced by Alberto Calder{\'o}n in \cite{Calderon1980}, where he considers the question of determining the electrical conductivity of a material by measuring the electric potential and current on its boundary. If one considers anisotropic conductivities, one can recast this question in the language of differential geometry as follows: can we recover a compact Riemannian manifold $(M,g)$ with boundary up to isometry from the set of Cauchy data of harmonic functions on the boundary? The Cauchy data is conveniently captured by the Dirichlet-to-Neumann map $\Lambda_g : C^\infty(\d M) \to C^\infty(\d M)$, which sends $f \in C^\infty(\d M)$ to the normal derivative $\d_n \vp |_{\d M}$, where $\vp \in C^\infty(M)$ is the unique solution to $\Delta \vp = 0$ with $\vp|_{\d M} = f$. The Calder{\'o}n problem can therefore also be formulated by asking if a compact Riemannian manifold is determined up to isometry by its Dirichlet-to-Neumann map. 

The purpose of the present paper is to introduce a Calder{\'o}n problem for the Dirac operator, a first-order elliptic differential operator acting on sections of a complex vector bundle that plays a fundamental role in geometry and physics. We consider well-posed local boundary conditions for the Dirac operator on a spin manifold with boundary, and define an analogue of the Dirichlet-to-Neumann map for this first-order boundary value problem, which we call the boundary conjugation map. The Calder{\'o}n problem for the Dirac operator then takes the following form: can we recover a Riemannian manifold, unitary connection, and other geometric structures from the boundary conjugation map associated to the Dirac operator with these boundary conditions? We shall see in this paper that the answer is yes when all the data are real-analytic, by adapting methods used to tackle other Calder{\'o}n problems for second-order operators to our first-order setting.

\vspace{5pt}

There is abundant literature on the Calder{\'o}n problem and its many generalizations to other geometric settings. The anisotropic Calder{\'o}n conjecture for the scalar Laplacian has been proved in dimension $2$ for smooth data \cite{LassasUhlmann2001}; in dimensions greater than $2$, the conjecture remains open in the smooth case, though it has been proven for real-analytic data \cite{LassasUhlmann2001}. The first key step in the proof comes in the classical work of Lee and Uhlmann \cite{Lee1989}, wherein it is shown that the Dirichlet-to-Neumann map is a pseudodifferential operator of order $1$ whose symbol determines the Taylor series of the metric at the boundary. Building on this, Lassas and Uhlmann proved in \cite{LassasUhlmann2001} that the Dirichlet-to-Neumann map determines the Green's functions on an extended manifold, and then recover the manifold as a suitable quotient of a sheaf constructed from the Green's functions. In \cite{LTU2003}, the authors use a different approach, still based on the real-analyticity of the Green's functions, to extend this reconstruction result to complete manifolds with boundary. Here, however, the Green's functions are used to embed two manifolds with equal Dirichlet-to-Neumann maps into a suitable Sobolev space, from which an isometry can then be constructed. This latter method seems easier to adapt to other geometric settings, and indeed has been used in \cite{KLU2011} to recover a real-analytic Riemannian manifold from the Dirichlet-to-Neumann map of its Hodge Laplacian acting on differential forms; in \cite{LTS2022} to recover the conformal class of a real-analytic manifold from the Cauchy data of its conformal Laplacian; and in \cite{Gabdurakhmanov2025} to recover a real-analytic manifold and connection from the Cauchy data of the connection Laplacian. In all of these case, the geometric structure can be recovered from {\em local} data, that is to say knowledge of the Dirichlet-to-Neumann map on an open subset of the boundary. A version of this argument is also used in \cite{EV2024} where a Calder{\'o}n problem for the curl operator on a Riemannian $3$-manifold is considered. There, in lieu of Green's functions, which do not exist for the curl operator, small current loops are used to recover a real-analytic simply connected $3$-manifold from global boundary data. A key idea in \cite{EV2024} is to first relate the Calder{\'o}n problem for the first-order curl operator to that of the second order Hodge Laplacian, for which there are more tools available. This idea of relating a first-order operator to a second-order one shall also play an important role in the present paper, though the analysis for the Dirac operator appears more difficult than for the curl operator, due to the conformal symmetry.

On the other hand, inverse problems for Dirac operators and connections have also been studied in the literature. Using the boundary control method, Kurylev and Lassas showed in \cite{KurylevLassas2009} that a Riemannian manifold and Dirac bundle can be reconstructed from the response operator associated to the hyperbolic Dirac operator with self-adjoint boundary conditions; equivalently, instead of the response operator, one can start with knowledge of the boundary spectral data of the elliptic Dirac operator, that is, the set of eigenvalues and the boundary values of the eigenfunctions. This is in contrast to the present paper, in which we consider boundary data for the elliptic problem at a {\em fixed} non-zero frequency $m$. Other recent works that consider inverse problems for Dirac-type operators include \cite{QuanUhlmann2022}, in which the question of recovering a closed Riemannian manifold and Dirac bundle from the source-to-solution map of the fractional Dirac operator is considered, as well as \cite{Roberts2022, Roberts2022thesis}, which consider the application of the boundary control method to the Hodge--Dirac operator acting on differential forms. We also mention that there have been a number of works on inverse problems for the Dirac operator on domains in $\bR^3$ that consider the boundary conjugation map under other names; see, for example, \cite{Tsuchida1998, Li2007, Salo2008, Salo2010} and the references found therein for a more complete review. One inverse problem considered in these works is the recovery of a vector potential up to gauge equivalence from boundary values of spinors, assuming the data are $C^\infty$, or possibly less regular. As unitary connections generalize the electromagnetic vector potential, the present work can be seen as extending some of these results to arbitrary manifolds and metrics, at the cost of assuming that the data be $C^\omega$. In addition to the aforementioned work \cite{Gabdurakhmanov2025}, the Calder{\'o}n problem for connections is also considered in \cite{Cekic2017}, in \cite{Cekic2020} for Yang--Mills connections, and in \cite{Kurylev2018} for the wave equation associated to the connection Laplacian. Using the methods of \cite{Cekic2020}, the recovery of Yang--Mills--Dirac connections from the Dirichlet-to-Neumann map of the Dirac Laplacian is considered in \cite{Valero2024}, some results of which are built upon in the present paper.

Finally, we mention that while the anisotropic Calder{\'o}n problem for the scalar Laplacian remains open for smooth metrics, there are a number of uniqueness results known for $C^{\infty}$ metrics in the settings of conformally transversally anisotropic manifolds \cite{DKSU2009, DKLS2016, KS2013}, Einstein manifolds \cite{GS2009}, and conformally St{\"a}ckel manifolds \cite{DKN2021}. On the other hand, there are also a number of recent non-uniqueness results for the Calder{\'o}n problem with local or disjoint data \cite{DKN2018, DKN2019, DKN2020}, and with global $C^k$ data at non-zero frequency \cite{DHKN2024}. For a survey of uniqueness and non-uniqueness results for the anisotropic Calder{\'o}n problem, we refer the reader to \cite{DKN2018}. 










\subsection{Main Results}\label{subsec: main results}

Here, we briefly introduce the boundary conjugation map associated to the twisted Dirac equation $D_A - m$ with chiral boundary conditions acting on sections of a twisted spinor bundle over a spin manifold with boundary, and state the two main uniqueness results that we prove in this paper. A more detailed overview of the fundamental ideas and definitions are provided in Section \ref{sec: prelim}. We end this section with a brief summary of the paper.

Suppose $(M,g)$ is a Riemannian spin manifold with boundary $\d M$, endowed with a complex spinor bundle $\bS$, along with a Hermitian vector bundle $(E,h)$ carrying a unitary connection $A$. We then consider the twisted spinor bundle $\bS_E := \bS \otimes E$, which carries a twisted spin connection $\nabla^A$ and twisted Dirac operator $D_A$. Let $\gamma^n$ denote Clifford multiplication on $\bS_E|_{\d M}$ by the inward unit normal, and suppose that there exists a chirality operator $\Pi$ on $\bS_E$ satisfying \eqref{chirality operator conditions 1}--\eqref{chirality operator conditions 2}. Then we have $\bS_E|_{\d M} = \bV_E^+ \oplus \bV_E^-$, where $\bV_E^{\pm}$ is the $\pm 1$ eigenspace of $\gamma^n \Pi$. Let $\bB^{\pm} : \bS_E|_{\d M} \to \bV_E^{\pm}$ be the orthogonal projector onto $\bV_E^{\pm}$. Then, for any real $m$ not in the chiral spectrum of $D_A$, there exists a unique solution to the boundary value problem
\begin{equation}\label{chiral BVP intro}
    \begin{cases}
        (D_A - m)\psi = 0, \\
        \bB^{+} \psi = f,
    \end{cases}
\end{equation}
for $f \in C^{\infty}(\d M, \bV^+ \otimes E)$. We then define the boundary conjugation map $\Theta_{g,A,m} : \bV_E^+ \to \bV_E^-$ by
\begin{equation}
    \Theta_{g,A,m}(f) := \bB^- \psi,
\end{equation}
where $\psi$ is the unique solution to the boundary value problem \eqref{chiral BVP intro}. For an open subset $\sU \sub \d M$, we can also define a local version of the boundary conjugation map by restricting $\Theta_{g,A,m}$ to sections with support in $\sU$, and restricting the image to $\sU$ (see Definitions \ref{definition of Theta map, global} and \ref{definition of Theta map, local}).

The first result we prove is the following boundary determination theorem for the boundary conjugation map, which states that $\Theta_{g,A,m, \sU}$ determines the Taylor series of the metric and connection modulo gauge on $\sU$ when the conformal symmetry is broken in dimensions $3$ or higher, and the Taylor series of the metric in dimension $2$ when $m \neq 0$ and the connection is absent:


\begin{bigtheorem}\label{main theorem: boundary determination, intro}
    Let $(M,g)$ be an $n$-dimensional spin manifold with boundary $\d M$, and let $(E, h, A)$ be a Hermitian bundle with unitary connection over $M$. Assume that $m$ is not in the chiral spectrum of $D_A$ and that $m^2$ is not in the Dirichlet spectrum of $D_A^2$. Let $\Theta_{g,A,m,\sU}$ denote the local boundary conjugation map associated to $D_A - m$ with chiral boundary conditions as above.
    \begin{enumerate}[{\em \ \ \ (1)}]
        \item Suppose $n \geq 3$. If $m \neq 0$, then $\Theta_{g,A,m,\sU}$ determines all normal derivatives of the metric on $\sU$, and all normal derivatives of the connection in a gauge where $A_n = 0$ near $\sU$. If $m = 0$, then $\Theta_{g,A,0,\sU}$ determines $g|_{\sU}$ up to a locally constant conformal factor on $\sU$. Moreover, the normal derivatives of $g$ and $A$ are still determined by $\Theta_{g,A,0,\sU}$ if $g$ and all normal derivatives of the mean curvature are prescribed on $\sU$.

        \item Suppose that $n = 2$, and that the auxiliary bundle $E$ and connection $A$ are absent. Then $\Theta_{g,m,\sU}$ determines all normal derivatives of the metric on $\sU$ if $m \neq 0$.
    \end{enumerate}
\end{bigtheorem}

\noindent We shall see in Section \ref{sec: boundary determination} the reasons for omitting the auxiliary connection in the case $n = 2$. 

The second result we prove is the following uniqueness result for the the boundary conjugation map on a spin manifold with boundary with analytic data. Here, for $i = 1,2$, we let $(M_i, g_i)$ be a connected spin manifold with boundary, and $\bS_i$ a complex spinor bundle over $M_i$. Furthermore, let $(E_i, h_i, A_i)$ be a Hermitian bundle over $M_i$ endowed with a unitary connection $A_i$. We assume that $\d M_1$ and $\d M_2$ are identified over some open set $\sU$, and that the corresponding bundles are identified over $\sU$. Then we have the following: 

\begin{bigtheorem}\label{main theorem: reconstruction, intro}
    For $i=1,2$, let $(M_i, g_i, \bS_i, E_i, h_i, A_i)$ be as above. If $\dim{M_i} = 2$, then assume that $\bS_i$ has rank $2$, $E_i = \C$, and $A_i$ is the canonical flat connection. Assume that $\Theta_{g_1,A_1,m, \sU}$ is equivalent to $\Theta_{g_2, A_2, m, \sU}$ in the sense of Definition \ref{definition of equivalent Theta maps} for some non-zero $m$ not in the chiral spectrum of either Dirac operator, with $m^2$ not in the Dirichlet spectrum of either Dirac Laplacian. Then there exists a real-analytic unitary isomorphism $\Phi : \bS_1 \otimes E_1 \to \bS_2 \otimes E_2$ covering an isometry $\vp : M_1 \to M_2$ such that $\Phi$ intertwines the twisted spin connections: $\Phi^*\nabla^{A_2} = \nabla^{A_1}$. Also, there exist a complex line bundle $L$ and a unitary isomorphism $\Phi_E : E_1 \otimes L \to E_2$ such that $\Phi_E^* \nabla^{A_2} = \nabla^{A_1} \otimes \nabla^L$ for some flat unitary connection $\nabla^L$ on $L$.
\end{bigtheorem}

In particular, Theorem \ref{main theorem: reconstruction, intro} implies that $A_1$ and $A_2$ are locally gauge equivalent about any point. As we shall see in Section \ref{sec: recovery}, the complex line bundle $L$ arises as an obstruction to global recovery of the auxiliary bundles, due to the possibility of $\bS_1$ and $\bS_2$ being induced by a priori distinct spin structures. Further assumptions guarantee a global gauge equivalence:

\begin{corollary}\label{corollary: global gauge equiv, intro}
    If in Theorem \ref{main theorem: reconstruction, intro} we have either $\bS_1 \cong \bS_2$ or $E_1 \cong E_2$, then the bundle $L$ in Theorem \ref{main theorem: reconstruction, intro} is trivial and $\Phi_E$ is a global gauge equivalence between $A_1$ and $A_2$. If $E_1 \cong E_2$, then there exists a real-analytic isomorphism of Clifford module bundles $\Phi_S : \bS_1 \to \bS_2$.
\end{corollary}

In particular, when $E_1 \cong E_2$ and the spinor bundles are irreducible as Clifford modules over an even-dimensional manifold, then the $\Spin^c$ structures inducing $\bS_1$ and $\bS_2$ are isomorphic. Since we are dealing with complex spinor bundles, it is not generally possible to show that the Spin structure themselves are equivalent; see Remark \ref{remark: recover spin c} for further details. 

\vspace{5pt}

The structure of the paper is as follows. In Section 2, we recall the fundamentals of twisted Dirac operators, introduce chiral boundary conditions and the boundary conjugation map, and briefly recall some important results from the pseudodifferential calculus. In Section 3, we prove our main boundary determination result for the boundary conjugation map, Theorem \ref{main theorem: boundary determination, intro}. To this end, in Section \ref{subsec: relation between Lambda and Theta} we first obtain a relation between the Dirichlet-to-Neumann map $\Lambda_{g,A,m,\sU}$ of the Dirac Laplacian and the boundary conjugation map. Then, we compute the symbol of $\Lambda_{g,A,m,\sU}$ using the methods of \cite{Lee1989}, which have also been used for the Hodge Laplacian \cite{JoshiLionheart2005} and the connection Laplacian \cite{Cekic2017}. Although the symbol of $\Lambda_{g,A,m,\sU}$ has already been computed in principle in \cite{Valero2024}, we will need a much more precise form of the symbol than what is given in \cite{Valero2024}, and obtaining this form is the goal of Section \ref{subsec: symbol of Lambda}. In Section \ref{subsec: symbol of Theta}, we use the relation between $\Lambda_{g,A,m,\sU}$ and $\Theta_{g,A,m,\sU}$ to compute the symbol of $\Theta_{g,A,m,\sU}$ and prove Theorem \ref{main theorem: boundary determination, intro} in dimensions greater than $2$. We end Section \ref{sec: boundary determination} with some comments on how to prove Theorem \ref{main theorem: boundary determination, intro} in dimension $2$. Finally, in Section \ref{sec: recovery}, we introduce the Green's kernels for the Dirac operator with chiral boundary conditions and use them to prove Theorem \ref{main theorem: reconstruction, intro} by embedding the manifolds and vector bundles into an appropriate Sobolev space, following a similar line of reasoning as in \cite{LTU2003, KLU2011, Gabdurakhmanov2025, EV2024}.

\section{Preliminaries}\label{sec: prelim}

In this section, we gather some of the preliminary ideas and results that shall be needed later on in the paper. In Section \ref{subsec: spinors intro}, we recall some basic definitions from spin geometry, and define the twisted Dirac operator. In Section \ref{subsec: boundary conditions}, we shall introduce the chiral boundary value problem for the Dirac operator along with the corresponding analogue of the Dirichlet-to-Neumann map, which we call the boundary conjugation (BC) map, $\Theta_{g,A,m}$. We end in Section \ref{subsec: symbol calculus} with a brief overview of the symbol calculus of pseudodifferential operators, which we shall use of in Section \ref{sec: boundary determination}.

\subsection{Spinor bundles and Dirac operators}\label{subsec: spinors intro}

Recall that the Clifford algebra $\Cl(V)$ corresponding to a real inner product space $(V,\<\cdot, \cdot\>)$ is the algebra generated linearly and multiplicatively by all $v \in V$, subject to the relation
\begin{equation}
    v \cdot w + w \cdot v = -2\< v, w \>.
\end{equation}
Thus, we may also consider $\Cl(V)$ to be generated by an orthonormal basis $(e_i)$ whose elements pairwise anti-commute and square to $-1$. If $V$ has dimension $n$, then $\Cl(V)$ has dimension $2^n$ and is isomorphic as a vector space to the exterior algebra $\bigwedge V$. If $(M,g)$ is a Riemannian manifold, then we define the Clifford bundle $\Cl(M)$ to be the bundle of algebras whose fiber at each point $x \in M$ is the Clifford algebra $\Cl(T_x M)$ of the corresponding tangent space. The metric and Levi-Civita connection extend in a natural way to a metric and connection on the Clifford bundle.

We now want to introduce spinor bundles on manifolds. To this end, recall that the the {\em spin group} $\Spin(n) \subseteq \Cl(\bR^n)$ is defined to be the group generated multiplicatively by all elements of the form $v_1 \cdots v_{2k}$, where each $v_i$ is a unit vector in $\bR^n$. The importance of the spin group lies in the following fact: there exists a covering map $\Spin(n) \to \SO(n)$ with kernel $\{1,-1\}$. Thus, $\Spin(n)$ is a double cover of $\SO(n)$, and in particular is the universal cover of $\SO(n)$ when $n \geq 3$. 

Suppose now that $M$ is an $n$-dimensional oriented Riemannian manifold, and let $P_{\mathrm{SO}}$ be its principal $\SO(n)$-bundle of oriented orthonormal frames. A {\em spin structure} on $M$ is defined to be a principal $\Spin(n)$-bundle $P_{\Spin}$ along with a double covering map $P_{\Spin} \to P_{\mathrm{SO}}$ that is equivariant on each fiber with respect to the double covering map $\Spin(n) \to \SO(n)$ described above. A spin structure need not exist on a given manifold, and when one does, it need not be unique. If there exists a spin structure on a manifold $M$, we say that it is spin. A well-known necessary and sufficient condition for $M$ to be spin is the vanishing of its second Stiefel--Whitney class $w_2(M)$. In this case, the set of spin structures on $M$ are in one-to-one correspondence with $H_1(M, \bZ_2)$.

Now, given a complex left module $\cM$ for the Clifford algebra $\Cl(\bR^n)$, we may restrict the action to $\Spin(n)$ to obtain a complex representation of the spin group, $\mu : \Spin(n) \to GL(\cM)$. The associated complex vector bundle $\bS := P_{\Spin} \times_{\mu} \cM$ is then called the (complex) {\em spinor bundle} associated with the representation $\mu$. In the event that the representation $\mu$ is obtained as the restriction of an irreducible module of $\Cl(\bR^n)$, then the corresponding spinor bundle is called the {\em canonical} spinor bundle of $M$, though it is often referred to in the literature as simply {\em the} spinor bundle. The canonical spinor bundle has rank $2^{\frac{n}{2}}$ when $n$ is even and $2^{\frac{n-1}{2}}$ when $n$ is odd. Note that the spinor bundles here all depend on a choice of spin structure $P_{\Spin}$.

As a spinor bundle $\bS$ is a $\Cl(M)$-module by construction, we have a map $\gamma : TM \to \End(\bS)$ called the {\em Clifford multiplication}, which extends to the algebra $\Cl(M)$ and satisfies
\begin{equation}
    \gamma(X) \gamma(Y) + \gamma(Y) \gamma(X) = - 2 g(X,Y)
\end{equation}
for all vector fields $X,Y$ on $M$. Thanks to the isomorphism $TM \to T^*M$ induced by the metric, we also have an associated map $\gamma : T^*M \to \End(\bS)$, which we shall denote by the same symbol. We also mention that it is always possible to equip $\bS$ with a Hermitian metric $\< \cdot, \cdot \>$ such that
\begin{equation}
    \< \gamma(e) \psi_1, \psi_2 \> = -\< \psi_1, \gamma(e) \psi_2 \>
\end{equation}
for all $e \in TM$ with $|e| = 1$. In other words, Clifford multiplication by unit vectors is unitary. We will assume henceforward that our spinor bundles are equipped with such a Hermitian metric.

We now want to introduce the spin connection $\nabla^{s}$ on $\bS$, which is obtained as a lifting of the Levi-Civita connection on $TM$ as follows. Let $(e_i)$ be a local oriented orthonormal frame on $M$. Then the connection $1$-forms $\omega^i_{\ j}$ with respect to this frame are defined by
\begin{equation}
    \nabla_X e_j = \omega^i_{\ j}(X) e_i
\end{equation}
and satisfy the $\fs\fo(n)$ condition $\omega^i_{\ j} = -\omega^j_{\ i}$. On the other hand, since the spin structure $P_{\Spin}$ covers the orthonormal frame bundle of $M$, and since we also have an embedding of $P_{\Spin}$ into the orthonormal frame bundle of $\bS$, the orthonormal frame $(e_i)$ lifts to a local orthonormal frame $(\sigma_{\a})$ for $\bS$, which we call a canonical spinor frame determined by $(e_i)$. In fact, there are two such frames, the other being $(-\sigma_{\a})$, but the following shall not depend on our choice of lift. Now, the spin connection $\nabla^s$, being the lift of the Levi-Civita connection to $\bS$, is locally represented with respect to this orthonormal frame of spinors $(\sigma_{\a})$ as follows:
\begin{equation}\label{spin connection defn}
    \nabla^s_{X} \sigma_{\a} = -\frac{1}{2} \sum_{i < j} \omega^i_{\ j}(X) \gamma(e_i) \gamma(e_j) \sigma_{\a}.
\end{equation}
Note that here, we have explicitly written the sum over $i$ and $j$, as working with an orthonormal frame can often render the index summation convention confusing in some expressions. We shall thus henceforward write out the sum explicitly whenever such expressions arise.

We can now define the Dirac operator $D$ on the spinor bundle $\bS$ as follows. Let $(e_i)$ be a local orthonormal frame on $M$, and let $\psi \in \Gamma(\bS)$. We then define $D \psi$ by
\begin{equation}\label{defn dirac operator}
    D \psi := \sum_{i} \gamma(e_i) \nabla^s_{e_i} \psi.
\end{equation}
This definition does not depend on the choice of local orthonormal frame, and yields a well-defined first-order differential operator $D : \Gamma(\bS) \to \Gamma(\bS)$. Before discussing some of the properties of this operator, however, we wish to consider a slightly more general setting that is often of interest in physical models, in which we twist our spinor bundle by an auxiliary bundle $E$ equipped with a connection $A$, thus obtaining a twisted, or minimally-coupled, Dirac operator $D_A$.

Therefore, let us consider a Hermitian bundle $(E,h)$ equipped with a unitary connection $\nabla^A$. We shall often use the letter $A$ to denote both the local connection $1$-form with respect to a local trivialization, as well as the abstract connection itself. Let $\bS_E$ denote the twisted spinor bundle $\bS \otimes E$, which carries the connection induced by the spin connection $\nabla^s$ and the unitary connection $A$. For simplicity, we shall also refer to this induced connection on $\bS_E$ as $\nabla^A$. Thus, choosing a local orthonormal frame $(e_i)$ on $M$, a spin frame $(\sigma_{\a})$ determined by $(e_i)$, and a local orthonormal frame $(v_k)$ for $E$ with respect to which the $U(N)$-valued connection $1$-form is denoted by $A$, the connection $1$-form $\kappa_A$ for $\nabla^A$ on $\bS_E$ with respect to the frame $(\sigma_{\a} \otimes v_k)$ is given by 
\begin{equation}\label{twisted connection 1 form}
    \kappa_A(X) = -\frac{1}{2} \sum_{i<j} \omega^i_{\ j}(X)\gamma(e_i) \gamma(e_j) \otimes \id_E + \id_{\bS} \otimes\, A(X).
\end{equation}
Note that Clifford multiplication on $\bS_E$ is defined by acting only on the first factor $\bS$. Thus, as long as there is no danger of confusion, we shall simply write $\gamma(X)$ instead of $\gamma(X) \otimes \id_E$. Similarly, as $A(X)$ can only act on the $E$ factor, we shall also often omit writing the tensor with $\id_{\bS}$. 

The twisted Dirac operator $D_A : \Gamma(\bS_E) \to \Gamma(\bS_E)$ is then defined in analogy with \eqref{defn dirac operator},
\begin{equation}
    D_A \psi := \sum_i \gamma(e_i) \nabla_{e_i}^A \psi.
\end{equation}
The Dirac operator $D_A$ is a first-order elliptic differential operator, whose principal symbol is given by $\sigma(D_A)(x,\x) = i\gamma(\x)$. Note that the square of $\sigma(D_A)$ is $|\x|^2$ times the identity, which is also the principal symbol of the connection Laplacian $(\nabla^A)^*\nabla^A$. In fact, there is a fundamental result called the Lichnerowicz--Weitzenb\"ock formula, which tells us that the difference between $D_A^2$ and $(\nabla^A)^*\nabla^A$ is an operator of order $0$ involving only the curvature of $g$ and $A$. In order to state this result, we first define the Weitzenb\"ock curvature operator $\fF_A \in \End(\bS_E)$ associated to $A$ by
\begin{equation}
    \fF_A := \frac{1}{2} \sum_{j,k} \gamma(e_j) \gamma(e_k) F_A(e_j, e_k) = \sum_{j < k} \gamma(e_j) \gamma(e_k) F_A(e_j, e_k),
\end{equation}
where $F_A \in \Omega^2 \otimes \End(E)$ is the curvature $2$-form of $A$, and $(e_i)$ is any local orthonormal frame on $M$. Then the Lichnerowicz--Weitzenb\"ock formula says that
\begin{equation}\label{Lichnerowicz--Weitzenbock identity}
    D_A^2 = \left( \nabla^A \right)^*\nabla^A + \frac{1}{4} R + \fF_A,
\end{equation}
where $R$ denotes the scalar curvature of $g$; see \cite[\S II.8]{LawsonMichelsohn}, for example.

We end this section by discussing some important symmetries of the twisted Dirac operator. Let $(E',h')$ be another Hermitian vector bundle with unitary connection $A'$, and let $\Phi : E \to E'$ be a unitary isomorphism over $M$ such that $\Phi^* \nabla^{A'} = \nabla^A$. Then we have
\begin{equation}\label{conjugation of D_A by Phi}
    D_{A'} = \Phi \circ D_A \circ \Phi^{-1},
\end{equation}
where $\Phi$ is understood to act trivially on the $\bS$ factor. We call such isomorphisms gauge transformations, and say that $A$ and $A'$ are gauge equivalent connections. Moreover, if $\vp : (M,g) \to (M',g')$ is an isometry, then it induces an isomorphism between the orthonormal frame bundles, and thus also between corresponding spin structures. In particular, $\vp$ induces an isomorphism $\hat{\vp} : \bS_{g} \to \bS_{g'}$ between the corresponding spinor bundles, and one has
\begin{equation}\label{conjugation of D_A by isometry}
    D_{g',A'} = \hat{\vp} \circ D_{g,A} \circ \hat{\vp}^{-1},
\end{equation}
where here, $\hat{\vp}$ acts simply as the pullback by $\vp$ on $E$. 

Finally, and perhaps most importantly, we want to understand how the Dirac operator behaves under conformal scalings of the Riemannian metric, as this will play key role in our analysis. One can find a more detailed discussion of this subject in \cite{Hijazi1986}, the main results of which we summarize here. Recall that we say a metric $g$ is conformal to another metric $\bar{g}$ if we have $g = e^{2u} \bar{g}$ for some smooth function $u \in C^{\infty}(M)$. We have a map from the orthonormal frame bundle of $g$ to that of $\bar{g}$, given by sending $(e_i)$ to $(\bar{e}_i)$, where $e_i = e^{-u}\bar{e}_i$. This induces an isomorphism between spin structures, and thus between spinor bundles $\bS_g \to \bS_{\bar{g}}$, which we denote by $\psi \mapsto \bar{\psi}$. Identifying $\bS_g$ with $\bS_{\bar{g}}$ via this isomorphism allows us to compare the Dirac operators of $g$ and $\bar{g}$.

The Clifford multiplication map $\bar{\gamma}$ on $\bS_{\bar{g}}$ is related to the one on $\bS_{g}$ by
\begin{equation}\label{gamma matrix under conformal scaling}
    \gamma(X) = e^u \bar{\gamma}(X), \ \ \ \ \ \ \ \gamma(\x) = e^{-u} \bar{\gamma}(\x)
\end{equation}
for $X \in TM$ and $\x \in T^*M$. Suppose $\omega^i_{\ j}$ is the Levi-Civita connection $1$-form of $g$ with respect to an orthonormal frame $(e_i)$, while $\bar{\omega}^i_{\ j}$ is that of $\bar{g}$ with respect to the orthonormal frame $(\bar{e}_i)$. Then
\begin{equation}\label{connection 1 form under conformal scaling}
    \omega^i_{\ j}(X) = \bar{\omega}^i_{\ j}(X) + du(\bar{e}_j)\bar{g}(\bar{e}_i, X) - du(\bar{e}_i)\bar{g}(\bar{e}_j, X).
\end{equation}
This and equation \eqref{spin connection defn} imply that the spin connections $\bar{\nabla}^s$ and $\nabla^s$ are related by \cite{Hijazi1986}
\begin{equation}\label{spin connection under conformal change}
    \overline{\nabla^s_X \psi} = \bar{\nabla}^s_X \bar{\psi} + \frac{1}{2} \bar{\gamma}(du) \bar{\gamma}(X) \bar{\psi} + \frac{1}{2} X(u) \bar{\psi}.
\end{equation}
Note that equation \eqref{spin connection under conformal change} continues to hold with $\nabla^s$ replaced by the twisted connection $\nabla^A$ on $\bS_E$, since the conformal scaling does not affect the $E$ factor. Hence, one has that the Dirac operators $D_A$ and $\bar{D}_A$ corresponding to metrics $\bar{g}$ and $g$ are related by
\begin{equation}\label{Dirac operator under conformal change}
    \overline{D_A \psi} = e^{-\frac{n+1}{2} u}  \bar{D}_A \left( e^{\frac{n-1}{2} u} \bar{\psi} \right),
\end{equation}
or equivalently,
\begin{equation}\label{Dirac operator under conformal change v2}
    \bar{D}_A \bar{\psi} = e^{u} \left( \overline{D_A \psi} - \frac{n-1}{2} \bar{\gamma}(du) \bar{\psi} \right).
\end{equation}
The conformal symmetry of $D_A$ will play an important role below.

\subsection{Chiral boundary conditions and the boundary conjugation map}\label{subsec: boundary conditions}

Now we are ready to introduce a suitable boundary value problem for the Dirac operator $D_A$ on a spin manifold with boundary and define the associated map $\Theta_{g,A,m}$, a first-order analogue of the Dirichlet-to-Neumann map for Laplacians and the principal object of study in this paper.

Boundary conditions for Dirac operators are discussed in several places in the literature; see for example \cite{BW1993, FS1998, BL2001, HMR2002, Raulot2011} and the references therein. Here we shall only mention that in general, a boundary condition for the Dirac operator $D_A$ takes the form of a $0$-th order pseudodifferential operator $B : L^2(\bS_E|_{\d M}) \to L^2(\cV)$, where $\cV$ is some vector bundle over the boundary, and one typically considers $B$ for which the corresponding boundary value problem is elliptic. For then we have that the operator $D_A$ with boundary condition $B$ is Fredholm, and has a discrete set of eigenvalues, each with finite-dimensional eigenspaces spanned by smooth eigenfunctions. A well-known boundary condition for the Dirac operator is the APS boundary condition introduced in \cite{APS1975I, APS1975II, APS1976III}, for which $B$ is an $L^2$-projector onto the eigenspaces corresponding to positive (or negative) eigenvalues of a Dirac operator induced on the boundary. While such a global boundary condition has important applications in physics and geometry, we wish instead to consider {\em local} boundary conditions for $D_A$, namely those for which $B$ is a local operator. There are two such local boundary conditions usually considered in the literature: the chiral (or chiral bag) condition, and the MIT bag condition. The MIT bag condition was introduced by physicists in the seventies as a boundary condition for the Lorentzian Dirac operator to model the confinement of quarks; see \cite{Johnson1975}, for example. A Riemannian version of MIT boundary conditions has been studied in other contexts. On the other hand, chiral boundary conditions for the Riemannian Dirac operator have also been studied in several works; we refer the reader to \cite{HMZ2001, HMR2002, Raulot2011} and the references therein for an overview of the MIT and chiral boundary conditions. While the MIT bag condition has the advantage that it can be defined on any spin manifold with boundary, the chiral boundary condition has the advantage of being self-adjoint. In this paper, we will deal only with the chiral boundary condition, though it is likely that a similar analysis to the one we carry out can also be carried out for the Dirac operator with MIT boundary conditions.

In order to define chiral boundary conditions for the Dirac operator, we shall first assume the existence on $M$ of a {\em chirality operator} $\Pi \in \End(\bS)$, which satisfies
\begin{equation}\label{chirality operator conditions 1}
    \Pi^2 = \id, \ \ \ \ \ \ \ \< \Pi \psi_1, \Pi \psi_2 \> = \< \psi_1, \psi_2 \>,
\end{equation}
\begin{equation}\label{chirality operator conditions 2}
    \nabla_X \Pi = 0, \ \ \ \ \ \ \ \ \ \gamma(X) \Pi = - \Pi \gamma(X),
\end{equation}
for all vector fields $X$ and spinor fields $\psi_1, \psi_2 \in \Gamma(\bS)$ on $M$. Such an operator induces a splitting $\bS = \bS^{+} \oplus\, \bS^{-}$, where $\bS^{\pm}$ is the $\pm 1$ eigenbundle of $\Pi$. Note that when $M$ is of even dimension, then a chirality operator always exists, namely $\gamma(\omega_n)$, where $\omega_n$ is a normalized complex volume form. Indeed, the splitting of $\bS$ induced by $\gamma(\omega_n)$ is the usual chiral decomposition of $\bS$ into irreducible components for $\Spin(n)$. Another situation in which one has a chirality operator is the following. Suppose $\Tilde{M}$ is a Lorentzian manifold with spinor bundle $\Tilde{\bS}$. If $M \sub \Tilde{M}$ is a Riemannian hypersurface and $\bS$ is the spinor bundle on $M$ induced by $\Tilde{\bS}$, then  a chirality operator on $\bS$ is given by $\gamma(T)$, where $T$ is a timelike unit normal to $M$. This setting is explored further in \cite{Herzlich1998}, for instance.

Let $e_n$ be the inward-pointing unit normal on $\d M$. Then properties \eqref{chirality operator conditions 1}--\eqref{chirality operator conditions 2} imply that $\gamma(e_n) \Pi$ is a self-adjoint involution over $\d M$, which therefore has eigenvalues $\pm 1$. Let us denote the corresponding eigenbundle of $\bS|_{\d M}$ by $\bV^{\pm}$, and $\bV^{\pm} \otimes E$ by $\bV^{\pm}_E$. Let $\bB^{\pm} : \Gamma(\bS_E|_{\d M}) \to \Gamma(\bV^{\pm}_E)$ be the orthogonal projector onto $\bV^{\pm}_E$ over $\d M$. We have
\begin{equation}\label{B pm projector}
    \bB^{\pm} \psi = \frac{1}{2} \left( 1 \pm \gamma(e_n) \Pi \right) \psi.
\end{equation}
The projector $\bB^{\pm}$ defines a self-adjoint elliptic boundary value problem for $D_A$. In particular, the restriction of $D_A$ to the space of $\psi \in H^1(\bS_E)$ such that $\bB^{\pm} \psi = 0$ is a Fredholm operator whose spectrum consists entirely of isolated real eigenvalues $\lambda_k^{\pm}$ with finite multiplicity, and which go to $\pm \infty$ as $k \to \pm \infty$ \cite{HMZ2001, HMR2002}. We shall write $\spec^{\pm}{D_A}$ for the spectrum of $D_A$ with boundary conditions $\bB^{\pm}$. Thus, by the usual arguments, for real $m \notin \spec^{\pm}{D_A}$ the boundary value problem
\begin{equation}\label{chiral BVP}
    \begin{cases}
        (D_A - m)\psi^{\pm} = 0, \\
        \bB^{\pm} \psi^{\pm} = f^{\pm},
    \end{cases}
\end{equation}
has a unique solution $\psi^{\pm} \in H^{s+1}(\bS_E)$ for any $f^{\pm} \in H^{s+\frac{1}{2}}(\bV_E^{\pm})$. Furthermore, we have by standard elliptic arguments that if $f^{\pm}$ is smooth, then $\psi^{\pm}$ is smooth up to the boundary. We are now finally ready to introduce the $\Theta_{g,A,m}$ map associated to the boundary value problem \eqref{chiral BVP}.

\begin{definition}\label{definition of Theta map, global}
    For $m \notin \spec^{\pm}{D_A}$, we define the {\bf boundary conjugation (BC) map}
    \begin{equation}
       \Theta^{\pm}_{g,A,m} : \Gamma(\bV_E^{\pm}) \to \Gamma(\bV_E^{\mp}) 
    \end{equation}
    by
    \begin{equation}
        \Theta^{\pm}_{g,A,m}\left( f^{\pm} \right) = \bB^{\mp} \psi^{\pm},
    \end{equation}
    where $\psi^{\pm}$ is the unique solution to the boundary value problem \eqref{chiral BVP}.
\end{definition}

Note that the boundary conditions $\bB^{\pm}$ are interchanged by taking $-\Pi$ as the chirality operator instead of $\Pi$. Thus, questions about $\Theta_{g,A,m}^{+}$ and $\Theta_{g,A,m}^{-}$ are completely analogous, and so we shall henceforward deal only with $\Theta^+_{g,A,m}$ and omit the plus sign in the superscript.

The BC map behaves as expected from equations under gauge transformations and isometries. In particular, if $\vp : (M, g) \to (M', g')$ is an isometry and $\hat{\vp} : \bS_{g} \to \bS_{g'}$ denotes the induced map on spinors as in \eqref{conjugation of D_A by isometry}, then for the chirality operator $\Pi' := \hat{\vp}\, \Pi \,\hat{\vp}^{-1}$, we have corresponding boundary conditions $\bB^{\pm}_{g'}$ on $M'$ which satisfies $\bB^{\pm}_{g'} = \hat{\vp}\, \bB^{\pm}_g \,\hat{\vp}^{-1}$. One can then check that
\begin{equation}\label{conjugation of Theta by isometry}
    \Theta_{g',A',m} = \hat{\vp}\circ \Theta_{g,A,m} \circ \hat{\vp}^{-1},
\end{equation}
where as in \eqref{conjugation of D_A by isometry}, $\hat{\vp}$ acts on $E$ via the pullback by $\vp$. Indeed, by \eqref{conjugation of D_A by isometry}, if $\psi^{+}$ is the unique solution to \eqref{chiral BVP}, then $\hat{\vp} (\psi^{+})$ is the unique solution to the corresponding boundary value problem on $M'$ with boundary value $\hat{\vp} (f^{+})$. Applying $\bB^{-}_{g'}$ to $\hat{\vp} (\psi^{+})$ then yields \eqref{conjugation of Theta by isometry}. By \eqref{conjugation of D_A by Phi}, similar result holds for gauge transformations $\Phi : (E,h,A) \to (E',h',A')$,
\begin{equation}\label{conjugation of Theta by gauge transformations}
    \Theta_{g,A',m} = \Phi \circ \Theta_{g,A,m} \circ \Phi^{-1},
\end{equation}
where as in \eqref{conjugation of D_A by Phi}, $\Phi$ is understood to act trivially on the $\bS$ factor. We now want to consider how the  BC map behaves under conformal scaling of the metric. Let $g$ and $\bar{g}$ be two metrics on $M$ related by $g = e^{2u} \bar{g}$ for some smooth function $u$. As in the discussion preceding equation \eqref{gamma matrix under conformal scaling}, we shall identify the spinor bundles $\bS_g$ and $\bS_{\bar{g}}$ by the canonical map $\psi \mapsto \bar{\psi}$. Via this map, we have a chirality operator $\bar{\Pi}$ on $\bS_{\bar{g}}$ induced by $\Pi$. Furthermore, since $\bar{\gamma}(\bar{e}_n) = \gamma(e_n)$ by \eqref{gamma matrix under conformal scaling}, we have that the boundary projectors $\bB^{\pm}$ are also conformally invariant.

Now, from \eqref{Dirac operator under conformal change}, one sees that only in the case $m=0$ does the equation $(D_A - m) \psi = 0$ have a conformal symmetry. Let us thus suppose that $m = 0$. Then if $\psi^+$ satisfies the boundary value problem \eqref{chiral BVP} with $m=0$, then $\bar{\psi}^+$ satisfies the boundary value problem
\begin{equation}\label{chiral BVP conformally scaled}
    \begin{cases}
        \bar{D}_A \big( e^{-\frac{n-1}{2} u} \bar{\psi}^+ \big) = 0, \\
        \bar{\bB}^{+} \big( e^{-\frac{n-1}{2} u} \bar{\psi}^+ \big) = e^{-\frac{n-1}{2} u} \bar{f}^+, 
    \end{cases}
\end{equation}
from which one obtains the following relation,
\begin{equation}\label{Theta map under conformal scaling}
    \overline{{\Theta}_{g,A, 0}(f^+)} = e^{\frac{n-1}{2} u} \bar{\Theta}_{\bar{g}, A, 0} \left( e^{-\frac{n-1}{2} u} \bar{f}^+ \right).
\end{equation}
Note that \eqref{Theta map under conformal scaling} implies that if the conformal factor $u$ is a constant on the boundary, then the we have equality, $\Theta_{g,A,0} = \bar{\Theta}_{\bar{g},A,0}$, under the canonical identification $\psi \mapsto \bar{\psi}$.

We have so far identified spinor bundles of different metrics when they were related by either an isometry or a conformal scaling, and this has allowed us to compare their $\Theta_{g,A,m}$ maps. We now want to extend this idea so that we can compare the $\Theta_{g,A,m}$ maps of two arbitrary metrics. In particular, we want to have a notion of two distinct $\Theta_{g,A,m}$ maps being equivalent. In order to identify the spinor bundles of two metrics, we can use the fact that the oriented orthonormal frame bundles of any two metrics are always isomorphic as $SO(n)$-bundles. Such an isomorphism then extends to an isomorphism of the spin structures, and also of the spinor bundles. This is explored further in \cite{Maier1997}. We thus make the following definition,

\begin{definition}\label{definition of equivalent Theta maps}
    We shall say that the maps $\Theta_{g,A,m}$ and $\Theta_{g',A',m}$ are {\em equivalent} if there is a unitary isomorphism $\chi : E|_{\d M} \to E'|_{\d M}$, as well as a map $\hat{\vp} : \bS_g|_{\d M} \mapsto \bS_{g'}|_{\d M}$ induced by an isomorphism of the oriented orthonormal frame bundles of $g$ and $g'$ over $\d M$ such that
    \begin{equation}\label{definition of eqivalent Theta maps equation}
        \Theta_{g',A',m} = \left( \hat{\vp} \otimes \chi \right) \circ \Theta_{g,A,m} \circ \left( \hat{\vp} \otimes \chi \right)^{-1}.
    \end{equation}
\end{definition}
A concrete way of understanding Definition \ref{definition of equivalent Theta maps} is as follows. We say that $\Theta_{g,A,m}$ and $\Theta_{g',A',m}$ are equivalent if we can choose orthonormal frames for $E|_{\d M}$ and $E'|_{\d M}$, along with orthonormal frames for $g$ and $g'$ on $\d M$, so that with respect to the induced canonical spinor frames, the matrix representatives of $\Theta_{g,A,m}$ and $\Theta_{g',A',m}$ are equal.

We also want to recall the more familiar Dirichlet-to-Neumann map $\Lambda_{g,A,m} : \bS_E|_{\d M} \to \bS_E|_{\d M}$ for the Dirac Laplacian \cite{Valero2024}, which will play an important role in understanding $\Theta_{g,A,m}$. For real $m^2$ not in the Dirichlet spectrum of $D_A^2$, this map is defined by $\Lambda_{g,A,m}(f) := \nabla^A_n \psi|_{\d M}$ where $\psi$ is the unique solution to the Dirichlet problem
\begin{equation}\label{Dirichlet problem for D_A 2}
    \begin{cases}
        (D_A^2 - m^2)\psi = 0, \\
        \psi|_{\d M} = f.
    \end{cases}
\end{equation}
Observe the following: if $\psi^+$ is the unique solution to the chiral boundary value problem \eqref{chiral BVP}, then by definition of $\Theta_{g,A,m}$, it also satisfies the following Dirichlet problem for the Dirac Laplacian,
\begin{equation}\label{Dirichlet problem for psi +}
    \begin{cases}
        (D_A^2 - m^2)\psi^+ = 0, \\
        \psi^+|_{\d M} = f^+ \oplus \Theta_{g,A,m}(f^+).
    \end{cases}
\end{equation}
Therefore, if we further assume that $m^2 \notin \spec{D_A^2}$, then we have
\begin{equation}\label{relation between Theta and Lambda, sec 2.2}
    \Lambda_{g,A,m} \left( f^+ \oplus \Theta_{g,A,m}(f^+) \right) = \nabla^A_n \psi^+|_{\d M},
\end{equation}
where $\psi^+$ satisfies \eqref{chiral BVP}. This observation will be central to our analysis in Section \ref{sec: boundary determination}.

Finally, we mention that for an open set $\sU \sub \d M$, one can define a local counterpart of the BC map, in complete analogy with local counterparts of the Dirichlet-to-Neumann map. The above definitions and properties carry over to this local version in the obvious way.

\begin{definition}\label{definition of Theta map, local}
    For an open set $\sU \sub \d M$ and $m \notin \spec^{\pm}{D_A}$, we define, 
    \begin{equation}
        \Theta^{\pm}_{g,A,m, \sU} : \Gamma \big(\bV_E^{\pm}|_{\sU} \big) \to \Gamma\big(\bV_E^{\mp}|_{\sU} \big)
    \end{equation}
    by
    \begin{equation}
        \Theta^{\pm}_{g,A,m, \sU}\left( f^{\pm} \right) = \bB^{\mp} \psi^{\pm} |_{\sU},
    \end{equation}
    where $\psi^{\pm}$ is the unique solution to the boundary value problem \eqref{chiral BVP} with $f^{\pm} \in C^{\infty}_c(\sU, \bV^{\pm}_E)$.
\end{definition}

Note that if we extend $\Theta_{g,A,m,\sU}(f^+)$ by $0$ to the complement of $\sU$ in $\d M$, then \eqref{relation between Theta and Lambda, sec 2.2} continues to hold over $\sU$ with $\Lambda_{g,A,m}$ replaced by its local counterpart $\Lambda_{g,A,m,\sU}$, defined by applying $\Lambda_{g,A,m}$ to sections with compact support in $\sU$ and restricting the image to $\sU$.


\subsection{Pseudodifferential operators}\label{subsec: symbol calculus}

Our proof of Theorem \ref{main theorem: boundary determination, intro} will make heavy use of the theory of pseudodifferential operators and their symbols, and thus we briefly recall here some of the fundamental results thereof. For more details and proofs, we refer the reader to one of many excellent texts on the subject, such as \cite{Taylor1981} or \cite{Treves1980}.

Let $W \sub \bR^n$ be open. We then define the symbol class $S^m(W)$ to be the space of all functions $p \in C^{\infty}(W \times \bR^n)$ that satisfy for all $\a, \b \in \bN^n$ and $(x,\x) \in W \times \bR^n$ the estimate
\begin{equation}
    \left| \d_{\x}^{\a} \d_{x}^{\b} p(x,\x)  \right| \leq C_{\a \b} \< \x \>^{m-|\a|}, 
\end{equation}
where $\< \x \> := \sqrt{1 + |\x|^2}$. Similarly, we define the symbol class $S^m(W, \bC^{k_2 \times k_1})$ to be the space of all $k_2$ by $k_1$ matrix-valued functions on $W \times \R^n$ whose entries are in $S^m(W)$. Any $p(x,\x) \in S^m(W, \bC^{k_2 \times k_1})$ then gives us a map $P : C^{\infty}_c(W, \bC^{k_1}) \to C^{\infty}(W, \C^{k_2})$ defined by
\begin{equation}\label{defn pseudodiff op}
    (Pu)(x) := \int_{\R^n} e^{i x \cdot \x} p(x,\x) \hat{u}(\x)\, d\x,
\end{equation}
where $\hat{u}$ denotes the Fourier transform of $u$. We say that $P$ is a pseudodifferential operator of order $m$ if it is given by \eqref{defn pseudodiff op} for some $p \in S^m(W, \bC^{k_2 \times k_1})$, which we call the symbol of $P$. The corresponding space of pseudodifferential operators of order $m$ is denoted $\Psi^m(W, \C^{k_1}, \C^{k_2})$. Note that if $p(x,\x)$ is a polynomial in $\x$ of order $m$ with coefficients depending on $x$, then the induced pseudodifferential operator $P$ is simply a differential operator of order $m$.

We say that $P \in \Psi^m(W, \bC^{k_1}, \C^{k_2})$ is a classical pseudodifferential operator if it is induced by a symbol $p \in S^m(W, \bC^{{k_2} \times {k_1}})$ that is given by an asymptotic series
\begin{equation}\label{defn asymptotic series}
    p(x,\x) \sim \sum_{j=0}^{\infty} p_{m-j}(x,\x),
\end{equation}
where each $p_{m-j}(x,\x)$ is positive-homogeneous of degree $m-j$ in $\x$. The precise meaning of \eqref{defn asymptotic series} is that for any large positive $N$, there is an integer $J$ such that
\begin{equation}\label{defn asymptotic series precise}
    p(x,\x) - \sum_{j=1}^J p_{m-j}(x,\x) \in S^{-M}(W, \bC^{k_2 \times k_1}).
\end{equation}
Conversely, given any sequence of functions $p_{m-j}(x,\x)$ positive-homogeneous of degree $m-j$ in $\x$ and satisfying \eqref{defn asymptotic series precise}, we can construct a classical pseudodifferential operator with symbol given by \eqref{defn asymptotic series}. In fact, we only need the functions $p_{m-j}(x,\x)$ to be defined and smooth away from the zero section, $\x = 0$. For a classical pseudodifferential operator $P$ with symbol given by an asymptotic series as in \eqref{defn asymptotic series}, we call the leading order term $p_m(x,\x)$ the principal symbol of $P$. Note that this indeed generalizes the notion of principal symbol for differential operators.

Suppose now that $\cE_1, \cE_2$ are complex vector bundles of rank $k_1, k_2$ over a smooth manifold $M$. A map $P : C^\infty_c(M, \cE_1) \to C^{\infty}(M, \cE_2)$ is called a pseudodifferential operator of order $m$ if for every local chart $W$ on $M$, and every local trivialization of $\cE_1$ and $\cE_2$ over $W$, the induced map belongs to $\Psi^m(W, \C^{k_1}, \C^{k_2})$. The space of pseudodifferential operators of order $m$ from $\cE_1$ to $\cE_2$ is denoted $\Psi^m(M, \cE_1, \cE_2)$. We note that if $P \in \Psi^m(M, \cE_1, \cE_2)$, then $P$ extends to a map $\sE'(\cE_1) \to \sD'(\cE_2)$, where $\sD'(\cE)$ denotes distributions on $\cE$, and $\sE'(\cE)$ denotes those with compact support. Moreover, such a $P$ also extends to a map between Sobolev spaces, $H^s(\cE_1) \to H^{s-m}(\cE_2)$ for all $s \in \R$.

If $P \in \Psi^m(M, \cE_1, \cE_2)$ for all $m \in \bR$, then we say that $P$ is a smoothing operator, the set of which is denoted $\Psi^{-\infty}(M, \cE_1, \cE_2)$. These operators correspond locally to symbols in $S^{-\infty}(W, \bC^{k_1 \times k_2})$, the intersection of $S^{-\infty}(W, \bC^{k_2 \times k_1})$ for al $m \in \bR$. We shall often work with pseudodifferential operators modulo smoothing, for then we have a bijection $S^m/S^{-\infty} \to \Psi^m/\Psi^{-\infty}$. Thus, the difference of two pseudodifferential operators is a smoothing operator if and only if they have the same symbol modulo $S^{-\infty}$. Moreover, if we have $P_1 \in \Psi^{m_1}(M, \cE_1, \cE_2)$ with local symbol $p_1 \in S^{m_1}(W, \bC^{k_2 \times k_1})$, and $P_2 \in \Psi^{m_2}(M, \cE_2, \cE_3)$ with local symbol $p_2 \in S^{m_2}(W, \bC^{k_3 \times k_2})$, then their composition $Q := P_2 P_1$ is well-defined in $\Psi^{m_1 + m_2}(M, \cE_1, \cE_3)$ modulo smoothing, with local symbol $q \in S^{m_1+m_2}(W, \bC^{k_3 \times k_1})$ modulo $S^{-\infty}$ given by the following asymptotic series as in \eqref{defn asymptotic series},
\begin{equation}\label{composition of symbols}
    q(x,\x) \sim \sum_{\a} \frac{(-i)^{|\a|}}{\a!} \d_{\x}^{\a}p_2(x,\x) \d_{x}^{\a}p_1(x,\x).   
\end{equation}
We shall denote the series on the right-hand-side of \eqref{composition of symbols} by $(p_1 \# p_2)(x,\x)$.

\section{Boundary determination for $\Theta_{g,A,m,\sU}$} \label{sec: boundary determination}

In this section, we shall prove Theorem \ref{main theorem: boundary determination, intro}. To do so, we shall first derive a relation between the map $\Theta_{g,A,m, \sU}$ associated to the chiral boundary conditions for $D_A-m$, and the Dirichlet-to-Neumann $\Lambda_{g,A,m, \sU}$ associated with the Dirichlet problem for $D_A^2-m^2$. We do this in Section \ref{subsec: relation between Lambda and Theta}. Then, in Section \ref{subsec: symbol of Lambda}, we carry out a careful analysis of the symbol of $\Lambda_{g,A,m,\sU}$ that will be needed for the proof of Theorem \ref{main theorem: boundary determination, intro}, which is finally given in Section \ref{subsec: symbol of Theta} for $n \geq 3$, where we compute the symbol of $\Theta_{g,A,m,\sU}$ and show that it determines the Taylor series of $g$ and $A$ at the boundary. We end in Section \ref{subsec: proof for n=2} by explaining how to modify the proof given in Section \ref{subsec: symbol of Theta} for $n=2$.

Before embarking on the proof, however, we introduce some important notation that shall be used throughout this Section. We work in boundary normal coordinates $(x^\a, x^n)$ near $\sU$, for which the $(x^\a)$ form a coordinate chart on the boundary, and $x^n$ is the (normal geodesic distance from the boundary. We shall always use Greek letters for indices that range from $1$ to $n-1$. We will sometimes also denote the coordinates tangential to the boundary $(x^\a)$ by $x'$. In boundary normal coordinates, $\d_n$ is parallel and orthogonal to the leaves of the foliation induced by the boundary, so that the metric takes the form
\begin{equation}\label{metric in BN coordinates}
    g_{nn} = 1, \ \ \ \ g_{\a n} = 0, \ \ \ \ g_{\a \b} = g_{\a \b}(x',x^n)
\end{equation}
near $\sU$. In addition to working in such coordinates, we must also choose a local trivialization for $\bS$ and $E$ near the boundary. To this end, we shall choose an orthonormal frame $(e_a, e_n)$ for $TM$ near the boundary, in which $e_n = \d_n$ and $(e_a)$ is an orthonormal frame on the boundary. We then do our computations in one of the canonical frames for $\bS$ induced by $(e_i)$. We shall always use Latin letters near the beginning of the alphabet for indices corresponding to members of this orthonormal frame, which range from $1$ to $n-1$. We use Latin letters from the middle of the alphabet when we want to include $e_n$. As we shall see, we will also have the freedom to compute the symbol of our maps with respect to different trivializations, and so we shall later specify additional conditions for our orthonormal frames to satisfy. As for $E$, we assume we are working in an orthonormal frame for which $A_n = 0$ near $\sU$, which always exists.





\subsection{A relation between $\Lambda_{g,A,m,\sU}$ and $\Theta_{g,A,m,\sU}$}\label{subsec: relation between Lambda and Theta}

Here we derive a pseudodifferential equation relating $\Lambda_{g,A,m,\sU}$ and $\Theta_{g,A,m,\sU}$. Recall from the end of Section \ref{subsec: boundary conditions} that if we have $\psi \in \Gamma(\bS_E)$ satisfying the chiral boundary value problem
\begin{equation}\label{first order BVP}
    \begin{cases}
        (D_A - m) \psi = 0, \\
        \bB^{+} \psi \big|_{\sU} = f,
    \end{cases}
\end{equation}
where $f \in C^{\infty}_c(\sU, \bV_E^+)$, then $\psi$ also satisfies
\begin{equation}
    \begin{cases}\label{second order BVP}
        (D^2_A - m^2) \psi = 0, \\
        \psi \big|_{\d M} = f \oplus \Theta_{g,A,m,\sU}\left( f \right),
    \end{cases}
\end{equation}
where we define $\Theta_{g,A,m,\sU}(f)$ to be $0$ outside $\sU$. We thus have, by definition of $\Lambda_{g,A,m,\sU}$, that
\begin{equation}\label{relation between Theta and Lambda, v1}
     \Lambda_{g,A,m,\sU}\left( f \oplus \Theta_{g,A,m,\sU}\left( f \right) \right) = \nabla^A_n \psi \big|_{\sU}.
\end{equation}
The left-hand side of \eqref{relation between Theta and Lambda, v1} is a first-order pseudodifferential operator whose symbol can be expressed in terms of the symbols of $\Lambda_{g,A,m,\sU}$ and $\Theta_{g,A,m,\sU}$ using the composition rule \eqref{composition of symbols}. Let us now consider the right-hand-side. Since $\psi$ satisfies the first-order Dirac equation \eqref{first order BVP}, we can write the covariant normal derivative of $\psi$ in terms of a local orthonormal frame $(e_a, e_n)$ near $\sU$: 
\begin{equation}\label{relation: covariant normal derivative v1}
    \nabla^A_n \psi \big|_{\d M} = \gamma^n \left( \gamma^{a} \nabla_{a}^A \psi - m \psi \right) \big|_{\sU}.
\end{equation}
Note that the derivatives on the right-hand side of \eqref{relation: covariant normal derivative v1} are all tangential to the boundary. We would like to rewrite this right-hand side in terms of the $\pm$ components of $\psi$ on the boundary, namely $f$ and $\Theta_{g,A,m,\sU}(f)$. We thus need to understand how the covariant derivative interacts with the chiral projectors $\bB^{\pm}$. For this we have the following

\begin{lemma}\label{lemma: projectors and covariant derivative}
    Let $X$ be a vector field on the boundary and $\psi \in \Gamma(\bS_E)$. Then we have
    \begin{equation}\label{projectors and clifford multiplication}
        \bB^{\pm}\left( \gamma(X) \psi \right) = \gamma(X) \bB^{\pm} \psi \pm g( X, e_n) \Pi \psi.
    \end{equation}
    Moreover, if $(e_a, e_n)$ is a local orthonormal frame near $\d M$, we have
    \begin{equation}\label{projectors and covariant derivatives eqn}
        \bB^{\pm} \nabla_a^{A} \psi = \nabla_a^{A} \left( \bB^{\pm} \psi \right) \pm \frac{1}{2} \Pi\, \nabla_a\gamma^n \cdot \psi.
    \end{equation}
\end{lemma}

\begin{proof}
    Using the definition of the projectors $\bB^{\pm}$, as well as the Clifford relations and the anti-commutativity properties of $\Pi$, we compute
    \begin{align}
        \bB^{\pm} \left( \gamma(X) \psi \right) &= \frac{1}{2} \left(1 \pm \gamma(e_n) \Pi \right) \gamma(X) \psi \nonumber \\
        &= \frac{1}{2} \left(\gamma(X) \mp \gamma(e_n) \gamma(X) \Pi \right) \psi \nonumber \\
        &= \frac{1}{2} \gamma(X) \left(1 \pm \gamma(e_n) \Pi \right) \psi \pm g(X,e_n) \Pi \psi,
    \end{align}
    whence we obtain \eqref{projectors and clifford multiplication}. As for the proof of \eqref{projectors and covariant derivatives eqn}, we again compute, using $\nabla \Pi = 0$ to obtain
    \begin{align}
        \nabla_a^A \left( \bB^{\pm} \psi \right) &= \frac{1}{2} \nabla_a^A \left( 1 \pm \gamma(e_n) \Pi \right) \psi \nonumber \\
        &= \frac{1}{2} \nabla_a^A \psi \pm \frac{1}{2} \left( \nabla_a \gamma(e_n) \right) \Pi\, \psi \pm \frac{1}{2} \gamma(e_n) \Pi \nabla_a^A \psi \nonumber \\
        &= \bB^{\pm} \left( \nabla_a^A \psi \right) \pm \frac{1}{2} \left( \nabla_a \gamma(e_n) \right) \Pi\, \psi, 
    \end{align}
    which gives us equation \eqref{projectors and covariant derivatives eqn} upon noting that $\Pi$ anti-commutes with $\nabla_a \gamma(e_n)$.
\end{proof}

\begin{remark}\label{remark: projectors and clifford relations}
    In particular, note that \eqref{projectors and clifford multiplication} implies that
    \begin{equation}
        \bB^{\pm} \gamma^a = \gamma^a \bB^{\pm}, \ \ \ \ \ \bB^{\pm} \gamma^n = \gamma^n \bB^{\mp}.
    \end{equation}
    So the spaces $\bV_E^{\pm}$ are preserved by $\gamma^a$ and interchanged by $\gamma^n$. One can also easily prove that
    \begin{equation}
        \bB^{\pm} \Pi = \Pi \, \bB^{\mp},
    \end{equation}
    so that the two eigenspaces $\bV_E^{\pm}$ are also interchanged by $\Pi$. 
\end{remark}

Armed with Lemma \ref{lemma: projectors and covariant derivative} and Remark \ref{remark: projectors and clifford relations}, we can now decompose the right-hand side of equation \eqref{relation: covariant normal derivative v1} into its $+$ and $-$ components to obtain
\begin{align}
    \bB^{\pm} \nabla^A_n \psi \big|_{\d M} &= \bB^{\pm} \left( \gamma^n \gamma^a \nabla_a^A \psi - m \gamma^n \psi \right) \big|_{\d M} \nonumber \\
    &=  \gamma^n \left( \gamma^a \bB^{\mp} \nabla_a^A \psi - m \bB^{\mp} \psi \right) \big|_{\d M} \nonumber \\
    &= \gamma^n \left( \gamma^a \nabla_a^A \left( \bB^{\mp} \psi \right) \mp \frac{1}{2} \gamma^a \Pi \left(\nabla_a \gamma^n \right) \psi - m\, \bB^{\mp} \psi \right) \bigg|_{\d M} \nonumber \\
    &= \gamma^n \left( \gamma^a \nabla_a^A \left( \bB^{\mp} \psi \right) \pm \frac{1}{2} \Pi \gamma^a  \left(\nabla_a \gamma^n \right) \psi - m\, \bB^{\mp} \psi \right) \bigg|_{\d M}. \label{relation: covariant normal derivative v2}
\end{align}
The right-hand side of equation \eqref{relation: covariant normal derivative v2} is now entirely in terms of the $+$ and $-$ components of $\psi$ on the boundary and their tangential derivatives. The final simplification we wish to make is to identify the second term in \eqref{relation: covariant normal derivative v2}. Thus, we have the following
\begin{lemma}\label{lemma: term is mean curvature}
    We have $\gamma^a \nabla_a \gamma^n = (n-1) H$, where $H$ is the mean curvature of the boundary.
\end{lemma}

\begin{proof}
    Using the definition of the Levi-Civita connection, we compute
\begin{align}
    \gamma^a \left( \nabla_a \gamma^n \right) &= \gamma^a \gamma\left( \nabla_{a} e_n \right) = \omega^b_{\ an} \gamma^a \gamma^b  = -\omega^n_{\ ab} \gamma^a \gamma^b. \label{computing mean curvature v1}
\end{align}
Note that $\omega^n_{\ ab} = g( e_n, \nabla_{e_a} e_b ) = \mathrm{II}(e_a, e_b)$, where $\mathrm{II}$ denotes the second fundamental form of the boundary, which is symmetric. Therefore, using the Clifford relations, we may write \eqref{computing mean curvature v1} as
\begin{align}
    \gamma^a \left( \nabla_a \gamma^n \right) &= -\frac{1}{2} \mathrm{II}_{ab} \left( \gamma^a \gamma^b + \gamma^b \gamma^a \right) = g^{ab} \mathrm{II}_{ab} = (n-1) H.
\end{align}
since $(n-1)H$ is by definition the trace of the second fundamental form.
\end{proof}

And so, using Lemma \ref{lemma: term is mean curvature}, we write \eqref{relation: covariant normal derivative v2} as
\begin{equation}\label{relation: covariant normal derivative v3}
    \bB^{\pm} \nabla^A_n \psi \big|_{\d M} =  \left( \gamma^n \gamma^a \nabla_a^A \left( \bB^{\mp} \psi \right) \pm \frac{n-1}{2} H \left( \gamma^n \Pi \right) \psi - m \gamma^n \bB^{\mp} \psi \right) \bigg|_{\d M}.
\end{equation}
Recall that $\bV_E^{\pm}$ are eigenspaces of $\gamma^n \Pi$. Thus, writing the $+$ and $-$ components of $\psi$ explicitly as $f$ and $\Theta_{g,A,m,\sU}(f)$, and dropping the subscripts on $\Theta_{g,A,m,\sU}$ for clarity, we may rewrite \eqref{relation: covariant normal derivative v3} as
\begin{align}
    \bB^{+} \nabla^A_n \psi \big|_{\sU} &=  \left( \gamma^n \gamma^a \nabla_a^A \left( \Theta(f) \right) + \frac{n-1}{2} H \left( f - \Theta(f) \right) - m \gamma^n \Theta(f) \right) \bigg|_{\sU}, \label{B+ of normal derivative} \\
    \bB^{-} \nabla^A_n \psi \big|_{\sU} &=  \left( \gamma^n \gamma^a \nabla_a^A f  - \frac{n-1}{2} H \left( f - \Theta(f) \right) - m \gamma^n f \right) \bigg|_{\sU}. \label{B- of normal derivative}
\end{align}
These are the $+$ and $-$ components of the right-hand side of equation\eqref{relation between Theta and Lambda, v1}. Therefore, decomposing $\Lambda_{g,A,m,\sU}$ into its $+$ and $-$ components, we may write equations \eqref{B+ of normal derivative} and \eqref{B- of normal derivative} as
\begin{equation}\label{relation between Theta and Lambda, matrix}
    \begin{pmatrix} \Lambda^{++} & \Lambda^{+-} \\ \Lambda^{-+} & \Lambda^{--} \end{pmatrix} \begin{pmatrix} f \\ \Theta(f) \end{pmatrix} = \begin{pmatrix}
        \frac{n-1}{2} H & \gamma^n \left( \gamma^{a} \nabla^A_{a} - m \right) - \frac{n-1}{2} H \\ \gamma^n \left( \gamma^{a} \nabla^A_{a} - m\right) - \frac{n-1}{2} H & \frac{n-1}{2} H
    \end{pmatrix} \begin{pmatrix} f \\ \Theta(f) \end{pmatrix},
\end{equation}
which gives us the desired relation between $\Lambda_{g,A,m,\sU}$ and $\Theta_{g,A,m,\sU}$. In fact, we shall principally be concerned with the equation for the $-$ component, equation \eqref{B- of normal derivative}, which can be rewritten in the notation of equation \eqref{relation between Theta and Lambda, matrix} as
\begin{equation}\label{relation between Theta and Lambda, - eqn}
    \Lambda^{-+} (f) + \Lambda^{--}\left( \Theta\left( f \right) \right) = \gamma^n \left( \gamma^a \nabla_{a}^A - m \right) f - \frac{n-1}{2} H f + \frac{n-1}{2} H \Theta(f).
\end{equation}
An immediate consequence of equation \eqref{relation between Theta and Lambda, - eqn} is the following

\begin{proposition}\label{prop: principal symbol}
    $\Theta_{g,A,m,\sU}$ is a pseudodifferential operator of order $0$ with principal symbol
    \begin{equation}\label{principal symbol of Theta}
        \theta_0(x,\x) = -i \gamma^n \frac{\gamma(\x)}{|\x|_g}.
    \end{equation}
\end{proposition}

\begin{proof}
    Recall from \cite{Valero2024} that $\Lambda_{g,A,m,\sU}$ is a pseudodifferential operator of order $1$ with principal symbol $-|\x|_g \id$, and so is therefore elliptic. In particular, $\Lambda^{--}_{g,A,m,\sU}$ is also elliptic, while $\Lambda_{g,A,m,\sU}^{-+}$ is of order $0$. We may therefore multiply equation \eqref{relation between Theta and Lambda, - eqn} by a left-parametrix for $\Lambda^{--}_{g,A,m,\sU}$ to deduce that $\Theta_{g,A,m,\sU}$ is a pseudodifferential operator of order $0$. Writing out the degree $1$ part of the symbol equation corresponding to equation \eqref{relation between Theta and Lambda, - eqn}, we get
    \begin{equation}
        -|\x|_g \theta_0(x,\x) \vp = i \gamma^n \gamma(\x),
    \end{equation}
    from which equation \eqref{principal symbol of Theta} follows.
\end{proof}

As discussed previously, we shall use equation \eqref{relation between Theta and Lambda, - eqn} to compute the full symbol of $\Theta_{g,A,m,\sU}$ by expressing it in terms of the symbol of $\Lambda_{g,A,m,\sU}$, whose symbol can be computed independently using the methods of \cite{Lee1989}. This we do in the following section.

\subsection{The symbol of $\Lambda_{g,A,m,\sU}$}\label{subsec: symbol of Lambda}

In this section, we compute the full symbol of the Dirichlet-to-Neumann map modulo appropriate terms, which we shall need later in Section \ref{subsec: symbol of Theta}. While the symbol of the Dirichlet-to-Neumann map has been computed in principle in \cite{Valero2024}, we shall need a much more precise form of the symbol, in which we keep track of lower-order terms that are ignored in \cite{Valero2024}. We therefore give the more detailed calculation here. 

In order to make the computations as streamlined as possible, we first introduce some notation to help us keep track of terms involving normal derivatives of the metric and connection at the boundary up to the desired order. We thus let $\cD_{k}$ denote any expression depending on at most $k$ normal derivatives of the metric and connection, and tangential derivatives thereof. We will sometimes also speak loosely of $\cD_k$ as being the set of such quantities. In a similar fashion, we let $\cD_{k,k-1}$ denote any expression depending on at most $k$ normal derivatives of the metric and at most $k-1$ normal derivatives of the connection. Finally, we let $\slashed{\cD}_k$ denote any quantity that is a linear combination of the gamma matrices $\gamma^a$ for $a \in \{1, \dots, n-1 \}$ with scalar coefficients in $\cD_{k}$. We define the symbol $\slashed{\cD}_{k,k-1}$ analogously. The most important fact about quantities in $\slashed{\cD}_k$ and $\slashed{\cD}_{k,k-1}$ that we shall use is that they are traceless endomorphisms on $\bV_E^{\pm}$ when $n > 2$. Note also that we have $\d_n \cD_{k} = \cD_{k+1}$, and similarly for the other symbols.


As in \cite{Valero2024}, near $\sU$ we have a factorization of the form
\begin{equation}\label{factorization into 2 first-order pseudos}
    D_A^2 - m^2 = \left( - \nabla^A_n + (n-1)H - \cB(x^{\a}, x^n, -i\d_{\a} ) \right) \left( \nabla^A_n - \cB(x^{\a}, x^n, -i\d_{\a}) \right) + \sS
\end{equation}
where $\cB(x^{\a}, x^n, -i\d_{\a})$ is a first-order pseudodifferential operator (or more precisely, a family of first-order pseudodifferential operators on $\sU$ depending on $x^n$), $H$ is the mean curvature of the boundary\footnote{Note that the function called $E$ in \cite{Valero2024} is equal to $(n-1)H$}, and $\sS$ is a smoothing operator. Then one has \cite{Valero2024}

\begin{lemma}
    The first-order pseudodifferential operator $\cB(x^{\a}, x^n, -i\d_{\a})$ in equation \eqref{factorization into 2 first-order pseudos} satisfies
    \begin{equation}
        \Lambda_{g,A,m,\sU} = \cB(x^{\a}, x^n, -i\d_{\a}) \big|_{x^n = 0} + \sR
    \end{equation}
    where $\sR$ is smoothing. In particular, $\Lambda_{g,A,m,\sU}$ and $\cB(x^{\a}, 0, -i\d_{\a})$ have the same symbols.
\end{lemma}

Thus, we may determine the symbol of $\Lambda_{g,A,m,\sU}$ by computing the symbol of $\cB$ from equation \eqref{factorization into 2 first-order pseudos} and restricting to the boundary. Indeed, equation \eqref{factorization into 2 first-order pseudos} reduces to
\begin{equation}\label{factorization equation reduced}
    [\d_n, \cB] + [\kappa_A(\d_n), \cB] - (n-1)H \cB + \cB^2 = Q_2 + Q_1 + Q_0,
\end{equation}
where $\kappa_A$ is the twisted connection $1$-form from \eqref{twisted connection 1 form}, and\footnote{Note that what we call $Q_0$ here is denoted $Q_0'$ in \cite{Valero2024}}
\begin{align}
    Q_2 &= -g^{\a \b} \d_{\a} \d_{\b}, \nonumber \\
    Q_1 &= -2 g^{\a \b} \kappa_A(\d_{\a}) \d_{\b} + g^{\a \b} \Gamma^{\gamma}_{\ \a \b} \d_{\gamma}, \nonumber \\
    Q_0 &= -g^{\a \b} \d_\a \left( \kappa_A(\d_\b) \right) - g^{\a \b} \kappa_A(\d_\a) \kappa_A(\d_\b) + g^{\a \b} \Gamma^{\gamma}_{\a \b} \kappa_A(\d_{\gamma})  + \frac{1}{4} R + \fF_A - m^2.
\end{align}
Now, let us compute the symbol of $\cB$ with respect to a spinor frame induced by an orthonormal frame $(e_a, e_n)$ that is parallel along $e_n$ near the boundary. In particular, since we are working in a local trivialization of $E$ where $A_n = 0$, it follows that $\kappa_A(\d_n) = 0$ near the boundary. With respect to such a frame, we may write equation \eqref{factorization equation reduced} as an equation for the symbol of $\cB$,
\begin{equation}\label{B symbol equation}
    \d_n b - (n-1) H b + \sum_{\n} \frac{1}{i^{|\n|} \n!} \d_{\x}^{\n} b \cdot \d_{x'}^{\n} b  = q_2 + q_1 + q_0
\end{equation}
where $b$ is the symbol of $\cB$, $q_i$ is the symbol of $Q_i$, and $\n$ in the sum runs over all multi-indices. As is shown in \cite{Valero2024}, the symbol $b$ has the form of an asymptotic series
\begin{equation}\label{B asymptotic series}
    b(x,\x) = \sum_{k=0}^{\infty} b_{1-k}(x,\x)
\end{equation}
where each $b_{1-k}$ is positive-homogeneous of degree $1-k$ in $\x$. By plugging the asymptotic series \eqref{B asymptotic series} into equation \eqref{B symbol equation} and equating terms of similar degrees, we can inductively solve for the components $b_{1-k}$. In particular, from the degree $2$ part of \eqref{B symbol equation}, we immediately have
\begin{equation}
    b_1(x,\x) = \sqrt{q_2} = -|\x|_g \id_{\bS_E}.
\end{equation}
With respect to the splitting $\bV_E^+ \oplus \bV_E^-$, we therefore have
\begin{equation}
    b_1^{\pm \pm}(x,\x) = -|\x|_g \id_{\bV_E^{\pm}}, \ \ \ \ \ \ \ b_1^{\pm \mp}(x,\x) = 0.
\end{equation}
In the future, we will often omit writing the identity operator explicitly.

Let us now compute the term $b_0$. To simplify the following computations, we let
\begin{equation}\label{unit vector}
    \hat{\x} := \frac{\x}{|\x|},
\end{equation}
and shall use the same symbol for symmetric tensor powers of $\hat{\x}$. That is, we let $\hat{\x}_{\a \b}$ denote $\hat{\x}_{\a} \hat{\x}_{\b}$, and so on. Note that $\hat{\x}_{\a_1 \cdots \a_k}$ is always positive homogeneous in $\x$ of degree $0$.

Now, proceeding with the computation of $b_0$, we obtain from the degree $1$ part of equation \eqref{B symbol equation} the following expression:
\begin{align}
    b_0 &= \frac{1}{2 |\x|} \left(\d_n b_1 -i \d_{\x} b_1 \cdot \d_{x'} b_1 - (n-1) H b_1 - q_1  \right) \nonumber \\
    &= i g^{\a \b} \kappa_A(\d_{\a}) \hat{\x}_{\b} - \frac{i}{2} g^{\a \b} \Gamma^{\gamma}_{\ \a \b} \hat{\x}_{\g} - \frac{i}{4} g^{\gamma \delta}  \d_{\g} g^{\a \b} \hat{\x}_{\a \b \delta} + \frac{n-1}{2} H  - \frac{1}{4} \d_n g^{\a \b} \hat{\x}_{\a\b},
\end{align}
from which we deduce that
\begin{equation}\label{b_0 pm pm}
    b^{\pm \pm}_0 =  i g^{\a \b} \kappa^{\pm \pm}_A(\d_\a) \hat{\x}_{\b} - \frac{i}{2} g^{\a \b} \Gamma^{\gamma}_{\ \a \b} \hat{\x}_{\g} - \frac{i}{4} g^{\gamma \delta}  \d_{\g} g^{\a \b} \hat{\x}_{\a \b \delta} + \frac{n-1}{2} H  - \frac{1}{4} \d_n g^{\a \b} \hat{\x}_{\a\b},
\end{equation}
\vspace{-5pt}
\begin{equation}\label{b_0 pm mp}
    b_0^{\pm \mp} = ig^{\a \b} \kappa_A^{\pm \mp}(\d_{\a}) \hat{\x}_{\b},
\end{equation}
where $\kappa_A(\d_{\a})$ splits into $+$ and $-$ components as
\begin{align}
    \kappa_A^{\pm \pm}(\d_{\a}) &= -\frac{1}{2} \sum_{b<c} \omega^b_{\ c}(\d_{\a}) \gamma^b \gamma^c + A_{\a}, \label{form of kappa --} \\
    \kappa_A^{\pm \mp}(\d_{\a}) &= -\frac{1}{2} \sum_{b} \omega^b_{\ n}(\d_{\a}) \gamma^b \gamma^n. \label{form of kappa -+}
\end{align}
We also record the even and odd parts of the symbol under $\x \mapsto -\x$,
\begin{equation}\label{b_0 -- even}
    \left( b_0^{\pm \pm} \right)^{\mathfrak{e}} = \frac{n-1}{2} H - \frac{1}{4} \d_n g^{\a \b} \hat{\x}_{\a\b},
\end{equation}
\begin{equation}\label{b_0 -- odd}
    \left( b_0^{\pm \pm} \right)^{\mathfrak{o}} =  i g^{\a \b} \kappa^{\pm \pm}_A(\d_\a) \hat{\x}_{\b} - \frac{i}{2} g^{\a \b} \Gamma^{\gamma}_{\ \a \b} \hat{\x}_{\g} - \frac{i}{4} g^{\gamma \delta}  \d_{\g} g^{\a \b} \hat{\x}_{\a \b \delta},
\end{equation}
\begin{equation}\label{b_0 -+ odd}
    \left( b_0^{\pm \mp} \right)^{\mathfrak{o}} = ig^{\a \b} \kappa_A^{\pm \mp} \hat{\x}_{\b},
\end{equation}
\begin{equation}\label{b_0 -+ even}
    \left( b_0^{\pm \mp} \right)^{\mathfrak{e}} = 0.
\end{equation}
Note that the even part of $b_0^{\pm \pm}$ is a multiple of the identity. Note also that the only parts of $b_0$ that depend on normal derivatives of the metric are the even part of $b_0^{\pm \pm}$ and the odd part of $b_0^{\pm \mp}$.

Let us now compute $b_{-1}$ modulo $\cD_0$. The degree $0$ part of equation \eqref{B symbol equation} gives us
\begin{align}
    b_{-1} &= -\frac{1}{2|\x|}\left( q_0 -\d_n b_0 + (n-1) H b_0 - b_0^2 + i \d_{\x} b_1 \cdot \d_{x'} b_0 + i \d_{\x} b_0 \cdot \d_{x'} b_1 \right),
\end{align}
where
\begin{equation}\label{symbol q_0}
    q_0 = -g^{\a \b} \d_{\a} \left( \kappa_A(\d_{\b}) \right) - g^{\a \b} \kappa_A(\d_{\a}) \kappa_A(\d_{\b}) + g^{\a \b} \Gamma^{\gamma}_{\ \a \b} \kappa_A(\d_{\gamma}) + \frac{1}{4} R + \fF_A - m^2.
\end{equation}
We want to carefully consider the structure of the symbol $b_1$, as it this structure will be repeated in all of the other components $b_{1-k}$ for $k > 2$, which will be of crucial importance in Section \ref{subsec: symbol of Theta}. For the moment, we will focus on two particular parts of the symbol, namely the odd part of $b_{-1}^{--}$ and the even part of $b_{-1}^{-+}$ (the analysis of $b_{-1}^{++}$ and $b_{-1}^{+-}$ is analogous). We have
\begin{align}\label{b_-1 --}
    b_{-1}^{--} &= -\frac{1}{2|\x|}\bigg( q_0^{--} -\d_n b^{--}_0 + (n-1) H b^{--}_0 - \left( b_0^{--} \right)^2 - b_0^{-+} b_0^{+-} + i \d_{\x} b_1 \cdot \d_{x'}  b^{--}_0 \nonumber \\
    & \hspace{325pt} + i \d_{\x}  b^{--}_0 \cdot \d_{x'} b_1 \bigg),
\end{align}
while on the other hand,
\begin{align}\label{b_-1 -+}
    b_{-1}^{-+} &= -\frac{1}{2|\x|}\bigg( q_0^{-+} -\d_n b^{-+}_0 + (n-1) H b^{-+}_0 -  b_0^{--} b_0^{-+} - b_0^{-+} b_0^{++} + i \d_{\x} b_1 \cdot \d_{x'}  b^{-+}_0 \nonumber \\
    & \hspace{325pt} + i \d_{\x}  b^{-+}_0 \cdot \d_{x'} b_1 \bigg).
\end{align}
Note that from \eqref{symbol q_0} we have that 
\begin{align}
    q^{-+}_0 &= -g^{\a \b} \d_{\a} \left( \kappa^{-+}_A(\d_{\b}) \right) - g^{\a \b} \kappa^{--}_A(\d_{\a}) \kappa^{-+}_A(\d_{\b}) - g^{\a \b} \kappa^{-+}_A(\d_{\a}) \kappa^{++}_A(\d_{\b}) \nonumber \\
    & \hspace{250pt} + g^{\a \b} \Gamma^{\gamma}_{\ \a \b} \kappa^{-+}_A(\d_{\gamma}) +  \fF^{-+}_A. \label{q_0^-+}
\end{align}
For later use, we note that \eqref{form of kappa -+} implies $q_0^{-+}$ has the form
\begin{align}
    q^{-+}_0 &=  - g^{\a \b} \kappa^{--}_A(\d_{\a}) \kappa^{-+}_A(\d_{\b}) - g^{\a \b} \kappa^{-+}_A(\d_{\a}) \kappa^{++}_A(\d_{\b}) + \gamma^n \slashed{\cD}_{1,0} + \fF^{-+}_A. \label{q_0^-+ form}
\end{align}
Thus, we deduce from \eqref{b_-1 --} that the odd part of $b_{-1}^{--}$ is
\begin{align}\label{b_-1 -- odd}
    \left( b_{-1}^{--} \right)^{\mathfrak{o}} &= -\frac{1}{2|\x|}\bigg( -\d_n \left( b^{--}_0 \right)^{\mathfrak{o}} + (n-1) H \left( b^{--}_0 \right)^{\mathfrak{o}} - \left( b_0^{--} \right)^{\mathfrak{o}} \left( b_0^{--} \right)^{\mathfrak{e}} - \left( b_0^{--} \right)^{\mathfrak{e}} \left( b_0^{--} \right)^{\mathfrak{o}} \nonumber \\
    &\hspace{200pt} + i \d_{\x} b_1 \cdot \d_{x'}  \left( b^{--}_0 \right)^{\mathfrak{e}} + i \d_{\x}  \left( b^{--}_0 \right)^{\mathfrak{e}} \cdot \d_{x'} b_1 \bigg), \nonumber \\
    &= -\frac{1}{2|\x|}\bigg( -\d_n \left( b^{--}_0 \right)^{\mathfrak{o}} + (n-1) H \left( b^{--}_0 \right)^{\mathfrak{o}} - 2 \left( b_0^{--} \right)^{\mathfrak{o}} \left( b_0^{--} \right)^{\mathfrak{e}} + \cD_{1,0} \cdot \id \bigg),
\end{align}
where we have used the fact that $(b_0^{--})^{\mathfrak{e}}$ has the form $\cD_{1,0} \cdot \id$. Moreover, the even part of $b_{-1}^{-+}$ is
\begin{align}\label{b_-1 -+ even}
    \left( b_{-1}^{-+} \right)^{\mathfrak{e}} &= -\frac{1}{2|\x|}\bigg( q_0^{-+} -  \left( b_0^{--} \right)^{\mathfrak{o}} b_0^{-+} - b_0^{-+} \left( b_0^{++} \right)^{\mathfrak{o}} + i \d_{\x} b_1 \cdot \d_{x'}  b^{-+}_0 + i \d_{\x}  b^{-+}_0 \cdot \d_{x'} b_1 \bigg) \nonumber \\
    &= -\frac{1}{2|\x|}\bigg( q_0^{-+} -  \left( b_0^{--} \right)^{\mathfrak{o}} b_0^{-+} - b_0^{-+} \left( b_0^{++} \right)^{\mathfrak{o}} + \gamma^n \slashed{\cD}_{1,0} \bigg).
\end{align}
Moving forward, we want to impose an assumption which allows us to greatly simplify the computations in Section \ref{subsec: symbol of Theta}. Let $x_0$ be an arbitrary point in $\sU$. We shall henceforward assume that we are computing the symbol of $\Lambda_{g,A,m,\sU}$ with respect to a local trivialization of $\bS$ induced by an orthonormal frame $(e_a)$ on the boundary such that $\nabla^{\d M} e_a$ vanishes at $x_0$, where $\nabla^{\d M}$ denotes the intrinsic Levi-Civita connection of the metric induced on the boundary. Such a frame always exists, and if the metric is known on $\sU$, such a frame can always be constructed for an arbitrary point $x_0$. In such a frame, $\omega^a_{\ bc}$ vanishes at $x_0$. Moreover, if the connection on $\sU$ is known, we may also construct a frame of $E$ along the boundary so that the connection one-form $A$ also vanishes at $x_0$. The implication of working in such a frame of $\bS_E$ is that $\kappa_A^{\pm \pm}$ vanishes when evaluated at $x_0$. Thus, we see from equations \eqref{b_0 -- even}--\eqref{b_0 -+ even} that $b_0^{--}$ is a multiple of the identity at $x_0$. This will greatly simplify the computations in the following section, where we will allow $x_0$ to vary. 

With this in mind, we see from \eqref{b_-1 -- odd} and \eqref{b_-1 -+ even} that
\begin{align}\label{b_-1 -- odd at x_0}
    \left( b_{-1}^{--} \right)^{\mathfrak{o}} \big|_{x_0} &= -\frac{1}{2|\x|}\bigg( -\d_n \left( b^{--}_0 \right)^{\mathfrak{o}} + \cD_{1,0} \cdot \id + \, i \d_{\x} b_1 \cdot \d_{x'}  \left( b^{--}_0 \right)^{\mathfrak{e}} + i \d_{\x}  \left( b^{--}_0 \right)^{\mathfrak{o}} \cdot \d_{x'} b_1 \bigg) \bigg|_{x_0}.
\end{align}
Note that $\left( b_0^{--} \right)^{\mathfrak{e}}$ is always a multiple of the identity by \eqref{b_0 -- even}, and that $\d_{\x} \left( b_0^{--} \right)^{\mathfrak{o}}$ is still a multiple of the identity when evaluated at $x_0$, by \eqref{b_0 -- odd}. Therefore, equation \eqref{b_-1 -- odd at x_0} reduces to
\begin{align}\label{b_-1 -- odd at x_0 v1.5}
    \left( b_{-1}^{--} \right)^{\mathfrak{o}} \big|_{x_0} &= -\frac{1}{2|\x|}\bigg( -\d_n \left( b^{--}_0 \right)^{\mathfrak{o}} + \cD_{1,0} \cdot \id \bigg) \bigg|_{x_0},
\end{align}
which upon using equation \eqref{b_0 -- odd}, becomes
\begin{align}\label{b_-1 -- odd at x_0 v2}
    \left( b_{-1}^{--} \right)^{\mathfrak{o}} \big|_{x_0} &= \frac{1}{2|\x|}\bigg( \d_n \left( i g^{\a \b} \kappa^{--}_A(\d_\a) \right) \hat{\x}_{\b} + \cD_{1,0} \cdot \id \bigg) \bigg|_{x_0}, \nonumber \\
    &= \frac{1}{2|\x|} \bigg( i g^{\a \b} \d_n \left( \kappa^{--}_A(\d_\a) \right) \hat{\x}_{\b} + \cD_{1,0} \cdot \id \bigg) \bigg|_{x_0},
\end{align}
where in the last line we have used the fact that $\kappa^{\pm \pm}_A$ vanishes at $x_0$.
Turning now to the even part of $b_{-1}^{-+}$, we use equation \eqref{b_-1 -+ even} and the fact that $b_0^{-+}$ is $\gamma^n \slashed{\cD}_{1,0}$ to obtain
\begin{align}\label{b_-1 -+ even at x_0}
    \left( b_{-1}^{-+} \right)^{\mathfrak{e}} \big|_{x_0} &= -\frac{1}{2|\x|}\bigg( q_0^{-+} + \gamma^n \slashed{\cD}_{1,0} \bigg) \bigg|_{x_0}.
\end{align}
Using the vanishing of $\kappa_A^{\pm \pm}$ at $x_0$, we see from \eqref{q_0^-+ form} that
\begin{equation}\label{q_0^-+ form at x_0}
    q_0^{-+} \big|_{x_0} = \fF_A^{-+} + \gamma^n \slashed{\cD}_{1,0}. 
\end{equation}
Therefore, it follows from \eqref{b_-1 -+ even at x_0} that
\begin{align}\label{b_-1 -+ even at x_0 v2}
    \left( b_{-1}^{-+} \right)^{\mathfrak{e}} \big|_{x_0} &= -\frac{1}{2|\x|} \left( \fF_A^{-+} + \gamma^n \slashed{\cD}_{1,0} \right) \Big|_{x_0}.
\end{align}

To finish our analysis of $b_{-1}$, we want to consider the highest order terms in the even part of $b_{-1}^{\pm \pm}$ and the odd part of $b_{-1}^{\pm \mp}$, which are the only components of $b_{-1}$ containing terms in $\cD_2$. We shall thus keep track of terms of second order in the metric modulo terms in $\cD_1$. 

As the scalar curvature appears in $q_0$, let us us first note that
\begin{align}
    \frac{1}{4} R &= \frac{1}{4} \left( g^{\a \b} \d_n \Gamma^n_{\ \a \b} - \d_n \Gamma^{\a}_{\ n\a} \right) + \cD_{1,0} \nonumber \\
    &= \frac{1}{4} \d_n \left( g^{\a \b} \Gamma^n_{\ \a \b} \right) - \frac{1}{4} \d_n \Gamma^{\a}_{\ n\a} + \cD_{1,0} \nonumber \\
    &= \frac{n-1}{2} \d_n H + \cD_{1,0}. \label{scalar curvature}
\end{align}
It is thus clear from equations \eqref{symbol q_0}, \eqref{b_-1 --}, \eqref{b_0 -- even}, and \eqref{scalar curvature} that
\begin{align}
    \left( b_{-1}^{--} \right)^{\mathfrak{e}} &= -\frac{1}{2|\x|} \left( q_0^{--} - \d_n \left( b_0^{--} \right)^{\mathfrak{e}}\, \right) + \cD_{1,0} \nonumber \\
    &= -\frac{1}{8|\x|} \d^2_n g^{\a \b} \hat{\x}_{\a\b} + \cD_{1,0}. \label{b_-1 -- highest order}
\end{align}
For the odd part of $b_{-1}^{-+}$, we have from equations \eqref{b_0 -+ odd} and \eqref{b_-1 -+} that
\begin{align}
    \left( b_{-1}^{-+} \right)^{\mathfrak{o}} &= \frac{1}{2|\x|} \d_n b_0^{-+}  + \cD_{1,0} \nonumber \\
    &= \frac{1}{2|\x|}  i g^{\a \b} \d_n \left( \kappa_A^{-+}(\d_{\a}) \right) \hat{\x}_{\b}  + \cD_{1,0}. \label{b_-1 -+ highest order}
\end{align}
It is noteworthy that here, we have not made use of our particular choice of frame.

We can now inductively study the structure of $b_{-(k+1)}$ for $k > 0$, modulo terms in $\cD_k$. We thus keep track of terms in $\cD_{k+2}$ and $\cD_{k+1}$, and we want to show that the symbol $b_{-(k+1)}$ is analogous in structure to $b_{-1}$ with respect to these terms. To this end, let us consider the term $b_{-(k+1)}$. The degree $-k$ part of equation \eqref{B symbol equation} gives us
\begin{align}
    b_{-(k+1)} &= -\frac{1}{2|\x|} \left( -\d_n b_{-k} + (n-1) H b_{-k} - \sum_{\substack{i + j - |\nu| = -k \\ -k \leq i,j \leq 1}} \frac{1}{i^{|\nu|}\nu!} \d^{\nu}_{\x} b_i \d_{x'}^{\nu} b_j \right) \nonumber \\
    &= \frac{1}{2|\x|} \Big( \d_n b_{-k} - (n-1) H b_{-k} + b_0 b_{-k} + b_{-k} b_0 - i \d_{\x} b_{-k} \d_{x'} b_1 - i\d_{\x} b_1 \d_{x'} b_{-k} \Big) + \cD_{k}. \label{b -(k+1) equation}
\end{align}
One immediate consequence of the above equation is that the highest normal derivatives in $b_{-(k+1)}$ enter through the $\d_n b_{-k}$ term. In particular, given the form of $b_{-1}$ in equations \eqref{b_-1 -- highest order} and \eqref{b_-1 -+ highest order}, as well as equation \eqref{b -(k+1) equation}, it is easy to see by induction that we have the following
\begin{proposition}\label{prop: form of b_-(k+1) to highest order}
    For $k \geq 0$, the symbol $b_{-(k+1)}$ has the following form:
    \begin{align}
    \left( b^{--}_{-(k+1)} \right)^{\mathfrak{o}} &= \cD_{k+1}, \label{b-(k+1) -- odd is D_k+1} \\
        \left( b^{-+}_{-(k+1)} \right)^{\mathfrak{e}} &= \cD_{k+1} \label{b-(k+1) -+ even is D_k+1} \\
        \left( b^{--}_{-(k+1)} \right)^{\mathfrak{e}}  &= -\frac{1}{ 2^{k+3} |\x|^{k+1} } \d_n^{k+2} g^{\a \b} \hat{\x}_{\a \b} + \cD_{k+1,k}, \label{b-(k+1) -- even highest term} \\
        \left( b^{-+}_{-(k+1)} \right)^{\mathfrak{o}} &= \frac{i}{2^{k+1}|\x|^{k+1}}  g^{\a \b} \d_n^{k+1} \left( \kappa_A^{-+}(\d_{\a}) \right) \hat{\x}_{\b} + \cD_{k+1,k}. \label{b-(k+1) -+ odd highest term}
\end{align}
Analogous results hold with $+$ and $-$ interchanged.
\end{proposition}
In particular, we see that the highest order part of $b^{--}_{-(k+1)}$ is a multiple of the identity, while the highest order part of $b_{-(k+1)}^{-+}$ has the form $\gamma^n \slashed{\cD}_{k+2}$. However, in analogy with our analysis of $b_{-1}$, we want to keep careful track of the highest order terms in \eqref{b-(k+1) -- odd is D_k+1}, \eqref{b-(k+1) -+ even is D_k+1} modulo $\cD_{k}$. In particular, we want to prove the following more precise result:

\begin{proposition}\label{prop: precise form of b -(k+1)}
    The term $b_{-(k+1)}$ has the following form for $k \geq 0$:
    \begin{align}
        \left( b^{--}_{-(k+1)} \right)^{\mathfrak{o}} &= \frac{1}{2|\x|} \left( \frac{i}{2^k|\x|^k} g^{\a \b} \d_n^{k+1} \left( \kappa_A^{--}(\d_{\a}) \right) \hat{\x}_{\b} + \cD_{k+1,k} \cdot \left( b_0^{--} \right)^{\mathfrak{o}} + \cD_{k+1,k} \cdot \id \right) + \cD_{k}, \label{prop: precise form of b -(k+1) -- eqn} \\
        \left( b^{-+}_{-(k+1)} \right)^{\mathfrak{e}} &= \frac{1}{2|\x|} \left( -\frac{1}{2^k|\x|^k} \d_n^k q_0^{-+} + \left( b_0^{--} \right)^{\mathfrak{o}} \gamma^n \slashed{\cD}_{k+1, k} + \gamma^n \slashed{\cD}_{k+1, k} \left( b_0^{++} \right)^{\mathfrak{o}} + \gamma^n \slashed{\cD}_{k+1, k}  \right) + \cD_{k}. \label{prop: precise form of b -(k+1) -+  eqn}
    \end{align}
    Moreover, we have
    \begin{equation}\label{normal derivatives of q_0}
        \d_n^{k} q_0^{-+} = \d_n^k \fF_A^{-+} - g^{\a \b} \kappa^{--}_A(\d_{\a}) \left( \d_n^k  \kappa^{-+}_A(\d_{\b}) \right) - g^{\a \b} \left( \d^k_n \kappa^{-+}_A(\d_{\a}) \right) \kappa^{++}_A(\d_{\b}) + \gamma^n \slashed{\cD}_{k+1, k} + \cD_{k}.
    \end{equation}
    Analogous results hold with $+$ and $-$ interchanged.
\end{proposition}

\begin{proof}
    The proof is by induction. Note that for $k = 0$, the lemma follows from the expressions for $\left( b^{--}_{-1} \right)^{\mathfrak{o}}$ and $\left( b^{-+}_{-1} \right)^{\mathfrak{e}}$ given in equations \eqref{b_-1 -- odd}, \eqref{b_-1 -+ even}. Suppose now that Proposition \ref{prop: precise form of b -(k+1)} holds for $k-1$. Breaking equation \eqref{b -(k+1) equation} into components and using Proposition \ref{prop: form of b_-(k+1) to highest order}, we obtain
    \begin{align}
    \left( b^{--}_{-(k+1)} \right)^{\mathfrak{o}} &= \frac{1}{2|\x|} \Big( \d_n \left( b^{--}_{-k} \right)^{\mathfrak{o}} - (n-1) H \left( b^{--}_{-k} \right)^{\mathfrak{o}} + \left( b^{--}_0\right)^{\mathfrak{o}} \left( b^{--}_{-k} \right)^{\mathfrak{e}} + \left( b^{--}_0\right)^{\mathfrak{e}} \left( b^{--}_{-k} \right)^{\mathfrak{o}} + b^{-+}_0 \left( b^{+-}_{-k} \right)^{\mathfrak{e}} \nonumber \\
    & + \left( b^{--}_{-k} \right)^{\mathfrak{e}} \left( b^{--}_0 \right)^{\mathfrak{o}} + \left( b^{--}_{-k} \right)^{\mathfrak{o}} \left( b^{--}_0 \right)^{\mathfrak{e}} - \left( b^{-+}_{-k} \right)^{\mathfrak{e}} b^{+-}_0 - i \d_{\x} \left( b^{--}_{-k}\right)^{\mathfrak{e}} \d_{x'} b_1 \nonumber \\
    &\hspace{260pt} - i\d_{\x} b_1 \d_{x'} \left( b^{--}_{-k} \right)^{\mathfrak{e}} \Big) + \cD_{k}.
\end{align}
Now, using Proposition \ref{prop: form of b_-(k+1) to highest order}, as well as the fact that $b_1$ is a multiple of the identity, we see that
\begin{align}\label{-- inductive step}
    \left( b^{--}_{-(k+1)} \right)^{\mathfrak{o}} &= \frac{1}{2|\x|} \Big( \d_n \left( b^{--}_{-k} \right)^{\mathfrak{o}} + \cD_{k+1,k-1} \cdot \left( b^{--}_0\right)^{\mathfrak{o}} + \cD_{k+1,k-1} \cdot \id \Big) + \cD_{k}.
\end{align}
By assumption, the odd part of $b^{--}_{-k}$ has the form
\begin{equation}\label{inductive hypo for --}
    \left( b^{--}_{-k} \right)^{\mathfrak{o}} = \frac{1}{2|\x|} \left( \frac{i}{2^{k-1}|\x|^{k-1}} g^{\a \b} \d_n^{k} \left( \kappa_A^{--}(\d_{\a}) \right) \hat{\x}_{\b} + \cD_{k,k-1} \cdot \left( b_0^{--} \right)^{\mathfrak{o}} + \cD_{k,k-1} \cdot \id \right) + \cD_{k-1}.
\end{equation}
Thus, differentiating \eqref{inductive hypo for --} and plugging into \eqref{-- inductive step}, we obtain \eqref{prop: precise form of b -(k+1) -- eqn}. Turning now to the even part of $b^{-+}_{-(k+1)}$, we again use equation \eqref{b -(k+1) equation} and Proposition \ref{prop: form of b_-(k+1) to highest order} to obtain
 \begin{align}
    \left( b^{-+}_{-(k+1)} \right)^{\mathfrak{e}} &= \frac{1}{2|\x|} \Big( \d_n \left( b^{-+}_{-k} \right)^{\mathfrak{e}} - (n-1) H \left( b^{-+}_{-k} \right)^{\mathfrak{e}} + b_0^{-+}\left( b^{++}_{-k} \right)^{\mathfrak{o}} + \left( b_0^{--} \right)^{\mathfrak{o}} \left( b_{-k}^{-+} \right)^{\mathfrak{o}} + \left( b_{-k}^{--} \right)^{\mathfrak{o}} b_0^{-+}  \nonumber \\
    &\hspace{95pt} + \left( b_{-k}^{-+} \right)^{\mathfrak{o}} \left( b_0^{++} \right)^{\mathfrak{o}} - i\d_{\x} \left( b^{-+}_{-k} \right)^{\mathfrak{o}} \d_{x'} b_1 - i\d_{\x} b_1 \d_{x'} \left( b^{-+}_{-k} \right)^{\mathfrak{o}} \Big) + \cD_{k}.
\end{align}
Using Proposition \ref{prop: form of b_-(k+1) to highest order} once more, the above reduces to
 \begin{align}\label{-+ inductive step}
    \left( b^{-+}_{-(k+1)} \right)^{\mathfrak{e}} &= \frac{1}{2|\x|} \Big( \d_n \left( b^{-+}_{-k} \right)^{\mathfrak{e}} + \left( b_0^{--} \right)^{\mathfrak{o}} \left( b_{-k}^{-+} \right)^{\mathfrak{o}}  + \left( b_{-k}^{-+} \right)^{\mathfrak{o}} \left( b_0^{++} \right)^{\mathfrak{o}} + \gamma^n \slashed{\cD}_{k+1,k} \Big) + \cD_{k}.
\end{align}
By assumption, the even part of $b_{-k}^{-+}$ has the form
\begin{equation}\label{inductive hypo -+}
    \left( b^{-+}_{-k} \right)^{\mathfrak{e}} = \frac{1}{2|\x|} \left( - \frac{1}{2^{k-1}|\x|^{k-1}} \d_n^{k-1} q_0^{-+} + \left( b_0^{--} \right)^{\mathfrak{o}} \gamma^n \slashed{\cD}_{k, k-1} + \gamma^n \slashed{\cD}_{k, k-1} \left( b_0^{++} \right)^{\mathfrak{o}} + \gamma^n \slashed{\cD}_{k, k-1}  \right) + \cD_{k-1}.
\end{equation}
Thus, differentiating \eqref{inductive hypo -+} and plugging into \eqref{-+ inductive step}, we obtain \eqref{prop: precise form of b -(k+1) -+  eqn}. Equation \eqref{normal derivatives of q_0} follows easily from \eqref{q_0^-+ form}. This completes the proof of Proposition \ref{prop: precise form of b -(k+1)}.
\end{proof}

Finally, to end our analysis of the symbol of $\Lambda_{g,A,m,\sU}$, we consider the form that the symbol $b_{-(k+1)}$ has when evaluated at $x_0$ with respect to our chosen frames. Recall that when evaluated at $x_0$, we have that $\kappa_A^{\pm \pm}$ vanishes and $b^{\pm \pm}_0$ is a multiple of the identity. Therefore, the following result follows easily from Proposition \ref{prop: precise form of b -(k+1)}.

\begin{proposition}\label{prop: precise form at x_0}
    For $k \geq 0$, the symbol $b_{-(k+1)}$ has the following form when evaluated at $x_0$:
    \begin{align}
        \left( b^{--}_{-(k+1)} \right)^{\mathfrak{o}} \Big|_{x_0} &= \frac{1}{2|\x|} \left( \frac{i}{2^k|\x|^k} g^{\a \b} \d_n^{k+1} \left( \kappa_A^{--}(\d_{\a}) \right) \hat{\x}_{\b} + \cD_{k+1,k} \cdot \id \right) \bigg|_{x_0} + \cD_{k}, \label{prop: precise form at x_0 -- eqn} \\
        \left( b^{-+}_{-(k+1)} \right)^{\mathfrak{e}} \Big|_{x_0} &= \frac{1}{2|\x|} \left( - \frac{1}{2^k|\x|^k} \d_n^{k} \fF_A^{-+} + \gamma^n \slashed{\cD}_{k+1, k}  \right) \bigg|_{x_0} + \cD_{k}. \label{prop: precise form at x_0 -+ eqn}
    \end{align}
    Moreover, it is clear from \eqref{prop: precise form at x_0 -+ eqn} that the even part of $b^{-+}_{-(k+1)}$ at $x_0$ is of the form $\gamma^n \slashed{\cD}_{k+1}$.
\end{proposition}

Propositions \ref{prop: form of b_-(k+1) to highest order} and \ref{prop: precise form at x_0} concerning the highest order terms in the symbol of $\Lambda_{g,A,m,\sU}$ shall play an important role in our analysis of the symbol of $\Theta_{g,A,m,\sU}$ in the following section.

\subsection{The symbol of $\Theta_{g,A,m,\sU}$ and proof of Theorem \ref{main theorem: boundary determination, intro} for $n \geq 3$}\label{subsec: symbol of Theta}

In this section we shall prove our boundary determination result by first using the relation between $\Theta_{g,A,m,\sU}$ and the Dirichlet-to-Neumann map to compute the symbol of $\Theta_{g,A,m, \sU}$ in terms of the symbol of $\Lambda_{g,A,m, \sU}$, and then using the precise form of the symbol of $\Lambda_{g,A,m, \sU}$ derived in Section \ref{subsec: symbol of Lambda} to show that one can invert the expressions to obtain the normal derivatives of the metric and connection at the boundary, modulo the appropriate symmetries.

Let us recall equation \eqref{relation between Theta and Lambda, - eqn} relating $\Lambda_{g,A,m,\sU}$ and $\Theta_{g,A,m,\sU}$,
\begin{equation}\label{relation between Theta and Lambda, - eqn second time}
    \Lambda^{-+}  + \Lambda^{--} \Theta   = \gamma^n \left( \gamma^a \nabla_{a}^A - m \right) - \frac{n-1}{2} H  + \frac{n-1}{2} H \Theta.
\end{equation}
Writing the symbols of $\Lambda$ and $\Theta$ as $b(x,\x)$ and $\theta(x,\x)$ respectively, equation \eqref{relation between Theta and Lambda, - eqn second time} becomes
\begin{equation}\label{relation: full symbol eqn}
    b^{-+} + b^{--} \# \theta = \gamma^n \left( \gamma^a \nabla^A_a - m \right) - \frac{n-1}{2} H + \frac{n-1}{2} H \theta,
\end{equation}
where $\#$ denotes the composition of symbols. Let us consider a local expression for the first term on the right hand side of \eqref{relation: full symbol eqn}. Recall that we let $\kappa_A$ denote the connection $1$-form corresponding to the sum of the spin connection and $A$, so that
\begin{equation}
    \kappa_A(X) := -\frac{1}{2} \sum_{i < j} \omega^i_{\ j}(X) \gamma^i \gamma^j + A(X).
\end{equation}
Moreover, we let $h$ denote the matrix relating our chosen orthonormal frame $(e_a)$ over $\sU$ to the coordinate frame $(\d_{\a})$ over $\sU$, so that we have
\begin{equation}
    e_a = h^{\a}_{\ a} \d_{\a}.
\end{equation}
Thus, in a local trivialization, we compute using the Clifford relations,
\begin{align}
    \gamma^n \gamma^a \nabla^A_a \vp &= \gamma^n \gamma^a e_a(\vp) + \gamma^n \gamma^a \kappa_A(e_a) \vp \nonumber \\
    &= \gamma^n \gamma^a h^{\a}_{\ a} \d_{\a} \vp - \frac{1}{2} \gamma^n \gamma^a \sum_{i < j} \omega^i_{\ j}(e_a) \gamma^i \gamma^j \vp + \gamma^n \gamma^a A(e_a) \nonumber \\
    &= \gamma^n \gamma^a h^{\a}_{\ a} \d_{\a} \vp - \frac{1}{2}  \sum_{b} \omega^b_{\ a n} \gamma^n \gamma^a \gamma^b \gamma^n \vp - \frac{1}{2}  \sum_{b < c} \omega^b_{\ a c} \gamma^n \gamma^a \gamma^b \gamma^c \vp + \gamma^n \gamma(A) \vp \nonumber \\
    &= \gamma^n \gamma^a h^{\a}_{\ a} \d_{\a} \vp + \frac{1}{2}  \sum_{b} \omega^b_{\ a n} \gamma^a \gamma^b \vp - \frac{1}{2}  \sum_{b < c} \omega^b_{\ a c} \gamma^n \gamma^a \gamma^b \gamma^c \vp + \gamma^n \gamma(A) \vp \nonumber \\
    &= \gamma^n \left( \gamma^a h^{\a}_{\ a} \d_{\a} + \Omega   + \gamma(A) \right) \vp  - \frac{1}{2} \omega^n_{\ a b} \gamma^a \gamma^b \vp, \label{local expression of dirac op on boundary}
\end{align}
where we have defined the endomorphism
\begin{equation}\label{definition of Omega}
    \Omega := - \frac{1}{2}  \sum_{b < c} \omega^b_{\ a c}  \gamma^a \gamma^b \gamma^c,
\end{equation}
and let $\gamma(A)$ denote $\gamma^a A(e_a)$. Note that by the computations in Lemma \ref{lemma: term is mean curvature}, the final term in equation \eqref{local expression of dirac op on boundary} can be rewritten in terms of the mean curvature, and thus we have
\begin{equation}
    \gamma^n \left( \gamma^a \nabla^A_a - m \right) = \gamma^n \left( \gamma^a h^{\a}_{\ a} \d_{\a} + \Omega + \gamma(A) - m \right) + \frac{n-1}{2} H. 
\end{equation}
Therefore, equation \eqref{relation: full symbol eqn} reduces to
\begin{equation}\label{local full symbol relation}
    b^{-+} + b^{--} \#\, \theta = \gamma^n \left( \gamma^a h^{\a}_{\ a} \d_{\a} + \Omega + \gamma(A) - m \right) + \frac{n-1}{2} H \theta.
\end{equation}
Finally, we substitute the asymptotic series
\begin{equation}
    b(x,\x) = \sum_{k=0}^{\infty} b_{1-k}(x,\x), \ \ \ \ \ \ \theta(x,\x) = \sum_{k=0}^{\infty} \theta_{-k}(x,\x),
\end{equation}
where $b_{-k}(x,\x)$ and $\theta_{-k}(x,\x)$ are positive-homogeneous of degree $-k$ in $\x$, and collect terms of similar degrees. The degree $1$ part of equation \eqref{local full symbol relation} has already been considered in Proposition \ref{prop: principal symbol}, and as observed there, it gives us the principal symbol of $\Theta_{g,A,m,\sU}$,
\begin{equation}
    \theta_0(x,\x) = -i \gamma^n \frac{\gamma(\x)}{|\x|}.
\end{equation}
An immediate consequence of the Clifford relations is that the conformal class of $g|_{\sU}$ is determined by the principal symbol $\theta_0$. Indeed, we have
\begin{align}
     \theta_0(x,\x) \theta_0(x,\eta)  +  \theta_0(x,\eta) \theta_0(x,\x) = -\frac{\gamma(\x) \gamma(\eta) + \gamma(\eta) \gamma(\x)}{|\x||\eta|} = \frac{2g(\x,\eta)}{|\x| |\eta|}.  \label{determining angles}
\end{align}
The right-hand side of \eqref{determining angles} is equal to the cosine of the angle between the rays determined by $\x$ and $\eta$. In particular, this allows us to determine all angles between directions tangent to $\d M$ at a point $x \in \sU$, and hence the norms of covectors only up to a scaling factor. Therefore, the conformal structure of $g|_{\sU}$ is determined from the principal symbol $\theta_0(x,\x)$. \\

\noindent {\bf Determining $g|_{\sU}$ when $m \neq 0$, and $g|_{\sU}$ up to a constant when $m = 0$}

\vspace{10pt}

\noindent Let us now turn to the degree $0$ part of equation \eqref{local full symbol relation}. Since the degree $0$ part of $b^{--} \#\, \theta$ is
\begin{equation}
    \left( b^{--}\#\, \theta \right)_0 = b^{--}_1 \theta_1 + b^{--}_0 \theta_0 - i \sum_{\a} \d_{\x_{\a}} b^{--}_1 \cdot \d_{x^{\a}} \theta_0,
\end{equation}
we obtain
\begin{equation}\label{master relation degree 0 part}
    b^{-+}_0 + b^{--}_1 \theta_1 + b^{--}_0 \theta_0 - i \sum_{\a} \d_{\x_{\a}} b^{--}_1 \cdot \d_{\a} \theta_0 = \gamma^n \left( \Omega + \gamma(A) - m \right) + \frac{n-1}{2} H \theta_0.
\end{equation}
Since $b_1 = -|\x| \id$, we can invert this equation to solve for $\theta_{-1}$,
\begin{equation}\label{solve for theta_-1}
    \theta_{-1} = -\frac{1}{|\x|} \left( \gamma^n(\Omega + \gamma(A) - m) + \frac{n-1}{2} H \theta_0 - b_0^{-+} - b_0^{--}\theta_0 + i \sum_{\a} \d_{\x_{\a}} b^{--}_1 \cdot \d_{\a} \theta_0 \right).
\end{equation}
Note that
\begin{equation}
    \d_{\x_{\a}} b^{--}_1 = -g^{\a \b} \hat{\x}_{\b},
\end{equation}
where we recall that $\hat{\x}$ denotes the unit vector in the $\x$ direction; see \eqref{unit vector}. Moreover, we have
\begin{align}
    \d_{\a} \theta_0 &= -\d_{\a} \left( i \gamma^n \gamma^a h^{\b}_{\ a} \hat{\x}_{\b} \right) \nonumber \\
    &= -i\gamma^n \gamma^a \left( \d_{\a} h^{\b}_{\ a} \hat{\x}_{\b} - \frac{1}{2} h^{\b}_{\ a} \d_{\a}g^{\g \delta} \hat{\x}_{\b \g \delta} \right).
\end{align}
Therefore, the final term in \eqref{solve for theta_-1} is
\begin{equation}\label{sum term in theta_-1}
    i \sum_{\a} \d_{\x_{\a}} b^{--}_1 \cdot \d_{\a} \theta_0 = -\gamma^n \gamma^a \left( g^{\a \lambda} \d_{\a} h^{\b}_{\ a} \hat{\x}_{\lambda \beta} - \frac{1}{2} g^{\a \lambda} h^{\b}_{\ a} \d_{\a} g^{\g \delta} \hat{\x}_{\lambda \b \g \delta} \right).
\end{equation}
We now want to consider the part of $\theta_{-1}$ that is even in $\x$. From \eqref{solve for theta_-1} we have
\begin{equation}\label{theta_-1 form in terms of b}
    \theta_{-1}^{\mathfrak{e}} = -\frac{1}{|\x|} \left( \gamma^n \left( \Omega + \gamma(A) - m \right) - \left( b_0^{-+} \right)^{\mathfrak{e}} - \left( b_0^{--} \right)^{\mathfrak{o}} \theta_0 + i \sum_{\a} \d_{\x_{\a}} b^{--}_1 \cdot \d_{\a} \theta_0 \right),
\end{equation}
where we have used the fact that $\theta_0$ is purely odd while $b_1$ is purely even. Now, by using equation \eqref{sum term in theta_-1} and the expressions \eqref{b_0 -- even}--\eqref{b_0 -+ even} derived in Section \ref{subsec: symbol of Lambda}, we obtain
\begin{align}
    \theta_{-1}^{\mathfrak{e}} &= -\frac{1}{|\x|} \bigg( \gamma^n \left( \Omega + \gamma(A) - m \right) - i\left( g^{\a \b} \kappa^{--}_A(\d_\a) \hat{\x}_{\b} - \frac{1}{2} g^{\a \b} \Gamma^{\gamma}_{\ \a \b} \hat{\x}_{\g} - \frac{1}{4} g^{\gamma \delta}  \d_{\g} g^{\a \b} \hat{\x}_{\a \b \delta} \right) \theta_0  \nonumber \\
    &\hspace{170pt} -\gamma^n \gamma^a \left( g^{\a \lambda} \d_{\a} h^{\b}_{\ a} \hat{\x}_{\lambda \beta} - \frac{1}{2} g^{\a \lambda} h^{\b}_{\ a} \d_{\a} g^{\g \delta} \hat{\x}_{\lambda \b \g \delta} \right) \bigg).
\end{align}
Multiplying on the left by $\gamma^n$ and noting that $\gamma^n$ commutes with $\kappa_A^{--}(\d_{\a})$, we have
\begin{align}
    \gamma^n\theta_{-1}^{\mathfrak{e}} &= \frac{1}{|\x|} \bigg( \Omega + \gamma(A) - m  - \left(  g^{\a \b} \kappa^{--}_A(\d_\a) \hat{\x}_{\b \lambda} - \frac{1}{2} g^{\a \b} \Gamma^{\gamma}_{\ \a \b} \hat{\x}_{\g \lambda} - \frac{1}{4} g^{\gamma \delta}  \d_{\g} g^{\a \b} \hat{\x}_{\a \b \delta \lambda} \right) h^{\lambda}_{\ a} \gamma^a   \nonumber \\
    &\hspace{170pt} - \gamma^a \left( g^{\a \lambda} \d_{\a} h^{\b}_{\ a} \hat{\x}_{\lambda \beta} - \frac{1}{2} g^{\a \lambda} h^{\b}_{\ a} \d_{\a} g^{\g \delta} \hat{\x}_{\lambda \b \g \delta} \right) \bigg). \label{- gamma^n theta_ -1}
\end{align}
Note that $\gamma^n \theta_{-1}$ is endomorphism of $\bV_E^+$. Now, since we have determined the conformal class of $g|_{\sU}$, we can choose a reference metric $\bar{g}$ over $\sU$ such that $g^{\a \b} = e^{-2u} \bar{g}^{\a \b}$ for some yet unknown smooth function $u$ on $\sU$. We want to rewrite the expression on the right-hand-side of \eqref{- gamma^n theta_ -1} in terms of the known metric $\bar{g}$ and the unknown function $u$. First, let us consider the various terms appearing in \eqref{- gamma^n theta_ -1} separately. We define
\begin{equation}\label{cR_0 of g}
    \cR_0[g,A] := \frac{1}{|\x|_g} \left( \Omega + \gamma(A) \right),
\end{equation}
\begin{equation}\label{cR_1 of g}
    \cR_1[g,A] := - \frac{1}{|\x|_g} \left(  g^{\a \b} \kappa^{--}_A(\d_\a) \hat{\x}_{\b \lambda} - \frac{1}{2} g^{\a \b} \Gamma^{\gamma}_{\ \a \b} \hat{\x}_{\g \lambda} - \frac{1}{4} g^{\gamma \delta}  \d_{\g} g^{\a \b} \hat{\x}_{\a \b \delta \lambda} \right) h^{\lambda}_{\ a} \gamma^a,
\end{equation}
\begin{equation}\label{cR_2 of g}
    \cR_2[g] := - \frac{1}{|\x|_g} \gamma^a \left( g^{\a \lambda} \d_{\a} h^{\b}_{\ a} \hat{\x}_{\lambda \beta} - \frac{1}{2} g^{\a \lambda} h^{\b}_{\ a} \d_{\a} g^{\g \delta} \hat{\x}_{\lambda \b \g \delta} \right).
\end{equation}
Note that $|\x|_g = e^{-u} |\x|_{\bar{g}}$, and thus $\hat{\x} = e^u \bar{\hat{\x}}$. Moreover, since $e_a = e^{-u} \bar{e}_a$ we have $h^{\a}_{\ a} = e^{-u}\bar{h}^{\a}_{\ a}$. Thus, one readily computes that for $\cR_2[g]$, we have
\begin{equation}\label{cR_2 g conformal}
    \cR_2[g] = - \frac{1}{|\x|_{\bar{g}}} \bar{\gamma}^a \left( \bar{g}^{\a \lambda} \d_{\a} \bar{h}^{\b}_{\ a} \bar{\hat{\x}}_{\lambda \beta} - \frac{1}{2} \bar{g}^{\a \lambda} \bar{h}^{\b}_{\ a} \d_{\a} \bar{g}^{\g \delta} \bar{\hat{\x}}_{\lambda \b \g \delta} \right) = \cR_2[\bar{g}].
\end{equation}
Given the expression of $\Omega$ in \eqref{definition of Omega}, we find that
\begin{equation}\label{Omega under conformal}
    \Omega = e^{-u} \left( \bar{\Omega} + \frac{n-2}{2} \bar{\gamma}(du) \right).
\end{equation}
Indeed, equation \eqref{Omega under conformal} could also be deduced from \eqref{Dirac operator under conformal change v2} and the fact that $\Omega$ is nothing more than the $0$-th order part of the tangential Dirac operator on the boundary. We therefore get
\begin{equation}\label{cR_0 g conformal}
    \cR_0[g,A] = \frac{1}{|\x|_{\bar{g}}} \left( \bar{\Omega} + \frac{n-2}{2} \bar{\gamma}(du) + \bar{\gamma}(A) \right) = \cR_0[\bar{g}, A] + \frac{n-2}{2} \frac{1}{|\x|_{\bar{g}}} \bar{h}^{\gamma}_{\ a} \left( \d_{\gamma} u \right) \bar{\gamma}^a.
\end{equation}
Finally, let us consider the expression $\cR_1[g,A]$. Under a conformal scaling, we have
\begin{equation}
    \Gamma^{\gamma}_{\ \a \b} = \bar{\Gamma}^{\gamma}_{\ \a \b} + \delta^{\gamma}_{\a} \d_{\b} u + \delta^{\gamma}_{\b} \d_{\a} u - \bar{g}_{\a \b} \bar{g}^{\gamma \delta} \d_{\delta} u,
\end{equation}
and so the last two terms in parentheses in \eqref{cR_1 of g} can be written as
\begin{equation}\label{cR_1 g conformal term B}
    \frac{1}{2 |\x|_g} g^{\a \b} \Gamma^{\gamma}_{\ \a \b} h^{\lambda}_{\ a}  \hat{\x}_{\g \lambda} \gamma^a = \frac{1}{|\x|_{\bar{g}}} \left( \frac{1}{2} \bar{g}^{\a \b} \bar{\Gamma}^{\gamma}_{\ \a \b} \bar{h}^{\lambda}_{\ a}  \bar{\gamma}^a - \frac{n-3}{2} \bar{g}^{\gamma \delta} \bar{h}^{\lambda}_{\ a}  \bar{\gamma}^a \d_{\delta }u \right)\bar{\hat{\x}}_{\g \lambda}, 
\end{equation}
\begin{equation}\label{cR_1 g conformal term C}
     \frac{1}{4|\x|_g} g^{\gamma \delta}  \d_{\g} g^{\a \b} h^{\lambda}_{\ a} \hat{\x}_{\a \b \delta \lambda} \gamma^a = \frac{1}{|\x|_{\bar{g}}} \left( \frac{1}{4} \bar{g}^{\gamma \delta}  \d_{\g} \bar{g}^{\a \b} \bar{h}^{\lambda}_{\ a} \bar{\gamma}^a - \frac{1}{2} \bar{g}^{\a \b} \bar{g}^{\gamma \delta} \bar{h}^{\lambda}_{\ a} \bar{\gamma}^a \d_{\gamma} u  \right) \bar{\hat{\x}}_{\a \b \delta \lambda}.
\end{equation}
As for the first term in \eqref{cR_1 of g}, we first deduce how the connection $1$-form $\kappa_A^{--}(\d_{\a})$ changes under conformal scaling from either equation \eqref{spin connection under conformal change} or direct computation. We then find that
\begin{equation}\label{cR_1 g conformal term A}
    \frac{1}{|\x|_{g}} g^{\a \b} \kappa_A^{--}(\d_{\a}) h^{\lambda}_{\ a} \gamma^a \hat{\x}_{\b \lambda} = \frac{1}{|\x|_{\bar{g}}} \bar{g}^{\a \b} \left( \bar{\kappa}_A^{--}(\d_{\a})  + \frac{1}{2}  \bar{h}^{\gamma}_{\ b} (\bar{h}^{-1})^c_{\ \a} (\d_{\gamma} u) \bar{\gamma}^b \bar{\gamma}_c + \frac{1}{2} \d_{\a} u \right) \bar{h}^{\lambda}_{\ a} \bar{\gamma}^a \bar{\hat{\x}}_{\b \lambda}.
\end{equation}
In particular, putting together \eqref{cR_1 g conformal term B}, \eqref{cR_1 g conformal term C} and \eqref{cR_1 g conformal term A}, we have
\begin{equation}\label{cR_1 g conformal}
    \cR_{1}[g,A] = \cR_{1}[\bar{g}, A] - \frac{1}{|\x|_{\bar{g}}} \left( \frac{n-1}{2} \bar{g}^{\gamma \delta} \d_{\gamma} u + \frac{1}{2} \bar{g}^{\beta \delta} \bar{h}^{\gamma}_{\ b} (\bar{h}^{-1})^c_{\ \beta} (\d_{\gamma} u) \bar{\gamma}^b \bar{\gamma}_c \right) \bar{h}^{\lambda}_{\ a} \bar{\gamma}^a \bar{\hat{\x}}_{ \delta \lambda}.
\end{equation}
Letting $\cR[g,A] = \sum_i \cR_i[g,A]$, we now put together \eqref{cR_2 g conformal}, \eqref{cR_0 g conformal}, and \eqref{cR_1 g conformal}, to get
\begin{align}
    \gamma^n \theta_{-1}^{\mathfrak{e}} &= \cR[g,A] - \frac{m}{|\x|_g} \nonumber \\
    &= \cR[\bar{g}, A] + \frac{1}{|\x|_{\bar{g}}}\bigg( \frac{n-2}{2} \bar{g}^{\delta \lambda} \bar{h}^{\gamma}_{\ a} (\d_{\gamma} u)\bar{\gamma}^a - e^u \bar{g}^{\delta \lambda} m \nonumber \\
    &\ \ \ \ \ \ \ \ \ \ \ \ \ \ \ - \left( \frac{n-1}{2} \bar{g}^{\gamma \delta} \d_{\gamma} u + \frac{1}{2} \bar{g}^{\beta \delta} \bar{h}^{\gamma}_{\ b} (\bar{h}^{-1})^c_{\ \beta} (\d_{\gamma} u) \bar{\gamma}^b \bar{\gamma}_c \right) \bar{h}^{\lambda}_{\ a} \bar{\gamma}^a \bigg) \bar{\hat{\x}}_{\delta \lambda}.
\end{align}
Finally, as the connection $A$ is yet undetermined, we separate out the part of $\cR[g,A]$ that depends on $A$, which can easily be read off from \eqref{cR_0 of g}--\eqref{cR_2 of g}. Thus we get
\begin{align}
    \gamma^n \theta_{-1}^{\mathfrak{e}} &= \cR'[\bar{g}] + \frac{1}{|\x|_{\bar{g}}}\bigg( \frac{n-2}{2} \bar{g}^{\delta \lambda} \bar{h}^{\gamma}_{\ a} (\d_{\gamma} u)\bar{\gamma}^a + \bar{g}^{\delta \lambda}\bar{h}^{\a}_{\ a}A_{\a} \bar{\gamma}^a - e^u \bar{g}^{\delta \lambda} m \nonumber \\
    &\ \ \ \ \ \ \ \ \ \ \ \ \ \ \ - \left( \bar{g}^{\a \delta} A_{\a} + \frac{n-1}{2} \bar{g}^{\gamma \delta} \d_{\gamma} u + \frac{1}{2} \bar{g}^{\beta \delta} \bar{h}^{\gamma}_{\ b} (\bar{h}^{-1})^c_{\ \beta} (\d_{\gamma} u) \bar{\gamma}^b \bar{\gamma}_c \right) \bar{h}^{\lambda}_{\ a} \bar{\gamma}^a \bigg) \bar{\hat{\x}}_{\delta \lambda} \nonumber \\
    &=: \cR'[\bar{g}] + \frac{1}{|\x|_{\bar{g}}} \cT_0^{\delta \lambda} \bar{\hat{\x}}_{\delta \lambda} \label{theta_-1 under conformal change}
\end{align}
Therefore, as the reference metric $\bar{g}$ is known, we can determine the symmetrization of the tensor $\cT_0^{\delta \lambda}$ in \eqref{theta_-1 under conformal change}. Upon contracting with $\bar{g}_{\delta \lambda}$ and collecting terms, we can determine
\begin{equation}\label{equation to determine conformal factor}
     \bar{g}_{\delta \lambda} \cT_0^{\delta \lambda} = \frac{1}{2}(n-2)^2 \bar{\gamma}(du) + (n-2) \bar{\gamma}(A) - e^u(n-1) m.
\end{equation}
Note that if $n > 2$, then the gamma matrices $\gamma^a$ are traceless when restricted to $\bV_E^+$. Indeed, have the following more general result, which we will use in the rest of this section: 

\begin{lemma}\label{lemma: trace of gamma matrices}
    Let $\gamma^I := \gamma^{a_1} \cdots \gamma^{a_k}$ be a product of $k$ distinct gamma matrices corresponding to an orthonormal basis on the boundary, which has dimension $n - 1$. Then we have $\Tr_{\bV^{+}} \gamma^I = 0$ if either $k$ is even, or if $k$ is odd and $k < n-1$.
\end{lemma}

\begin{proof}
    If $k$ is even, then using the Clifford relations and the cyclic property of the trace, we have
    \begin{align}
        \Tr_{\bV^+} \left( \gamma^{a_1} \cdots \gamma^{a_k} \right) = -\Tr_{\bV^+} \left( \gamma^{a_k} \gamma^{a_1} \cdots \gamma^{a_{k-1}} \right) = -\Tr_{\bV^+} \left( \gamma^{a_1} \cdots \gamma^{a_{k}} \right),
    \end{align}
    and so $\Tr_{\bV^+} \gamma^I = 0$. On the other hand, if $k$ is odd and $k < n-1$, then there is $b \notin \{ a_1, \dots, a_k \}$. So, again using the Clifford relations and the cyclic property of the trace, we compute
    \begin{align}
        \Tr_{\bV^+} \gamma^I = - \Tr_{\bV^+} \left( \gamma^I \gamma^b \gamma^b \right) = \Tr_{\bV^+} \left( \gamma^b \gamma^I \gamma^b  \right) = \Tr_{\bV^+} \left( \gamma^b \gamma^b \gamma^I \right) = - \Tr_{\bV^+} \gamma^I,
    \end{align}
    and thus again in this case we have $\Tr_{\bV^+} \gamma^I = 0.$
\end{proof}

Note that the proof above fails if $k = n-1$. Indeed, in this case, one can show that $\gamma^I$ must be a multiple of the identity on $\bV^{\pm}$. This applies in particular to the single gamma matrix $\gamma^1$ in the case $n=2$, when the boundary is $1$-dimensional. We shall come back to this in Section \ref{subsec: proof for n=2}. For the time being, it is noteworthy that when $n \geq 3$, the gamma matrices $\gamma^a$ are traceless when restricted to $\bV^+$. Thus, when $m \neq 0$, we can take a trace of \eqref{equation to determine conformal factor} to obtain
\begin{equation}
    \Tr_{\bV_E^+} \left( \bar{g}_{\delta \lambda} \cT^{\delta \lambda}_0 \right) = -rm(n-1)e^u,
\end{equation}
where $r$ is the rank of $\bV_E^+$. Therefore, we can determine the conformal factor $u$ on $\sU$, and thus also the original metric $g|_{\sU}$. On the other hand, if $m = 0$, then \eqref{equation to determine conformal factor} reads
\begin{equation}\label{equation to determine conformal factor m=0}
     \bar{g}_{\delta \lambda} \cT_0^{\delta \lambda} =  \left( \frac{1}{2}(n-2)^2 \d_{\a} u + (n-2) A_{\a} \right) \bar{h}^{\alpha}_{\ a} \bar{\gamma}^a.
\end{equation}
By multiplying \eqref{equation to determine conformal factor m=0} by $\bar{\gamma}^b$, taking the partial trace over $\bV^+$, and using Lemma \ref{lemma: trace of gamma matrices}, one can obtain the coefficient of $\bar{\gamma}^b$ in \eqref{equation to determine conformal factor m=0}. By multiplying by $\bar{h}^{-1}$, we can thus determine 
\begin{equation}
    \cQ_{\a} := \frac{1}{2}(n-2)^2 \d_{\a}u + (n-2)A_{\a}.
\end{equation}
Finally, we note that since $A$ is a unitary connection, its local representative with respect to an orthonormal frame on $E$ has pure imaginary trace. This allows us to obtain
\begin{equation}
    \Re{\Tr_E (\cQ_{\a})} = \frac{1}{2}N_E(n-2)^2 \d_{\a}u,
\end{equation}
where $N_E$ is the rank of $E$. Therefore, we can determine the tangential derivatives of the conformal factor $u$ on $\sU$, and thus also the conformal factor $u$ up to a locally constant function. As this is the best we can do when $m = 0$, in that case we will need to assume that $g|_{\sU}$ has been specified in order to proceed. Either way, we henceforth assume that $g|_{\sU}$ has been determined.

We also record for later use that from \eqref{- gamma^n theta_ -1} and Lemma \ref{lemma: trace of gamma matrices}, we have
\begin{equation}
    \Tr_{\bV_E^+}\left( \gamma^n \theta_{-1}^{\mathfrak{e}} \right) = \frac{1}{|\x|} \left( \Tr{\Omega} - r m - h^{\lambda}_{\ a} g^{\a \b} \hat{\x}_{\b \lambda} \Tr\left( \kappa_A^{--}(\d_{\a}) \gamma^a \right) \right). \label{theta_ -1 trace}
\end{equation}
When $n \neq 4$, the first and last terms in \eqref{theta_ -1 trace} vanish by Lemma \ref{lemma: trace of gamma matrices}, since they involve a product of three distinct gamma matrices. They do not vanish, however, when $n = 4$. \\

\noindent {\bf Determining $A|_{\sU}$ and the traceless second fundamental form}

\vspace{10pt}

\noindent Now that the metric on the boundary and all tangential derivatives thereof are determined, we can go back to the even part of $\theta_{-1}$ in order to determine the gauge potential on the boundary. Note that proceeding as in the derivation of equation \eqref{equation to determine conformal factor}, we have
\begin{align}
    \gamma^n \theta_{-1}^{\mathfrak{e}} &= \frac{1}{|\x|} \left( \gamma(A) - g^{\a \b} A_{\a} h^{\lambda}_{\ a} \gamma^a \hat{\x}_{\b \lambda} \right) + \cD_{0,-1} \nonumber \\
    &= \gamma^a \left( g^{\b \lambda} h^{\a}_{\ a} A_{\a} - g^{\a \b} A_{\a} h^{\lambda}_{\ a}  \right) \frac{\hat{\x}_{\b \lambda}}{|\x|} + \cD_{0,-1}, \label{determining connection from theta_ -1 v1}
\end{align}
where $\cD_{0,-1}$ denotes quantities depending only on the metric on the boundary as well as tangential derivatives thereof, but not on the connection. Thus, we can determine the symmetrization of the tensor in parentheses in the final line of \eqref{determining connection from theta_ -1 v1}. Just as in \eqref{equation to determine conformal factor}, we contract with $g_{\b \lambda}$ to obtain
\begin{align}
    (n-1) h^{\a}_{\ a} A_{\a} - A_{\a} h^{\a}_{\ a} = (n-2) A_a.
\end{align}
Therefore, since we have assumed $n > 2$, the connection $A$ is determined on $\sU$ in directions along the boundary (recall that $A_n = 0$ for our choice of gauge).

We have thus far shown that the metric and connection are determined on the boundary when $m \neq 0$ by considering only the even part of $\theta_{-1}$. We now want to consider the odd part of $\theta_{-1}$, which will allows us to determine most of the normal derivatives of the metric on the boundary. To see how, we first use \eqref{solve for theta_-1} to write down the odd part,
\begin{equation}\label{theta_-1 odd part}
    \theta_{-1}^{\mathfrak{o}} = \frac{1}{|\x|} \left( \left( b_0^{-+} \right)^{\mathfrak{o}} + \left( b_0^{--} \right)^{\mathfrak{e}} \theta_0 - \frac{n-1}{2} H \theta_0  \right),
\end{equation}
which using equations \eqref{b_0 -- even}--\eqref{b_0 -+ even} becomes
\begin{align}
    \theta_{-1}^{\mathfrak{o}} &= \frac{1}{|\x|} \left( i g^{\a \b} \kappa_A^{-+}(\d_{\a}) \hat{\x}_{\b} - \frac{1}{4} \d_ng^{\a \b} \theta_0 \hat{\x}_{\a \b} \right) \nonumber \\
    &= -\frac{i}{2|\x|} \left( g^{\a \b} \sum_{b} \omega^b_{\ n}(\d_{\a}) \gamma^b \gamma^n \hat{\x}_{\b} - \frac{1}{2} \d_ng^{\a \b} \gamma^n \gamma^b h^{\lambda}_{\ b} \hat{\x}_{\a \b \lambda} \right) \nonumber \\
    &= \frac{i}{2|\x|} \gamma^n \left( g^{\a \lambda} g^{\delta \b} \omega^b_{\ n}(\d_{\delta}) + \frac{1}{2} \d_ng^{\a \b}   h^{\lambda}_{\ b}  \right)\gamma^b \hat{\x}_{\a \b \lambda}. \label{theta_-1 odd v1}
\end{align}
Note that we have
\begin{equation}
    \omega^b_{\ n}(\d_{\delta}) = g(e_b, \nabla_{\d_{\delta}} \d_n) = h^\b_{\ b} \,g(\d_{\b}, \Gamma^\lambda_{\ \delta n} \d_{\lambda}) = h^{\b}_{\ b} g_{\b \lambda} \Gamma^{\lambda}_{\ \delta n}, 
\end{equation}
and so using $\Gamma^{\lambda}_{\ \delta n} = \frac{1}{2} g^{\lambda \mu} \d_n g_{\mu \delta}$, we obtain
\begin{align}
    \omega^b_{\ n}(\d_{\delta}) &= h^{\b}_{\ b} g_{\b \lambda} \Gamma^{\lambda}_{\ \delta n} = \frac{1}{2} h^{\nu}_{\ b} g_{\nu \lambda} g^{\lambda \mu} \d_n g_{\mu \delta} = \frac{1}{2} h^{\mu}_{\ b} \d_n g_{\mu \delta}. \label{rewrite connection in terms of christoffel}
\end{align}
Plugging \eqref{rewrite connection in terms of christoffel} into \eqref{theta_-1 odd v1}, we have
\begin{align}
    \theta_{-1}^{\mathfrak{o}} &= \frac{i}{4|\x|} \gamma^n \left( g^{\a \lambda} g^{\delta \b} h^{\mu}_{\ b} \d_n g_{\mu \delta} + \d_ng^{\a \b}   h^{\lambda}_{\ b}  \right)\gamma^b \hat{\x}_{\a \b \lambda}. \label{theta_-1 odd v2}
\end{align}
From \eqref{theta_-1 odd v2}, we determine $\left( \cS_1 \right)^{\a \b \lambda}_b$, the symmetrization of the tensor in parentheses,
\begin{equation}\label{S ab lambda b}
    \left( \cS_1 \right)_b^{\a \b \lambda} := g^{(\a \lambda} g^{ \b )\delta} h^{\mu}_{\ b} \d_n g_{\mu \delta} + \d_ng^{(\a \b}   h^{\lambda)}_{\ b}.
\end{equation}
Since $g$ is determined on the boundary, we can also determine the matrix $h$ relating the coordinate and orthonormal frames. Thus, multiplying \eqref{S ab lambda b} by $(h^{-1})^{b}_{\ \gamma}$ and summing, we obtain
\begin{equation}
    (h^{-1})^{b}_{\ \gamma} \left( \cS_1 \right)^{\a \b \lambda}_b = g^{(\a \lambda} g^{ \b )\delta} \d_n g_{\gamma \delta} + \d_ng^{(\a \b}   \delta^{\lambda)}_{\gamma}.
\end{equation}
We now contract with $g_{\a \b}$ and $g_{\mu \lambda}$ to obtain
\begin{equation}\label{obtaining the SFF v1}
    g_{\a \b} g_{\mu \lambda} (h^{-1})^{b}_{\ \gamma} \left( \cS_1 \right)^{\a \b \lambda}_b = \frac{2(n-1)}{3} \left( H g_{\mu \gamma} + \frac{1}{2} \d_n g_{\mu \gamma} \right).
\end{equation}
The second term in parentheses in equation \eqref{obtaining the SFF v1} is equal to $-\mathrm{II}_{\mu \gamma}$, where $\mathrm{II}$ denotes the second fundamental form. We can therefore determine the traceless part of the second fundamental form,
\begin{equation}\label{defn of tracelss SFF}
    \Sigma_{\mu \lambda} := \mathrm{II}_{\mu \lambda} - H g_{\mu \lambda}.
\end{equation}
Note that in this way, having determined the metric on the boundary, from $\theta_{-1}$, we determine {\em almost} all of the normal derivatives of the metric on the boundary. Note that in the case $m = 0$ where we have a conformal symmetry, this is the best we can expect to do when the metric is prescribed on the boundary. Indeed, if we conformally scale the metric on $M$ by a factor $e^{2u}$, then the mean curvature $\bar{H}$ of the scaled metric on the boundary is given by
\begin{equation}
    \bar{H} = e^{-u} \left( H - \d_n u \right),
\end{equation}
and so we cannot determine the mean curvature even when $u$ is prescribed on the boundary. In the case $m = 0$, therefore, we will henceforth assume that $H|_{\sU}$ is prescribed. \\

\noindent {\bf Determining $\d_n g|_{\sU}$, $\d_n A|_{\sU}$ and $\d_n \Sigma|_{\sU}$ from $\theta_{-2}$}

\vspace{10pt} 

\noindent Before moving forward, we recall the notation introduced at the beginning of Section \ref{subsec: symbol of Lambda}, such as $\cD_{k}$, $\cD_{k,k-1}$, and $\slashed{\cD}_{k}$, which we shall use liberally in the rest of this section to keep track of known and unknown terms at each step. In light of the above discussion, we also introduce here the following new notation: let $\cD_{k+1}^*$ denote any quantity depending on $k$ normal derivatives of the metric, $k$ normal derivatives of the connection, and $k$ normal derivatives of $\Sigma$, the traceless second fundamental form. Thus $\cD_{k+1}^*$ denotes quantities depending on almost all $(k+1)$-th order normal derivatives of the metric, in that they do not depend on the $k$-th normal derivatives of the mean curvature. Note that we clearly have $\cD_k \subseteq \cD^*_{k+1} \subseteq \cD_{k+1}$. Finally, by \eqref{defn of tracelss SFF}, we have
\begin{equation}\label{normal derivatives modulo D_1*}
    \d_n g_{\a \b} = -2 H g_{\a \b} + \cD_1^*, \ \ \ \ \ \ \d_n g^{\a \b} = 2H g^{\a \b} + \cD_1^*.
\end{equation}
We will also include in the above symbols any quantities depending on the symbol of $\Theta_{g,A,m}$ and tangential derivatives thereof, as these are assumed to be known.

Moreover, since we have determined the metric on the boundary, we can now construct various orthonormal frames tangent to the boundary, and consider the symbol of $\Theta_{g,A,m,\sU}$ computed in local trivializations induced by these orthonormal frames. This freedom will allow us to greatly simplify the following computations. As indicated in the discussion following equation \eqref{b_-1 -+ even}, we will consider orthonormal frames on the boundary $(e_a)$ that are parallel at a given point $x_0 \in \sU$. We recall that this means the connection coefficients $\omega^a_{\ bc}$ vanish at $x_0$. Moreover, since the connection is also determined on the boundary, we can construct a frame of $E$ so that $A$ vanishes at $x_0$, and thus we have that $\kappa_A^{\pm \pm}$ vanishes at $x_0$ with respect to this frame. Note that by equation \eqref{theta_ -1 trace}, in such a local trivialization we have that
\begin{equation}\label{theta_ -1 trace simplified}
    \Tr_{\bV_E^+}\left( \gamma^n \theta_{-1}^{\mathfrak{e}} \right) \big|_{x_0} = -\frac{rm}{|\x|},
\end{equation}
since we have also $\Omega(x_0) = 0$. This fact shall be used below.

We are now ready to compute $\theta_{-2}$ modulo terms in $\cD_{1}^*$, which were determined above. The degree $-1$ part of equation \eqref{local full symbol relation} gives us
\begin{equation}
    b^{-+}_{-1} + (b^{--} \#\, \theta)_{-1} = \frac{n-1}{2} H \theta_{-1},
\end{equation}
and the degree $-1$ part of $b^{--} \#\, \theta$ is
\begin{equation}\label{b comp theta deg -1}
    b_{-1}^{--} \theta_0 + b_{0}^{--} \theta_{-1} + b_1^{--} \theta_{-2} - i \d_{\x} b_0^{--} \d_{x'} \theta_0 + \d_{\x}^2 b_{1} D_{x'}^2 \theta_0.
\end{equation}
Note that the last term in \eqref{b comp theta deg -1} is $\cD_0$. We therefore get
\begin{equation}\label{theta_-2 eqn}
    \theta_{-2} = \frac{1}{|\x|} \left( b_{-1}^{-+} + b_{-1}^{--} \theta_0 + b_{0}^{--} \theta_{-1} - i \d_{\x} b_0^{--} \cdot \d_{x'} \theta_0 - \frac{n-1}{2} H \theta_{-1} \right) + \cD_0.
\end{equation}
Note that using the expression \eqref{b_0 pm pm} for $b_0$, we have that
\begin{equation}
    b_0^{--} \theta_{-1} - \frac{n-1}{2} H \theta_{-1} = - \frac{1}{4} \d_n g^{\a \b} \hat{\x}_{\a \b} \theta_{-1} + \cD_0.
\end{equation}
We thus have
\begin{equation}\label{theta_-2 eqn v2}
    \theta_{-2} = \frac{1}{|\x|} \left( b_{-1}^{-+} + b_{-1}^{--} \theta_0 - \frac{1}{4} \d_n g^{\a \b} \hat{\x}_{\a \b} \theta_{-1} - i \d_{\x} b_0^{--} \cdot \d_{x'} \theta_0 \right) + \cD_0.
\end{equation}
Now, just as with $\theta_{-1}$ we want to consider the even part of $\theta_{-2}$. We have
\begin{equation}\label{theta_-2 even part}
    \theta^{\mathfrak{e}}_{-2} = \frac{1}{|\x|} \left(  \left( b_{-1}^{-+} \right)^{\mathfrak{e}} + \left( b_{-1}^{--} \right)^{\mathfrak{o}} \theta_0 - \frac{1}{4} \d_n g^{\a \b} \hat{\x}_{\a \b} \theta^{\mathfrak{e}}_{-1} - i \d_{\x} \left( b_0^{--} \right)^{\mathfrak{e}} \cdot \d_{x'} \theta_0 \right) + \cD_0.
\end{equation}
We now want to evaluate $\theta_{-2}^{\mathfrak{e}}$ at $x_0$. Using \eqref{b_-1 -+ even at x_0 v2}, \eqref{b_-1 -- odd at x_0 v2}, and \eqref{b_0 pm pm}, we get that
\begin{equation}\label{theta_-2 even at x_0}
    \theta^{\mathfrak{e}}_{-2} \big|_{x_0} = \frac{1}{|\x|} \left( \frac{1}{2|\x|} i g^{\a \b} \d_n \left( \kappa^{--}_A(\d_\a) \right) \hat{\x}_{\b} \theta_0 - \frac{1}{4} \d_n g^{\a \b} \hat{\x}_{\a \b} \theta^{\mathfrak{e}}_{-1} + \gamma^n \slashed{\cD}_1 \right) \bigg|_{x_0} + \cD_0.
\end{equation}
Therefore, when we multiply by $\gamma^n$ and take a trace over $\bV_E^+$ just as with $\theta_{-1}$, the third term in \eqref{theta_-2 even at x_0} vanishes. By equation \eqref{theta_ -1 trace simplified} we have
\begin{align}
    \Tr_{\bV_E^+} \left( \gamma^n \theta^{\mathfrak{e}}_{-2} \right) \big|_{x_0} &= \frac{rm}{4|\x|^2} \d_n g^{\a \b} \hat{\x}_{\a \b} + \frac{i}{2|\x|^2} g^{\a \b} \Tr_{\bV_E^+} \left( \gamma^n \d_n\left( \kappa_A^{--}(\d_{\a})  \right) \theta_0 \right) \hat{\x}_{\b} + \cD_0 \nonumber \\
    &= \frac{rm}{4|\x|^2} \d_n g^{\a \b} \hat{\x}_{\a \b} - \frac{1}{4|\x|^2} g^{\a \b} \left( h^{-1} \right)^{d}_{\ \a} h^{\gamma}_{\ a} \sum_{b < c} \d_n\omega^b_{\ dc} \Tr\left( \gamma^n \gamma^b \gamma^c \gamma^n \gamma^a \right) \hat{\x}_{\b \gamma} + \cD_0 \nonumber \\
    &= \frac{1}{2|\x|^2} \left( rm \d_n g^{\b \gamma} + g^{\a \b} \left( h^{-1} \right)^{d}_{\ \a} h^{\gamma}_{\ a} \sum_{b < c} \d_n\omega^b_{\ dc} \Tr\left( \gamma^b \gamma^c \gamma^a \right) \right) \hat{\x}_{\b \gamma} + \cD_0. \label{theta_-2 trace at x_0}
\end{align}
Since the metric is known on $\sU$, we can determine the symmetric tensor
\begin{equation}
    \cT_1^{\b \gamma} := rm \d_n g^{\b \gamma} + g^{\a (\b} h^{\gamma)}_{\ a} \left( h^{-1} \right)^{d}_{\ \a}  \sum_{b < c} \d_n \omega^b_{\ dc} \Tr\left( \gamma^b \gamma^c \gamma^a \right)
\end{equation}
If $n \neq 4$, then the boundary is not $3$-dimensional, and so by Lemma \ref{lemma: trace of gamma matrices}, we can recover $\d_n g^{\b \gamma}$ at the point $x_0$ when $m \neq 0$. Since $x_0$ is arbitrary, the normal derivative of the metric is determined on $\sU$. Let us therefore consider the case $n = 4$, in which the boundary is $3$-dimensional and the trace in \eqref{theta_-2 trace at x_0} does not necessarily vanish. In this case, we have
\begin{align}
    g_{\b \gamma} \cT_1^{\b \gamma} &= \frac{rm(n-1)}{2}H + \sum_{b < c} \d_n\omega^b_{\ ac} \Tr\left( \gamma^b \gamma^c \gamma^a \right) \nonumber \\
    &= \frac{rm(n-1)}{2} H - \d_n\left( \omega^1_{\ 23} + \omega^2_{\ 31} + \omega^3_{\ 12} \right) \Tr\left( \gamma^1 \gamma^2 \gamma^3 \right). \label{trace of T_1 for n=4}
\end{align}
In order to analyze the last term in \eqref{trace of T_1 for n=4}, we prove the following 
\begin{lemma}\label{lemma: upsilon derivatives}
    Let $n = 4$ and $k \geq 0$. Then, with respect to an orthonormal frame on the boundary that is parallel at $x_0$ and parallel along $e_n$, we have 
    \begin{equation}
        \d_n^k \left( \omega^1_{\ 23} + \omega^2_{\ 31} + \omega^3_{\ 12} \right) \big|_{x_0} = \cD_k^*.
    \end{equation}
\end{lemma}

\begin{proof}
    Let $\Upsilon := \omega^1_{\ 23} + \omega^2_{\ 31} + \omega^3_{\ 12}$. By our choice of frame, we have $\Upsilon|_{x_0} = 0$. Let us now consider the normal derivative of $\Upsilon$, and for this, let us first consider the normal derivative of one term:
    \begin{align}
        \d_n \omega^1_{\ 23} &= \nabla_{e_n} g(e_1, \nabla_{e_2} e_3) \nonumber \\
        &= g(\nabla_{e_n} e_1, \nabla_{e_2} e_3) + g(e_1, \nabla_{e_n} \nabla_{e_2} e_3).
    \end{align}
    The first term vanishes since our frame is parallel along $e_n$. For the second term we use
    \begin{equation}
        \nabla_X \nabla_Y Z - \nabla_Y \nabla_X Z - \nabla_{[X,Y]} Z = R(X,Y)Z
    \end{equation}
    to obtain
    \begin{align}
        \d_n \omega^1_{\ 23} &= g(e_1, \nabla_{e_2} \nabla_{e_n} e_3) + g(e_1, \nabla_{[e_n, e_2]} e_3) + g(e_1, R(e_n, e_2)e_3) \nonumber \\
        &= g(e_1, \nabla_{[e_n, e_2]} e_3) + g(e_1, R(e_n, e_2)e_3). \label{upsilon lemma eqn 1}
    \end{align}
    By the symmetries of the Riemann tensor, the last term in \eqref{upsilon lemma eqn 1} is
    \begin{equation}
        g(R(e_n, e_2)e_3, e_1) = -g(R(e_2, e_n)e_3, e_1) = -g(R(e_3, e_1)e_2, e_n).
    \end{equation}
    For the first term in \eqref{upsilon lemma eqn 1}, note that
    \begin{align}\label{upsilon lemma bracket eqn}
        [e_n, e_2] = [\d_n, h^{\a}_{\ 2} \d_{\a}] = \left( \d_n h^{\a}_{\ 2} \right) \d_{\a}.
    \end{align}
    Now, using the fact that $\nabla_{e_n} e_a = 0$ near $\d M$ and \eqref{normal derivatives modulo D_1*}, we find
    \begin{align}
        0 = \nabla_{e_n} e_a &= \left( \d_n h^{\a}_{\ a} \right) \d_{\a} + h^{\a}_{\ a} \Gamma^\b_{\ n \a} \d_{\b} \nonumber \\
        &= \left( \d_n h^{\a}_{\ a} \right) \d_{\a} + \frac{1}{2} h^{\a}_{\ a} g^{\b \delta} \d_n g_{\delta \a} \d_{\b} \nonumber \\
        &= \left( \d_n h^{\a}_{\ a} \right) \d_{\a} - H h^{\b}_{\ a} \d_{\b} + \cD_1^* \nonumber \\
        &= \left( \d_n h^{\a}_{\ a} \right) \d_{\a} - H e_a + \cD_1^*. \label{normal derivative of h modulo D_1*}
    \end{align}
    Thus, equation \eqref{upsilon lemma bracket eqn} becomes
    \begin{equation}
        [e_n, e_2] = \left( \d_n h^{\a}_{\ 2} \right) \d_{\a} = H e_2 + \cD_1^*,
    \end{equation}
    and therefore, equation \eqref{upsilon lemma eqn 1} becomes
    \begin{align}
        \d_n \omega^1_{\ 23} &= g(e_1, \nabla_{[e_n, e_2]} e_3) + g(e_1, R(e_n, e_2)e_3) \nonumber \\
        &= Hg(e_1, \nabla_{e_2} e_3) -g(R(e_3, e_1)e_2, e_n) + \cD_1^* \nonumber \\
        &= H \omega^1_{\ 23} - g(R(e_3, e_1)e_2, e_n) + \cD_1^*. \label{upsilon lemma eqn 2}
    \end{align}
    To get $\d_n \Upsilon$, we sum the cyclic permutations of equation \eqref{upsilon lemma eqn 2}. The sum of terms involving the Riemann curvature vanish due to the algebraic Bianchi identity. We therefore obtain
    \begin{equation}\label{upsilon lemma main eqn}
        \d_n \Upsilon = H \Upsilon + \cD_1^*.
    \end{equation}
    Since $\Upsilon|_{x_0} = 0$, we get $\d_n \Upsilon|_{x_0} = \cD_1^*$. For $k > 1$, note that equation \eqref{upsilon lemma main eqn} holds not only on $\d M$, but in a neighbourhood of the boundary. Therefore, we can differentiate equation \eqref{upsilon lemma main eqn} to obtain
    \begin{equation}
        \d_n^k \Upsilon = \left( \d_n^{k-1} H \right) \Upsilon + \cD_k^*
    \end{equation}
    for $k > 1$ by induction. Evaluating at $x_0$ and using $\Upsilon|_{x_0} = 0$ completes the proof.
\end{proof}

Thus, by Lemma \ref{lemma: upsilon derivatives}, equation \eqref{trace of T_1 for n=4} becomes
\begin{align}
    g_{\b \gamma} \cT_1^{\b \gamma} &= \frac{rm(n-1)}{2} H - \d_n\left( \omega^1_{\ 23} + \omega^2_{\ 31} + \omega^3_{\ 12} \right) \Tr\left( \gamma^1 \gamma^2 \gamma^3 \right) \nonumber \\
    &= \frac{rm(n-1)}{2} H + \cD_1^*.
\end{align}
Therefore, even in the case $n = 4$, we can determine $H$ at $x_0$ if $m \neq 0$. Since $x_0$ is arbitrary, the normal derivative of $g$ is determined on $\sU$. In the case $m = 0$, we of course must assume that the mean curvature is prescribed on $\sU$ in order to proceed with the proof.

Having determined the normal derivatives of $g$, we can determine all quantities in $\cD_{1,0}$, i.e. all quantities depending on at most $1$ normal derivative of the metric and on $0$ normal derivatives of the connection. Now we shall go back to the even part of $\theta_{-2}$ in order to determine the normal derivatives of the connection. From equations \eqref{theta_-2 even part} and Proposition \ref{prop: precise form of b -(k+1)} we obtain
\begin{align}
    \theta_{-2}^{\mathfrak{e}} &= \frac{1}{|\x|} \left( \left( b_{-1}^{-+} \right)^{\mathfrak{e}} + \left( b^{--}_{-1} \right)^{\mathfrak{o}} \theta_0 + \cD_{1,0} \right) \nonumber \\
    &= \frac{1}{2|\x|^2} \left( -q_0^{-+} + i \theta_0 g^{\a \b} \d_n A_{\a} \hat{\x}_{\b} + \cD_{1,0} \right) \nonumber \\
    &= -\frac{1}{2|\x|^2} \left( \fF_A^{-+} - \gamma^n \gamma^a h^{\gamma}_{\ a} g^{\a \b} \d_n A_{\a} \hat{\x}_{\b \gamma} + \cD_{1,0} \right) \nonumber \\
    &= -\frac{1}{2|\x|^2} \gamma^n \gamma^a \left( F_A(e_n,e_a)g^{\b \gamma} -  h^{\gamma}_{\ a} g^{\a \b} \d_n A_{\a} \right) \hat{\x}_{\b \gamma} + \cD_{1,0}  \nonumber \\
    &= -\frac{1}{2|\x|^2} \gamma^n \gamma^a \left(h^{\a}_{\ a} g^{\b \gamma} \d_nA_{\a} -  h^{\gamma}_{\ a} g^{\a \b} \d_n A_{\a} \right) \hat{\x}_{\b \gamma} + \cD_{1,0}.
\end{align}
We thus determine the symmetrized tensor in parentheses, which on contracting with $g_{\b \gamma}$ yields
\begin{equation}
    (n-1) h^{\a}_{\ a} \d_nA_{\a} - h^{\gamma}_{\ a} \d_n A_{\gamma} = (n-2) h^{\a}_{\ a} \d_n A_{\a}.
\end{equation}
We can therefore recover $\d_n A_{\a}$ on the boundary as $n > 2$. Note here that the determination of $\d_n A|_{\sU}$ is valid with respect to a fixed local trivialization, as we are not evaluating at $x_0$.

We have so far recovered from $\theta_{-2}^{\mathfrak{e}}$ the first derivatives of the metric when $m \neq 0$, and the first derivatives of the connection. We now want to recover $\d_n \Sigma_{\a \b}$ from the odd part of $\theta_{-2}$, in complete analogy with how we recovered $\Sigma_{\a \b}$ from $\theta_{-1}^{\mathfrak{o}}$. So, having determined all quantities in $\cD_1$, we can now use equations \eqref{theta_-2 eqn}, \eqref{b_-1 -+ highest order} and \eqref{b_-1 -- highest order} to write
\begin{align}
    \theta^{\mathfrak{o}}_{-2} &= \frac{1}{|\x|} \left( \left( b_{-1}^{-+} \right)^{\mathfrak{o}} + \left( b_{-1}^{--} \right)^{\mathfrak{e}} \theta_0 \right) + \cD_1 \nonumber \\
    &= \frac{1}{2|\x|^2} \left(  i g^{\a \b} \d_n \left( \kappa_A^{-+}(\d_{\a}) \right) \hat{\x}_{\b} - \frac{1}{4} \theta_0\,  \d^2_n g^{\a \b} \hat{\x}_{\a\b} \right) + \cD_1 \nonumber \\
    &= \frac{i}{4|\x|^2} \gamma^n \gamma^a \left(- g^{\a \b} \d_n \left( \omega^n_{\ a}(\d_{\a}) \right) \hat{\x}_{\b} + \frac{1}{2} h^{\gamma}_{\ a}  \d^2_n g^{\a \b} \hat{\x}_{\a\b \gamma} \right) + \cD_1 \nonumber \\
    &= \frac{i}{4|\x|^2} \gamma^n \gamma^a \left( -g^{\a \gamma}g^{\delta \b} \d_n \left( \omega^n_{\ a}(\d_{\delta}) \right) + \frac{1}{2} h^{\gamma}_{\ a}  \d^2_n g^{\a \b}  \right)\hat{\x}_{\a\b \gamma} + \cD_1. 
\end{align}
Using equation \eqref{rewrite connection in terms of christoffel} and the fact that $\d_n h^{\mu}_{\ a} = \cD_1$, we have
\begin{equation}
    \theta_{-2}^{\mathfrak{o}} = \frac{i}{8|\x|^2} \gamma^n \gamma^a \left( g^{\a \gamma}g^{\delta \b} h^{\mu}_{\ a} \d^2_ng_{\mu \delta} +  h^{\gamma}_{\ a}  \d^2_n g^{\a \b}  \right)\hat{\x}_{\a\b \gamma} + \cD_1,
\end{equation}
and we can therefore recover the symmetric tensor
\begin{equation}
    \left( \cS_2 \right)^{\a \b \gamma}_{a} := g^{( \a \gamma}g^{\b) \delta} h^{\mu}_{\ a} \d_ng_{\mu \delta} +  \d^2_n g^{(\a \b} h^{\gamma)}_{\ a}.
\end{equation}
Proceeding as in equation \eqref{obtaining the SFF v1}, we compute
\begin{align}
    \frac{3}{2(n-1)}g_{\a \b} g_{\gamma \sigma} \left( h^{-1} \right)^a_{\ \lambda} \left( \cS_2 \right)^{\a \b \gamma}_{a} &:= g_{\lambda \sigma} g_{\a \b} \d_n^2 g^{\a \b} + \frac{1}{2} \d_n^2 g_{\lambda \sigma} \nonumber \\
    &=  g_{\lambda \sigma} \d_n \left( g_{\a \b} \d_n g^{\a \b} \right) + \frac{1}{2} \d_n^2 g_{\lambda \sigma}  + \cD_1 \nonumber \\
    &=  g_{\lambda \sigma} \d_n H + \frac{1}{2} \d_n^2 g_{\lambda \sigma} + \cD_1 \nonumber \\
    &= \d_n \left( H g_{\lambda \sigma} + \frac{1}{2} \d_n g_{\lambda \sigma} \right) + \cD_1 \nonumber \\
    &= - \d_n \Sigma_{\lambda \sigma} + \cD_1. \label{obtaining derivative of SFF}
\end{align}
And thus the normal derivative of $\Sigma$ is determined on the boundary, along with all quantities in $\cD_2^*$. That is to say, we have determined {\em almost} all second derivatives of the metric, save for those depending on $\d_n H$. We are now ready to proceed to the inductive step. \\

\noindent {\bf Determining $\d_n^{k+1} g|_{\sU}$, $\d_n^{k+1}A|_{\sU}$ and $\d_n^{k+1} \Sigma|_{\sU}$ from $\theta_{-(k+2)}$}

\vspace{10pt} 

\noindent Suppose now that we have determined all quantities in $\cD_{k+1}^*$ on $\sU$. That is, we have determined the normal derivatives of the metric and connection up to order $k$, and the normal derivatives of $\Sigma$ up to order $k$. We shall proceed in analogy with the above, first determining the $(k+1)$-th normal derivative of $g$ from the even part of $\theta_{-(k+2)}$ when $m \neq 0$, then determining the $(k+1)$-th normal derivative of the connection, and finally determining the $(k+1)$-th derivative of $\Sigma$ from the odd part of $\theta_{-(k+2)}$. We start with the degree $-(k+1)$ part of equation \eqref{local full symbol relation},
\begin{align}
    b^{-+}_{-(k+1)} + \left( b^{--} \#\, \theta \right)_{-(k+1)} = \frac{n-1}{2} H \theta_{-(k+1)}. \label{degree -(k+1) part of main eqn}
\end{align}
Note that since the symbol of $\theta$ is known and $H$ is determined, the right-hand side of \eqref{degree -(k+1) part of main eqn} is $\cD_k$. Now, using the composition rule for symbols and Proposition \ref{prop: form of b_-(k+1) to highest order}, we have
\begin{align}
    \left( b^{--} \#\, \theta \right)_{-(k+1)} &= \sum_{0 \leq i,j \leq k+2} \hspace{5pt} \sum_{|\n| + i + j = k+2} \frac{(-i)^{|\n|}}{\nu!} \d_{\x}^{\nu} b^{--}_{1-i} \, \d_{x'}^{\n} \theta_{-j} \nonumber \\
    &= b_1^{--} \theta_{-(k+2)} + b^{--}_{-(k+1)} \theta_0 + b_{-k}^{--} \theta_{-1} -i \d_{\x} b_{-k}^{--} \d_{x} \theta_0 + \cD_{k}. \label{composition b theta degree -(k+1)}
\end{align}
Plugging \eqref{composition b theta degree -(k+1)} into \eqref{degree -(k+1) part of main eqn} and solving for $\theta_{-(k+2)}$, we get
\begin{equation}
    \theta_{-(k+2)} = \frac{1}{|\x|} \left( b^{-+}_{-(k+1)} + b^{--}_{-(k+1)} \theta_0 + b_{-k}^{--} \theta_{-1} -i \d_{\x} b_{-k}^{--} \d_{x} \theta_0 \right) + \cD_k. \label{theta_-(k+2)}
\end{equation}
Using Proposition \ref{prop: form of b_-(k+1) to highest order}, we find that the even part of $\theta_{-(k+2)}$ is
\begin{align}
    \theta^{\mathfrak{e}}_{-(k+2)} &= \frac{1}{|\x|} \bigg( \left( b^{-+}_{-(k+1)} \right)^{\mathfrak{e}} + \left( b^{--}_{-(k+1)} \right)^{\mathfrak{o}} \theta_0 + \left( b_{-k}^{--} \right)^{\mathfrak{e}} \theta_{-1}^{\mathfrak{e}} + \left( b_{-k}^{--} \right)^{\mathfrak{o}} \theta_{-1}^{\mathfrak{o}} \nonumber \\
    &\hspace{250pt} -i \d_{\x} \left( b_{-k}^{--} \right)^{\mathfrak{e}} \d_{x} \theta_0 \bigg) + \cD_k \nonumber \\
    &= \frac{1}{|\x|} \bigg( \left( b^{-+}_{-(k+1)} \right)^{\mathfrak{e}} + \left( b^{--}_{-(k+1)} \right)^{\mathfrak{o}} \theta_0 + \left( b_{-k}^{--} \right)^{\mathfrak{e}} \theta_{-1}^{\mathfrak{e}} + \gamma^n \slashed{\cD}_{k+1,k} \bigg) + \cD_k  \label{theta_ -(k+2) even part}
\end{align}
As before, we compute this symbol with respect to an orthonormal frame that is parallel at $x_0$, and then evaluate at $x_0$. Propositions \ref{prop: form of b_-(k+1) to highest order} and \ref{prop: precise form at x_0} then give us
\begin{align}
    \theta^{\mathfrak{e}}_{-(k+2)} \big|_{x_0} &= \frac{1}{|\x|} \bigg(  -\frac{1}{2^{k+1}|\x|^{k+1}}  \d_n^{k} \fF_A^{-+} + \left( \frac{i}{2^{k+1}|\x|^{k+1}} g^{\a \b} \d_n^{k+1} \left( \kappa_A^{--}(\d_{\a}) \right) \hat{\x}_{\b} + \cD_{k+1,k} \cdot \id \right) \theta_0 \nonumber \\
    &\hspace{200pt} + \left( b_{-k}^{--} \right)^{\mathfrak{e}} \theta_{-1}^{\mathfrak{e}} + \gamma^n \slashed{\cD}_{k+1,k} \bigg)\bigg|_{x_0} + \cD_k \nonumber \\
    &= \frac{1}{|\x|} \bigg( \gamma^n \slashed{\cD}_{k+1} +  \frac{i}{2^{k+1}|\x|^{k+1}} g^{\a \b} \d_n^{k+1} \left( \kappa_A^{--}(\d_{\a}) \right) \hat{\x}_{\b} \theta_0  + \left( b_{-k}^{--} \right)^{\mathfrak{e}}  \theta_{-1}^{\mathfrak{e}}  \bigg)\bigg|_{x_0} + \cD_k \nonumber \\
    &= \gamma^n \slashed{\cD}_{k+1} + \frac{1}{2^{k+1}|\x|^{k+1}} \left( \frac{i}{|\x|} g^{\a \b} \d_n^{k+1} \left( \kappa_A^{--}(\d_{\a}) \right) \hat{\x}_{\b} \theta_0 + \frac{1}{2} \left( \d^{k+1}_n g^{\a \b} \right) \hat{\x}_{\a \b} \theta_{-1}^{\mathfrak{e}} \right)\bigg|_{x_0} + \cD_k. \label{theta_ (k+2) even part at x_0}
\end{align}
It immediately follows from equation \eqref{theta_ -1 trace simplified} that
\begin{align}
    \Tr_{\bV_E^+} \left( \gamma^n \theta_{-(k+2)}^{\mathfrak{e}} \right) \Big|_{x_0} &= \frac{1}{2^{k+1}|\x|^{k+2}} \Big( i g^{\a \b} \Tr\left( \gamma^n \d_n^{k+1} \left( \kappa_A^{--}(\d_{\a}) \right) \theta_0 \right) \hat{\x}_{\b} - \frac{rm}{2} \d_n^{k+1}g^{\a \b} \hat{\x}_{\a \b} \Big)\Big|_{x_0} + \cD_k. \label{trace of theta -(k+2)}
\end{align}
Now, if $n \neq 4$, then the first term in \eqref{trace of theta -(k+2)} vanishes, and if $m \neq 0$, then we can recover $\d_n^{k+1}g^{\a \b}$ at $x_0$. Since $x_0$ is arbitrary, the $(k+1)$-th normal derivatives of $g$ are determined on $\sU$. If $n = 4$, then the first term does not necessarily vanish, and we argue in analogy with \eqref{theta_-2 trace at x_0}:
\begin{align}
    \Tr_{\bV_E^+} \left( \gamma^n \theta_{-(k+2)}^{\mathfrak{e}} \right) \Big|_{x_0} &= \frac{1}{2^{k+1}|\x|^{k+2}} \Big( i g^{\a \b} \Tr\left( \gamma^n \d_n^{k+1} \left( \kappa_A^{--}(\d_{\a}) \right) \theta_0 \right) \hat{\x}_{\b} - \frac{rm}{2} \d_n^{k+1}g^{\a \b} \hat{\x}_{\a \b} \Big)\Big|_{x_0} + \cD_k \nonumber \\
    &= -\frac{1}{2^{k+2}|\x|^{k+2}} \bigg( g^{\delta \b} h^{\a}_{\ c} \sum_{a<b} \left( \d_n^{k+1} \omega^a_{\ b}(\d_{\delta}) \right) \Tr\left( \gamma^n \gamma^a \gamma^b \gamma^n \gamma^c \right) \hat{\x}_{\a \b} \nonumber \\
    &\hspace{200pt} + rm \d_n^{k+1}g^{\a \b} \hat{\x}_{\a \b} \bigg)\bigg|_{x_0} + \cD_k \nonumber \\
    &= -\frac{1}{2^{k+2}|\x|^{k+2}} \bigg( g^{\delta \b} h^{\a}_{\ c} \left( h^{-1} \right)^d_{\ \delta} \sum_{a<b} \left( \d_n^{k+1} \omega^a_{\ db} \right) \Tr\left( \gamma^a \gamma^b \gamma^c \right) \nonumber \\
    &\hspace{180pt} + rm \d_n^{k+1}g^{\a \b}  \bigg) \bigg|_{x_0} \hat{\x}_{\a \b} + \cD_k,
\end{align}
and thus we can recover the symmetric tensor
\begin{equation}
    \cT_{k+1}^{\a \b} := rm \d_n^{k+1}g^{\a \b} + g^{\delta (\b} h^{\a)}_{\ c} \left( h^{-1} \right)^d_{\ \delta} \sum_{a<b} \left( \d_n^{k+1} \omega^a_{\ db} \right) \Tr\left( \gamma^a \gamma^b \gamma^c \right),
\end{equation}
at $x_0$, which upon contracting with $g_{\a \b}$ yields
\begin{align}
    g_{\a \b}\cT_{k+1}^{\a \b} \big|_{x_0} &= rm g_{\a \b} \left( \d_n^{k+1}g^{\a \b} \right) \Big|_{x_0} + \sum_{a < b} \left(  \d_n^{k+1} \omega^a_{\ cb} \right) 
    \Big|_{x_0} \Tr\left( \gamma^a \gamma^b \gamma^c \right) \nonumber \\
    &= \frac{rm(n-1)}{2} \d^k_n H \big|_{x_0} - \d_n^{k+1}\left(\omega^{1}_{\ 23} + \omega^2_{\ 31} + \omega^3_{\ 12} \right) \big|_{x_0} \Tr\left( \gamma^1 \gamma^2 \gamma^3 \right) + \cD_{k}. \label{contracting T_ k+1}
\end{align}
By Lemma \ref{lemma: upsilon derivatives}, the second term in equation \eqref{contracting T_ k+1} is $\cD_{k+1,k}^*$, and is therefore determined. Hence, if $m \neq 0$, we can determine $\d_n^k H$ at $x_0$. Since we can do this for any $x_0 \in \sU$, it follows that we can determine $\d_n^k H$ on the boundary. Along with $\cD_{k+1,k}^*$, this determines all quantities in $\cD_{k+1,k}$. We now want to go back to $\theta_{-(k+2)}^{\mathfrak{e}}$ to determine the $(k+1)$-th normal derivatives of the connection. From equation \eqref{theta_ -(k+2) even part} and Proposition \ref{prop: precise form at x_0}, we have
\begin{align}
    \theta^{\mathfrak{e}}_{-(k+2)} &= \frac{1}{2^{k+1} |\x|^{k+2}} \left( - \d_n^{k} \fF_A^{-+} + i g^{\a \b} \d_n^{k+1} A_{\a} \hat{\x}_{\b} \theta_0 \right) + \cD_{k+1,k} \nonumber \\
    &= -\frac{1}{2^{k+1} |\x|^{k+2}} \gamma^n \gamma^a \left( g^{\b \gamma} h^{\a}_{\ a} \d^{k+1}_n A_{\a} - g^{\a \b} h^{\gamma}_{\ a} \d_n^{k+1} A_{\a} \right) \hat{\x}_{\b \gamma} + \cD_{k+1,k}. \label{recovering d_n^(k+1) connection tensor}
\end{align}
We thus recover the symmetrized tensor in parentheses, which upon contracting with $g_{\b \gamma}$ yields
\begin{equation}
    (n-1) h^{\a}_{\ a} \d_n^{k+1}A_{\a} - h^{\a}_{\ a} \d_n^{k+1} A_{\a} = (n-2) h^{\a}_{\ a} \d_n^{k+1} A_{\a}.
\end{equation}
We can therefore recover $\d_n^{k+1}A_{\a}$ on $\sU$ when $n > 2$, and thus determine all quantities in $\cD_{k+1}$. All that remains is to determine $\d_n^{k+1} \Sigma_{\a \b}$ from the odd part of $\theta_{-2}$. To this end, we use equation \eqref{theta_-(k+2)} and Proposition \ref{prop: form of b_-(k+1) to highest order} to write down the odd part,
\begin{align}
    \theta_{-(k+2)}^{\mathfrak{o}} &= \frac{1}{|\x|} \left( \left( b_{-(k+1)}^{-+} \right)^{\mathfrak{o}} + \left( b_{-(k+1)}^{--} \right)^{\mathfrak{e}} \theta_0 \right) + \cD_{k+1} \nonumber \\
    &= \frac{1}{2^{k+1}|\x|^{k+2}} \left( i g^{\delta \b} \d_n^{k+1}\left( \kappa_A^{-+}\left( \d_{\delta} \right) \right) \hat{\x}_{\b} - \frac{1}{4} \d_n^{k+2}g^{\a \b} \hat{\x}_{\a \b} \theta_0 \right) + \cD_{k+1} \nonumber \\
    &= \frac{i}{2^{k+2}|\x|^{k+2}} \gamma^n \gamma^a \left( g^{\a \gamma} g^{\b \delta} \d_n^{k+1} \left( \omega^a_{\ n}\left( \d_{\delta} \right) \right) + \frac{1}{2} h^{\gamma}_{\ a} \d_n^{k+2}g^{\a \b} \right) \hat{\x}_{\a \b \gamma} + \cD_{k+1}. \label{theta_-(k+2) odd part}
\end{align}
Now, differentiating equation \eqref{rewrite connection in terms of christoffel} yields
\begin{equation}
    \d_n^{k+1}  \left( \omega^a_{\ n}\left( \d_{\delta} \right) \right) = \frac{1}{2} h^{\mu}_{\ a} \d_n^{k+2} g_{\mu \delta} + \cD_{k+1},
\end{equation}
whereupon equation \eqref{theta_-(k+2) odd part} becomes
\begin{align}
    \theta_{-(k+2)}^{\mathfrak{o}} &= \frac{i}{2^{k+3}|\x|^{k+2}} \gamma^n \gamma^a \left( g^{\a \gamma} g^{\b \delta} h^{\mu}_{\ a} \d_n^{k+2} g_{\mu \delta}  + h^{\gamma}_{\ a} \d_n^{k+2}g^{\a \b} \right) \hat{\x}_{\a \b \gamma} + \cD_{k+1}. \label{theta_-(k+2) odd part v2}
\end{align}
We are thus able to recover the symmetric tensor
\begin{equation}
    \left( \cS_{k+2} \right)^{\a \b \gamma}_{a} := g^{(\a \gamma} g^{\b) \delta} h^{\mu}_{\ a} \d_n^{k+2} g_{\mu \delta}  +  \d_n^{k+2}g^{( \a \b} h^{\gamma)}_{\ a}. 
\end{equation}
Proceeding as in equation \eqref{obtaining derivative of SFF}, we compute
\begin{align}
    \frac{3}{2(n-1)}g_{\a \b} g_{\gamma \sigma} \left( h^{-1} \right)^a_{\ \lambda} \left( \cS_{k+2} \right)^{\a \b \gamma}_{a} &:= g_{\lambda \sigma} g_{\a \b} \d_n^{k+2} g^{\a \b} + \frac{1}{2} \d_n^{k+2} g_{\lambda \sigma} \nonumber \\
    &=  g_{\lambda \sigma} \d^{k+1}_n H + \frac{1}{2} \d_n^{k+2} g_{\lambda \sigma} + \cD_{k+1} \nonumber \\
    &= \d^{k+1}_n \left( H g_{\lambda \sigma} + \frac{1}{2} \d_n g_{\lambda \sigma} \right) + \cD_{k+1} \nonumber \\
    &= - \d^{k+1}_n \Sigma_{\lambda \sigma} + \cD_{k+1}.
\end{align}
This completes the inductive step, and thus also the proof of Theorem \ref{main theorem: boundary determination, intro} for $n \geq 3$. \hfill \qed

\subsection{Proof of Theorem \ref{main theorem: boundary determination, intro} for $n = 2$}\label{subsec: proof for n=2}

It is clear that the above proof of Theorem \ref{main theorem: boundary determination, intro} for $n \geq 3$ does not carry over to $n = 2$ for several reasons. First of all, at each step in Section \ref{subsec: symbol of Theta}, we could only determine the highest order derivative of the connection $A$ with a coefficient of $n - 2$. Indeed, one can check that the expression in \eqref{recovering d_n^(k+1) connection tensor} vanishes when $n = 2$. This suggests that neither $A$ nor its derivatives can be determined on the boundary from the symbol of $\Theta_{g,A,m,\sU}$ when $n = 2$. We shall therefore deal only with the case in which the auxiliary bundle $E$ and connection $A$ are not present.

Furthermore, an important tool in the proof above for $n \geq 3$ is the discarding of many terms in $\slashed{\cD}_{k+1,k}$ by taking a trace and using Lemma \ref{lemma: trace of gamma matrices}. However, as mentioned before, when $n = 2$, the gamma matrix $\gamma^1$ along the boundary is not traceless on $\bV^+$. The proof above, however, could be replicated for $n = 2$ if we had a way to eliminate terms in $\slashed{\cD}_{k+1,k}$. This can be done as follows. First note that $\gamma^1$ preserves $\bV^+$, so that $\bV^+$ is a representation of $\Cl_1 \cong \C$ at each point, and is therefore decomposable into $1$-dimensional irreducible representations. This can be done by diagonalizing $\gamma^1$ over $\sU$ and restricting to the space spanned by a single eigenvector. On this space, endomorphisms are canonically identified with $\C$, and since $(\gamma^1)^2 = -1$, we have that $\gamma^1$ acts as $\pm i$. Hence, we can eliminate terms in $\slashed{\cD}_{k+1,k}$ by taking real parts, provided we keep careful track of the real and imaginary parts of the coefficients appearing in the symbol $\theta_{-(k+2)}$. 

Finally, we shall deal only with the case $m \neq 0$. Indeed, as above, to deal with $m = 0$ we would have to assume that the mean curvature and all of its normal derivatives are prescribed on the boundary, but when the boundary is $1$-dimensional, this is equivalent to prescribing the normal derivatives of the metric. In fact, one can see from the definition that for $n = 2$, the traceless part of the second fundamental form vanishes identically.

So, let us indicate how the proof of Theorem \ref{main theorem: boundary determination, intro} given in Section \ref{subsec: symbol of Theta} for $n \geq 3$ can be modified in the case $n = 2$ with additional hypotheses indicated above. First, we must return to Section \ref{subsec: symbol of Lambda} and consider the symbol of $\Lambda_{g,m,\sU}$ for $n = 2$. Of course $b_1$ is unchanged, while $b_0$ is still given by \eqref{b_0 -- even}--\eqref{b_0 -+ even}, except for the fact that the even part of $b_0^{\pm \pm}$ vanishes, as can be readily checked from equation \eqref{b_0 -- even}, and we also have $\kappa_A^{\pm \pm} = 0$. This has several consequences. First of all, $b_0^{\pm \pm}$ is always a multiple of the identity. Secondly, when one computes $b_{-1}$ using \eqref{b_-1 --}--\eqref{b_-1 -+}, one finds that
\begin{equation}
    \left( b_{-1}^{--} \right)^{\mathfrak{o}} = \frac{i \hat{\x}}{|\x|} \cD_1^{\bR} + \cD_0, \ \ \ \ \left( b_{-1}^{-+} \right)^{\mathfrak{e}} = \frac{\hat{\x}^2}{|\x|} \cD_{1}^{\bR} \cdot \gamma^n \gamma^1 + \cD_0,
\end{equation}
where the symbol $\cD_{k}^{\bR}$ denotes any {\em real}-valued scalar function depending on at most $k$ normal derivatives of the metric on the boundary. On the other hand, since the even part of $b_0^{\pm \pm}$ vanishes, one finds from equations \eqref{scalar curvature}--\eqref{b_-1 -- highest order} that
\begin{align}
    \left( b_{-1}^{--} \right)^{\mathfrak{e}} &= -\frac{1}{4 |\x|} \d_n H + \cD_1.
\end{align}

Using the above forms for $b_0$ and $b_{-1}$, the following lemma can be easily be proved by induction using equation \eqref{b -(k+1) equation}, keeping careful track of imaginary units and noting that $i$ only appears when paired with $\x$, while all other coefficients are real.
\begin{lemma}\label{lemma: n=2 form of b}
    Let $n=2$. For $k \geq 0$, the symbols $b_{-(k+1)}$ have the following form,
    \begin{align}
        \left( b^{--}_{-(k+1)} \right)^{\mathfrak{o}} &= \frac{i \hat{\x}}{|\x|^{k+1}} \cD_{k+1}^{\bR} + \cD_{k}, \\
        \left( b^{-+}_{-(k+1)} \right)^{\mathfrak{e}} &= \frac{\hat{\x}^2}{|\x|^{k+1}} \cD_{k+1}^{\bR} \gamma^n \gamma^1 + \cD_{k}, \\
        \left( b^{--}_{-(k+1)} \right)^{\mathfrak{e}} &= -\frac{1}{2^{k+2}|\x|^{k+1}} \d^{k+1}_n H + \cD_{k+1}.
    \end{align}
\end{lemma}
Armed with this analogue of Propositions \ref{prop: form of b_-(k+1) to highest order} and \ref{prop: precise form of b -(k+1)}, the computation of the symbol of $\Theta_{g,m,\sU}$ from \eqref{local full symbol relation} proceeds as in Section \ref{subsec: symbol of Theta}. We determine the action of $\gamma^1$ on $\bV^+$ from $\theta_0$. Thus, we can choose an eigenvector $v^+_1$ of $\gamma^1$ and let $\bV^+_1$ be the span of $v^+_1$. Computing $\theta_{-1}$, we still have \eqref{- gamma^n theta_ -1}, but with $\Omega$, $A$ and $\kappa_A^{--}$ equal to $0$. In particular, since all the quantities there are real, under the canonical identification of $\End{\bV_1^+}$ with $\bC$, we have
\begin{equation}
    \Re\left( \gamma^n \theta^{\mathfrak{e}}_{-1} \big|_{\bV^+_1} \right) = -\frac{m}{|\x|}. \label{real part of theta_-1 is m}
\end{equation}
Thus, since $m \neq 0$, we can determine $g|_{\sU}$ and $\cD_0$. Note also that the odd part of $\theta_{-1}$ vanishes when $n=2$ (which is consistent with the vanishing of the traceless second fundamental form). 

Now, we compute $\theta^{\mathfrak{e}}_{-2}$ using \eqref{theta_-2 eqn}. Since $\left( b_0^{--} \right)^{\mathfrak{e}} = 0$, Lemma \ref{lemma: n=2 form of b} gives us
\begin{align}
    \gamma^n \theta^{\mathfrak{e}}_{-2} &= \frac{\gamma^n}{|\x|}  \left( \left( b^{-+}_{-1} \right)^{\mathfrak{e}} + \left( b_{-1}^{--} \right)^{\mathfrak{o}} \theta_0 - \frac{1}{2} H \theta_{-1}^{\mathfrak{e}} \right) + \cD_0, \nonumber \\
    &= \frac{\hat{\x}^2}{|\x|^2} \cD^{\bR}_1 \cdot \gamma^1 - \frac{1}{2|\x|} H \gamma^n \theta_{-1}^{\mathfrak{e}} + \cD_0,
\end{align}
So that we have, upon restricting to $\bV^+_1$, taking real parts, and using \eqref{real part of theta_-1 is m}, we get
\begin{equation}
    \Re\left( \gamma^n \theta_{-2}^{\mathfrak{e}} \big|_{\bV^+_1} \right) =  \frac{m}{2|\x|^2} H + \cD_0.
\end{equation}
We thus determine $H|_{\sU}$, which determines $\d_n g |_{\sU}$, since the boundary is $1$-dimensional.

Moving forward with the inductive step as in Section \ref{subsec: symbol of Theta}, let us assume that we know $\cD_k$. Then we find from equation \eqref{theta_ -(k+2) even part} that
\begin{align}
    \theta^{\mathfrak{e}}_{-(k+2)} &= \frac{1}{|\x|} \bigg( \left( b^{-+}_{-(k+1)} \right)^{\mathfrak{e}} + \left( b^{--}_{-(k+1)} \right)^{\mathfrak{o}} \theta_0 + \left( b_{-k}^{--} \right)^{\mathfrak{e}} \theta_{-1}^{\mathfrak{e}} -i \d_{\x} \left( b_{-k}^{--} \right)^{\mathfrak{e}} \d_{x} \theta_0 \bigg) + \cD_k,
\end{align}
which, by Lemma \ref{lemma: n=2 form of b}, yields
\begin{equation}
    \gamma^n  \theta^{\mathfrak{e}}_{-(k+2)} = \frac{\hat{\x}^2}{|\x|^{k+2}} \cD_{k+1}^{\bR} \gamma^1  - \frac{1}{2^{k+1}|\x|^{k+1}} \big(\d_n^{k} H \big) \gamma^n \theta_{-1}^{\mathfrak{e}} + \cD_k.
\end{equation}
Thus, by restricting to $\bV^+_1$ and taking real parts, we get
\begin{equation}
    \Re\left( \gamma^n \theta_{-(k+2)}^{\mathfrak{e}} \Big|_{\bV^+_1} \right) =\frac{m}{2^{k+1} |\x|^{k+2}} \big( \d_n^{k} H \big) + \cD_k,
\end{equation}
and so we can determine $\cD_{k+1}$. This completes the proof of Theorem \ref{main theorem: boundary determination, intro} for $n = 2$. \hfill \qed

\section{Recovering analytic data from $\Theta_{g,A,m, \sU}$}\label{sec: recovery}

In this section, we shall prove Theorem \ref{main theorem: reconstruction, intro} by adapting the methods of \cite{LTU2003, KLU2011, Gabdurakhmanov2025, EV2024}, which use the Green's kernels, or a suitable substitute in the case of \cite{EV2024}, to embed the two manifolds and their corresponding vector bundles into a common Sobolev space. To this end, in Section \ref{subsec: Greens kernels}, we first introduce the Green's kernel for the Dirac operator with chiral boundary conditions and review some its fundamental properties. Then, in Section \ref{subsec: reconstruction}, we shall use these Green's kernels to embed our twisted spinor bundles into a suitable Sobolev space, and then construct the required isomorphisms and isometries from this embedding.


\subsection{Green's kernels for the Dirac operator}\label{subsec: Greens kernels}

In order to prove Theorem \ref{main theorem: reconstruction, intro}, we first introduce the Green's kernel $G(x,y)$ for the Dirac operator $D_A - m$ over a spin manifold with boundary with chiral boundary conditions. Green's kernels for Dirac operators on manifolds with boundary have been studied before in \cite{Raulot2011, Calderbank1996thesis}, and some of the results we recall below are straightforward generalization of those found therein.

Let $M$ be a spin manifold with boundary, and let $\cE$ be a twisted spinor bundle over $M$. Let $\cE \boxtimes \cE^*$ denote the vector bundle over $M \times M$ whose fiber over $(x,y)$ is $\cE_x \otimes \cE^*_y = \Hom(\cE_y, \cE_x)$. Then, following \cite{Raulot2011}, for $m \notin \spec^+(D_A)$, we define the Green's kernel $G(x,y)$ for $D_A - m$ with chiral boundary conditions to be a distributional section of $\cE \boxtimes \cE^*$ satisfying
\begin{equation}\label{Greens kernel definition 1}
    \begin{cases}
        (D_A - m) G(x,y) = \delta_y \id_{\cE_y}, \\
        \bB^+ G(x,y) = 0.
    \end{cases}
\end{equation}
where the operators $D_A - m$ and $\bB^+$ act on the first factor $\cE_x$. Note that we have a Green's identity for the operator $D_A - m$ (see \cite[\S II.5]{LawsonMichelsohn} for example),
\begin{equation}\label{Green's identity for D_A}
    \int_M \< \left( D_A - m \right) \psi_1, \psi_2 \> - \<  \psi_1, \left( D_A - m \right) \psi_2 \> = -\int_{\d M} \< \gamma^n \psi_1, \psi_2 \>,
\end{equation}
for all smooth sections $\psi_1, \psi_2$ of $\cE$. Therefore, \eqref{Greens kernel definition 1} is equivalent to requiring
\begin{equation}\label{Greens kernel definition 2}
    \int_M \< G(x,y) \psi_1^{+}(y), \left( D_A - m \right) \psi_2^{+}(x) \> \, dx = \< \psi_1^+(y), \psi^+_2(y) \>
\end{equation}
for all smooth sections $\psi^{+}_1, \psi^{+}_2$ of $\cE$ such that $\bB^+ \psi^+_i = 0$ (c.f. \cite{Raulot2011}).

One can construct the Green's kernel $G(x,y)$ in several ways. One such way, assuming that $m^2$ is not in the Dirichlet spectrum of $D_A^2$, is to start with the Dirichlet Green's kernel $\hat{G}(x,y)$ for the second-order operator $D_A^2 - m^2$, which by the Lichnerowicz--Weitzenb\"ock formula \eqref{Lichnerowicz--Weitzenbock identity} is simply a Schr\"odinger operator for the connection Laplacian. The Green's kernel $\hat{G}(x,y)$ can be constructed via the parametrix method in \cite[\S 4.2]{Aubin1998}, for example, or one can appeal to the existence of a local fundamental solution \cite[\S XVII]{HormanderIII} and then globalize the solution by solving the Dirichlet problem with appropriate boundary data; see, for instance, \cite{Gabdurakhmanov2025} or \cite{Gabdurakhmanov2023thesis} for a more detailed exposition. Either way, once one has $\hat{G}(x,y)$, one obtains the chiral Green's kernel $G(x,y)$ by
\begin{equation}\label{Greens kernel construction 1}
    G(x,y) = (D_A + m) \hat{G}(x,y) - F(x,y),
\end{equation}
where $F(x,y)$ is the unique smooth solution to the chiral boundary value problem \eqref{chiral BVP} with
\begin{equation}
    \bB^+ F(x,y) = \bB^+ \big( \left( D_A + m \right) \hat{G}(x,y) \big).
\end{equation}
One then has that $G(x,y)$ defined by \eqref{Greens kernel construction 1} satisfies \eqref{Greens kernel definition 1}. 

Since $A$ is unitary and $m$ is real, the operator $D_A - m$ with chiral boundary condition $\bB^+$ is self-adjoint \cite{Raulot2011}. As a consequence, we have the following
\begin{proposition}\label{prop: Greens kernel is symmetric}
    The Green's kernel $G(x,y)$ is symmetric in the sense that we have
    \begin{equation}\label{Greens kernel is symmetric}
        G(x,y)^* = G(y,x)
    \end{equation}
    as elements of $\Hom(\cE_x, \cE_y)$ for all $x \neq y$.
\end{proposition}
The proof of Proposition \ref{prop: Greens kernel is symmetric} is identical to that of \cite[Prop. 4]{Raulot2011}. Moreover, by classical results of elliptic regularity \cite{MN1957, Petrowsky1939}, we have that $G(x,y)$ is analytic in $x$ for $x \neq y$, since the metric and connection are analytic. Proposition \ref{prop: Greens kernel is symmetric} implies that $G(x,y)$ is also analytic in $y$ when $y \neq x$. Moreover, recall that $\delta_y \in H^s$ for all $s < -\frac{n}{2}$. It follows from \eqref{Greens kernel definition 1} and standard elliptic theory that $G(x,y)$ is in $H^{s+1}$ for all such $s$. In fact we can say more: since $\d_y \delta_y$ is in $H^{s-1}$ for all such $s$, we have that $\d_y G(x,y)$ is in $H^{s}$. One can easily show that this map is continuous by the same arguments used in \cite{KLU2011} or \cite{Gabdurakhmanov2025}, for example. Summarizing, we have

\begin{proposition}\label{prop: Greens kernel C1 into Sobolev space}
    The map $y \mapsto G(\cdot, y)$ is a $C^1$-map into $H^s(M, \cE \boxtimes \cE^*)$ for $s < -\frac{n}{2}$.
\end{proposition}

By Proposition \ref{prop: Greens kernel is symmetric}, the same is true for the map $x \mapsto G(x, \cdot)$. Finally, we want to investigate the asymptotics of the Green's kernel as $x \to y$ for an interior point $y$. For this, note that it suffices to consider a local fundamental solution for the operator $D_A - m$ in a local trivialization near $y$, since by elliptic regularity the difference between such a fundamental solution and the Green's kernel must be smooth. A study of fundamental solutions for Dirac-type systems with $0$-th order terms has been carried out in \cite{DinhSon2012}, an implication of which is that there is $\lambda > 0$ such that
\begin{equation}\label{asymptotics of Greens kernel}
    G(x,y) =  \frac{c_n \gamma(e_r)}{r(x,y)^{n-1}} + O\left( r(x,y)^{1-n + \lambda} \right),
\end{equation}
\begin{equation}\label{asymptotics of Greens kernel derivatives}
    \nabla_x G(x,y) =  \frac{c'_n \gamma(e_r)}{r(x,y)^{n}} + O\left( r(x,y)^{-n + \lambda} \right),
\end{equation}
as $x \to y$, where $r(x,y)$ is the geodesic distance, $e_r$ is the radial unit vector in normal coordinates about $y$, and $c_n, c_n'$ are constants. Indeed, as shown in \cite{DinhSon2012}, the leading order term of $G(x,y)$ can be found by applying $D_A$ to the leading order term of the fundamental solution for $D_A^2$. Note that by the symmetry of the Green's kernel, $\nabla_y G(x,y)$ also has the asymptotics \eqref{asymptotics of Greens kernel derivatives}.




\subsection{Proof of Theorem \ref{main theorem: reconstruction, intro}}\label{subsec: reconstruction}

Let us henceforward assume the hypotheses of Theorem \ref{main theorem: reconstruction, intro}. We thus have an identification of the boundaries of $M_1$ and $M_2$ over the open set $\sU$, and a unitary isomorphism $\chi : E_1|_{\sU} \to E_2|_{\sU}$ of the auxiliary bundles. Moreover, the boundary conjugation maps $\Theta_{g_1, A_1,m,\sU}$ and $\Theta_{g_2, A_2, m, \sU}$ are equivalent over $\sU$ in the sense of Definition \ref{definition of equivalent Theta maps}. The first thing we do is use the assumption of real-analyticity and Theorem \ref{main theorem: boundary determination, intro} to extend the data on $M_i$ to a larger real-analytic manifold $\tilde{M}_i$.

To this end, we define the manifold $\tilde{M}_i$ by first choosing boundary normal coordinates about points in $\sU$, and then attaching a small ball of radius $\varepsilon$, so that points in the lower half-space are identified with points in $M_i$. Let $U := \tilde{M}_i \setm M_i$. Note that we can choose $\varepsilon$ sufficiently small so that $U$ does not depend on $M_i$. That is, we can assume that $\tilde{M}_1 \setm M_1$ is identified with $\tilde{M}_2 \setm M_2$. It is easy to see that by shrinking $U$ if necessary, we can assume that the real-analytic metrics $g_i$ extend to metrics $\tilde{g}_i$ on $\tilde{M}_i$. Moreover, the bundles $E_i$ also extend to Hermitian bundles $\tilde{E}_i$, which are trivial over $U$. The connections $A_i$ also extend to $U$ in an analytic frame for $\tilde{E}_i$. Note that the spinor bundle $\bS_i$ extends to the corresponding spinor bundle $\tilde{\bS}_i$ of $\tilde{g}_i$, which can also be assumed to be trivial on $U$. The chirality operator $\Pi$ of course extends as well. We also note that since finite systems of eigenvalues are continuous (see \cite[\S 17]{BW1993} or \cite[\S IV 3.5]{Kato1980}, for example), we can choose $\tilde{M}_i$ so that $m$ does not belong to the chiral spectrum of $\tilde{D}_A$ over $\tilde{M}_i$.

Now, by Theorem \ref{main theorem: boundary determination, intro}, the metrics $g_1$ and $g_2$ have the same Taylor series over $\sU$. Therefore, the extended metrics $\tilde{g}_1$ and $\tilde{g}_2$ agree in $U$. It follows that we can identify the spinor bundles $\bS_1$ and $\bS_2$ over $U$ by an isomorphism $j_S$ induced by an isometry $j : (U, \tilde{g}_1) \to (U, \tilde{g}_2)$. Also, the isomorphism $\chi : E_1|_{\sU} \to E_2|_{\sU}$ extends to a unitary isomorphism $j_E : \tilde{E}_1|_{U} \to \tilde{E}_2|_{U}$ defined by identifying $E_1$ and $E_2$ in boundary normal frames. By Theorem \ref{main theorem: boundary determination, intro}, the connection $1$-forms $A_1$ and $A_2$ have equal Taylor series with respect to these frames, and thus by analytic continuation, they are equal in $U$. In other words, we have that $j_E^* A_2 = A_1$ over $U$. In the following, we shall let $\cE_i := \bS_i \otimes E_i$ and similarly for $\tilde{\cE}_i$. Since $\tilde{\cE}_1|_{U}$ and $\tilde{\cE}_2|_{U}$ are identified, we shall sometimes denote both as $\tilde{\cE}|_{U}$.


Now, having extended our data to $U$, we have the Green's kernel $\tilde{G}_i(x,y)$ for $\tilde{D}_A - m$ with chiral boundary conditions over $\tilde{M}_i$, as defined in Section \ref{subsec: Greens kernels}. The first step is to show that the boundary conjugation maps determine the Green's kernels in $U$. Thus we prove the following

\begin{proposition}\label{prop: Greens functions agree in U}
    We have $\tilde{G}_1(x,y) = \tilde{G}_2(x,y)$ in $\Hom(\cE_y, \cE_x)$ for all $x,y \in U$ with $x \neq y$.
\end{proposition}

\begin{proof}
    For $y \in U$ and $\psi \in \tilde{\cE}_y$, we shall define a section $\tilde{\Gamma}_{y, \psi}$ of $\tilde{\cE}_2$ over $\tilde{M}_2$ that restricts to $\tilde{G}_1(x,y) \psi$ on $U$, and then show that it must coincide with $\tilde{G}_2(x,y) \psi$ on $\tilde{M}_2$. Thus, we define $\Gamma_{y, \psi}$ to be the unique solution to the following boundary value problem on $M_2$,
    \begin{equation}
        \begin{cases}
            (D_{g_2, A_2} - m) \Gamma_{y, \psi} = 0 \\ 
            \bB^+ \Gamma_{y, \psi}|_{\sU} = \bB^+ \tilde{G}_1(x,y)\psi, \\
            \bB^+ \Gamma_{y, \psi}|_{\d M_2 \setm \sU} = 0.
        \end{cases}
    \end{equation}
On the other hand, since $y \in U$, we have that $\tilde{G}_1(x,y) \psi$ satisfies
\begin{equation}
        \begin{cases}
            (D_{g_1, A_1} - m) \tilde{G}_1(x,y) \psi = 0 \\ 
            \bB^+ \tilde{G}_1(x,y) \psi |_{\sU} = \bB^+ \tilde{G}_1(x,y) \psi, \\
            \bB^+ \Gamma_{y, \psi}|_{\d M_1 \setm \sU} = 0,
        \end{cases}
    \end{equation}
on $M_1$ (note that $\d M_1 \setm \sU \subseteq \d \tilde{M}_1$). Therefore, since the boundary conjugation maps $\Theta_{g_1, A_1, m, \sU}$ and $\Theta_{g_2, A_2, m, \sU}$ are equivalent, upon making the above identifications of our bundles, we have
\begin{equation}
    \Theta_{g_2, A_2, m, \sU}\big(\, \bB^+ \Gamma_{y, \psi}\big) = \Theta_{g_1, A_1, m, \sU}\big(\, \bB^{+} \tilde{G}_1(x,y) \psi \big) = \bB^{-} \tilde{G}_1(x,y) \psi.
\end{equation}
Therefore, $\Gamma_{y, \psi}|_{\sU} = \tilde{G}_1(x,y)\psi$. So we can define the following section of $\tilde{\cE}_2$ over $\tilde{M}_2$,
\begin{equation}
    \tilde{\Gamma}_{y,\psi} = \begin{cases}
        \tilde{G}_1(x,y)\psi, & x \in U, \\
        \Gamma_{y,\psi}(x), & x \in M_2.
    \end{cases}
\end{equation}
Note that by construction, $\tilde{\Gamma}_{y,\psi}$ is a distributional solution of 
\begin{equation}
    \begin{cases}
        (D_{\tilde{g}_2, \tilde{A}_2} - m) \tilde{\Gamma}_{y, \psi} = \psi\, \delta_y, \\
        \bB^+  \tilde{\Gamma}_{y,\psi} |_{\d \tilde{M}_2} = 0.
    \end{cases}
\end{equation}
Therefore, by uniqueness, we must have $\tilde{\Gamma}_{y, \psi} = \tilde{G}_2(x,y) \psi$. Since this is true for all $\psi \in \tilde{\cE}_y$, we have $\tilde{G}_1(x,y) = \tilde{G}_2(x,y)$ for $x, y \in U$ with $x \neq y$. This completes the proof.
\end{proof}

Next, we define the following maps, which we shall show are embeddings of the twisted spinor bundles into a common Sobolev space. So, for $i = 1,2$ and $s < -\frac{n}{2}$, we define
\begin{align}\label{definition of embedding maps}
\begin{split}
    \sG_i : \tilde{\cE}_i &\to H^s(U, \tilde{\cE}|_U), \\
    \psi_y &\mapsto \tilde{G}_i(\cdot,y) \psi_y,
\end{split}
\end{align}
where $\psi_y$ is in the fiber of $\tilde{\cE}_i$ over $y \in \tilde{M}_i$. Note that by Proposition \ref{prop: Greens kernel C1 into Sobolev space}, the maps $\sG_i$ are $C^1$-smooth. Moreover, Proposition \ref{prop: Greens functions agree in U} implies that $\sG_1$ and $\sG_2$ agree when restricted to $U$. More precisely,
\begin{equation}\label{G_i maps intertwine}
    j_{\cE} \circ \sG_1|_U = \sG_2|_U \circ j_{\cE}, 
\end{equation}
where $j_{\cE} : \tilde{\cE}_1|_{U} \to \tilde{\cE}_2|_{U}$ is the isomorphism $j_S \otimes j_E$, described at the beginning of this section. 

We now want to prove the following theorem, which is an analogue of similar theorems in \cite{LTU2003, KLU2011, Gabdurakhmanov2025, EV2024}, and whose proof proceeds along the same lines. In fact, our proof follows closely that of \cite[Theorem 3.4]{Gabdurakhmanov2025}, which is better suited for vector bundles. Below, we shall let $\tilde{\cE}^0_i$ denote the vector bundle $\tilde{\cE}_i$ with the zero section removed.

\begin{theorem}\label{theorem: constructing Phi from the image of embeddings}
    Let $\sG_i$ be as above. Then $\sG_2(\tilde{\cE}^0_2)$ coincides with $j_{\cE} \circ \sG_1(\tilde{\cE}_1^0)$ in $H^s(U, \tilde{\cE}|_U)$ and the map $\sG_2^{-1} \circ j_{\cE} \circ \sG_1 : \tilde{\cE}_1^0 \to \tilde{\cE}_2^0$ extends to a real-analytic unitary isomorphism $\Phi : \tilde{\cE}_1 \to \tilde{\cE}_2$ projecting to an isometry $\varphi : (\tilde{M}_1, \tilde{g}_1) \to (\tilde{M}_2, \tilde{g}_2)$, such that $\vp|_U = j$ and $\Phi|_{U} = j_{\cE}$. Moreover, $\Phi$ intertwines the twisted spin connections, $\Phi^*\nabla^{A_2} = \nabla^{A_1}$.
\end{theorem}

\begin{proof}
    The key ingredients in the proof of Theorem \ref{theorem: constructing Phi from the image of embeddings} are the analyticity of the Green's kernels, the asymptotics \eqref{asymptotics of Greens kernel}, and the intertwining property \eqref{G_i maps intertwine}. First we shall show that the map $\sG_i$ is a linear embedding of each fiber $(\tilde{\cE_i})_x$ into $H^s(U, \tilde{\cE}|_U)$ when $x \notin \d \tilde{M}$, and an injective immersion on the set $\tilde{\cE}_i^0$. Indeed, note that if $\sG_i(\psi_y) = 0$ in $H^s(U, \tilde{\cE}|U)$, then we have $\tilde{G}_i(x,y)\psi_y = 0$ for all $x \in U$ with $x \neq y$. Thus, by analytic continuation, this must hold for all $x \in \tilde{M}_i$ with $x \neq y$. This, however, contradicts the asymptotics of the Green's kernel \eqref{asymptotics of Greens kernel}. Therefore, it follows that $\sG_i$ must be a linear embedding on each fiber away from the boundary, as claimed.

    We now show that $\sG_i$ is an injective immersion on $\tilde{\cE}^0_i$. Indeed, injectivity of $\sG_i$ follows since if $\sG_i(\psi_y) = \sG_i(\psi'_{y'})$, then $\tilde{G}_i(x,y)\psi_y = \tilde{G}_i(x,y')\psi'_{y'}$ for all $x \in U$ not equal to $y$ or $y'$. Since neither side is identically zero, by analyticity, this must hold for all $x \in \tilde{M}_i \setm \{ y, y' \}$. But since $\tilde{G}_i(x,y)$ is smooth for all $x \neq y$, it follows from the asymptotics \eqref{asymptotics of Greens kernel} that we must have $y = y'$. Since $\sG_i$ is injective on each fiber, it then follows that $\psi'_{y'} = \psi_y$. So $\sG_i$ is indeed injective on $\tilde{\cE}^0_i$. It remains to show that $\sG_i$ is an immersion on $\tilde{\cE}^0_i$. For this, we need to show that if $(D\sG_i)_{(y,\psi)}(v, \phi) = 0$ for $v$ in $T_y \tilde{M}_i$ and $\phi \in T_{\psi}\tilde{\cE}^0_i$, then $v = 0$ and $\phi = 0$. Note that for $x \neq y$,
    \begin{equation}\label{immersion equation}
        (D\sG_i)_{(y,\psi)}(v, \phi)\big|_x = v \cdot \nabla_{y} \tilde{G}_i(x,y) \psi  + \tilde{G}_i(x,y) \phi.
    \end{equation}
    If $ (D\sG_i)_{(y,\psi)}(v, \phi) = 0$, then the right-hand-side of \eqref{immersion equation} vanishes for all $x \in U$ with $x \neq y$. So we have again by analyticity that the right-hand-side vanishes for all $x \in \tilde{M}_i$ not equal to $y$. We deduce once again from the asymptotics \eqref{asymptotics of Greens kernel}--\eqref{asymptotics of Greens kernel derivatives} that we must have $v = 0$ and $\phi = 0$. Therefore, $D \sG_i$ is injective at every point in $\tilde{\cE}^0_i$, and so $\sG_i$ is indeed an immersion on $\tilde{\cE}_i^0$.

    We now want to show that $\sG_2(\tilde{\cE}_i^0)$ coincides with $j_{\cE} \circ \sG_1(\tilde{\cE}_i^0)$, and that $\sG_2^{-1} \circ j_{\cE} \circ \sG_1$ extends to a real-analytic isomorphism between $\tilde{\cE}_1$ and $\tilde{\cE}_2$. For this, let us fix $p_1 \in \sU$, and define $\Omega_1 \sub \tilde{M}_1$ to the largest connected open set containing $p_1$ such that for all $x \in \Omega_1$ there is a unique $\vp(x) \in \tilde{M}_2$ such that the images $j_{\cE} \circ \sG_1((\tilde{\cE_1})_x)$ and $\sG_2((\tilde{\cE}_2)_{\vp(x)})$ coincide, and
    \begin{equation}\label{Phi at x}
        \Phi_x = \sG_2^{-1} \circ j_{\cE} \circ \sG_1 : (\tilde{\cE}_1)_x \to (\tilde{\cE}_2)_{\vp(x)}
    \end{equation}
    is an isometry of the fibers with respect to their inner products. Note that equation \eqref{Phi at x} defines a fiber-preserving isometry $\Phi : \tilde{\cE}_1|_{\Omega_1} \to \tilde{\cE_2}$. We want to show that $\Omega_1 = \tilde{M}_1$. First, let us note that $\Omega_1$ is easily seen to contain $U$. Indeed, for $x \in U$, one can take as $\vp(x)$ the point $j(x) \in \tilde{M}_2$, where $j : (U, \tilde{g}_1) \to (U, \tilde{g}_2)$ is the isometry above. That \eqref{Phi at x} holds follows from \eqref{G_i maps intertwine}, which also implies that the map $\Phi$ defined above restricts to $j_{\cE}$ on $U$.

    Now let us show that $\Omega_1$ must be all of $\tilde{M}_1$. If this were not so, then there would exist $x_1 \in \d \Omega_1$ lying in the interior of $\tilde{M}_1$, so $x_1 \notin \tilde{M}_1$. We claim that $\Phi$ can be extended to the fiber $(\tilde{\cE}_1)_{x_1}$. For this, we need the following technical lemma, which is an analogue of similar results found in \cite{LTU2003, KLU2011, Gabdurakhmanov2025}. An important ingredient in the proofs of those results for second-order operators is the vanishing of the Green's kernel on the boundary. As our Green's kernel does not vanish on the boundary, we shall have to adapt the proof to the chiral boundary condition \eqref{Greens kernel definition 1}.

    \begin{lemma}\label{lemma: sequence lemma}
        Let $x_1 \in \d \Omega_1$ with $x_1 \notin \d \tilde{M}_1$, as above. Then there is a unique point $x_2$ in the interior of $\tilde{M}_2$ such that the images $j_{\cE} \circ \sG_1((\tilde{\cE}_1)_{x_1})$ and $\sG_2((\tilde{\cE}_2)_{x_2})$ coincide, and the map $\Phi_{x_1}$ is an isometry. Moreover, for any non-zero $v_{x_1} \in (\tilde{\cE}_1)_{x_1}$, there is a unique non-zero $w_{x_2} \in (\tilde{\cE}_2)_{x_2}$ such that
        \begin{equation}
            j_{\cE} \circ \sG_1(v_{x_1}) = \sG_2(w_{x_2}), \ \ \ \ \ |v_{x_1}|_{\cE_1} = |w_{x_2}|_{\cE_2},
        \end{equation}
        and for any sequence $v_{p_k} \to v_{x_1}$ with $p_k \to x_1$, we have $\Phi(v_{p_k}) \to w_{x_2}$.
    \end{lemma}

    \begin{proof}
        Let $p_k \in \Omega_1$ be a sequence such that $p_k \to x_1$, and let $q_k := \vp(p_k)$. We have by definition of $\Omega_1$ that the images $j_{\cE} \circ \sG_1((\tilde{\cE}_1)_{p_k})$ and $\sG_2((\tilde{\cE}_2)_{q_k})$ coincide, and the map $\Phi_{p_k}$ is an isometry of the fibers. Since $\tilde{M}_2$ is compact, by passing to a subsequence, we may assume that $q_k \to x_2$ for some $x_2 \in \tilde{M}_2$. We want to show that $x_2$ cannot be in $\d \tilde{M}_2$. So let us suppose that $x_2 \in \d \tilde{M}_2$. Then
        \begin{equation}
            j_{\cE} \circ ( \sG_1 )_{p_k} = (\sG_2)_{q_k} \circ \Phi_{p_k}
        \end{equation}
        on $(\tilde{\cE})_{p_k}$ for all $k$, which implies that
        \begin{equation}\label{sequence lemma Green kernel eqn}
            j_{\cE} \circ \tilde{G}_1(x, p_k) = \tilde{G}_2(x,q_k) \circ \Phi_{p_{k}}
        \end{equation}
        for all $x \in U$. In particular, since $j_{\cE}$ and $\Phi_{p_k}$ are unitary, by choosing orthonormal frames near $x_1$ and $x_2$ and taking determinants of \eqref{sequence lemma Green kernel eqn}, we obtain
        \begin{equation}\label{sequence lemma determinant eqn}
            |\det{\tilde{G}_1(x,p_k)}| = |\det{\tilde{G}_2(x,q_k)}| 
        \end{equation}
        for all $x \in U$. Taking the limit of \eqref{sequence lemma determinant eqn}, we obtain $|\det{\tilde{G}_1(x,x_1)}| = |\det{\tilde{G}_2(x,x_2)}|$ for all $x \in U$. On the other hand, note that since $\bB^+$ is the projection onto $\bV_E^+$, the chiral boundary conditions \eqref{Greens kernel definition 1} and Proposition \ref{prop: Greens kernel is symmetric} imply that $\tilde{G}_2(x,x_2)$ has non-trivial kernel for all $x$, when $x_2 \in \d \tilde{M}_2$. In particular, we have $\det{\tilde{G}_2(x,x_2)} = 0$, and thus also $\det{\tilde{G}_1(x,x_1)} = 0$ for all $x \in U$. By analyticity, we must therefore have $\det{\tilde{G}_1(x,x_1)} = 0$ for all $x \in \tilde{M}_1$ with $x \neq x_1$. But since $x_1$ is an interior point of $\tilde{M}_1$, the asymptotics \eqref{asymptotics of Greens kernel} imply that
        \begin{equation}
            |\det{\tilde{G}_1(x,x_1)}| = c_n^N r(x,x_1)^{-N(n-1)} + O\left( r(x,y)^{N(1-n + \lambda)} \right),
        \end{equation}
        where $N = \rank{\tilde{\cE}_1}$. This is a contradiction. Therefore, we must have $x_2 \notin \d \tilde{M}_2$.

        Now, to complete the proof of the lemma, let $v_{x_1}$ be a non-zero vector in $(\tilde{\cE}_1)_{x_1}$. We can choose a sequence $v_k \in (\tilde{\cE}_1)_{p_k}$ converging to $v_{x_1}$, and let $w_k := \Phi_{p_k}(v_k)$ so that $j_{\cE} \circ \sG_1(v_k) = \sG_2(w_k)$ and $|v_k|_{\cE_1} = |w_k|_{\cE_2}$. Since the sequence $w_k$ is bounded, we can pass to a subsequence and assume that $w_k \to w_{x_2}$ for some $w_{x_2}$ in $(\tilde{\cE}_2)_{x_2}$. We have $|w_{x_2}| = |v_{x_1}|$. Moreover, since $x_2$ is an interior point, the rest of the lemma follows from the continuity of $\sG_1$ and $\sG_2$, as well as the fact that $\sG_i$ is a linear embedding of the fibers away from the boundary.
    \end{proof}
    The rest of the proof of Theorem \ref{theorem: constructing Phi from the image of embeddings} proceeds as in \cite{KLU2011, EV2024, Gabdurakhmanov2025}, using Lemma \ref{lemma: sequence lemma} to show that a neighbourhood of $x_2$ must be contained in $\Omega_1$. We shall thus present the argument while omitting some of the details, which can be found in the aforementioned references.

    Let $\sR_i \sub H^s(U, \tilde{\cE}|_U)$ be the image of $\sG_i$. Let $v_{x_1} \in (\tilde{\cE}_1)_{x_1}$ be a non-zero vector, and let $x_2 \in \tilde{M_2}$ and $w_{x_2} \in (\tilde{\cE}_2)_{x_2}$ be the point and vector satisfying the conclusions of Lemma \ref{lemma: sequence lemma}. Then we have that $j_{\tilde{\cE}} \circ \sG_1(v_{x_1}) = \sG_2(w_{x_2})$, and we denote this common value by $u$. Since $\sG_i$ is an injective immersion, $j_{\cE}(\sR_1)$ and $\sR_2$ are submanifolds of $H^s(U, \tilde{\cE}|_U)$ whose tangent spaces are given by the images of the differentials $D(j_{\cE} \circ \sG_1)$ and $D\sG_2$, respectively. Lemma \ref{lemma: sequence lemma} then implies that $T_u j_{\cE}(\sR_1)$ and $T_u \sR_2$ coincide as subspaces in $H^s(U, \tilde{\cE}|_U)$; we shall denote this subspace by $\sV$. If $\Pi_{\sV}$ denotes the orthogonal projection of $H^s(U, \tilde{\cE}|_U)$ onto $\sV$, then $\Pi_{\sV} \circ \sG_2$ has an invertible differential near $w_{x_2}$. Therefore, by the inverse function theorem, there exists a $C^1$ inverse map $H_2 : \sO \to \tilde{\cE}_2$ defined in a neighbourhood $\sO$ of $\Pi_{\sV}(u)$. Then, near $u$ the image $\sR_2$ is the graph of the map 
    \begin{equation}\label{F_2 local graph definition}
        F_2 : \sO \to \sV^{\bot}, \ \ \ \ \ v \mapsto \sG_2(H_2(v)) - v.
    \end{equation}
    Similarly, there exists a $C^1$ map $H_1 : \sO \to \tilde{\cE}_1$ such that near $u$ the image $j_{\cE}(\sR_1)$ is the graph of a map $F_1$ defined in analogy with \eqref{F_2 local graph definition}. It is easy to see that the isomorphism $\Phi$ is equal to $H_2 \circ H_1^{-1}$ on the open subset $H_1(\sO) \cap \tilde{\cE}_1|_{\Omega_1}$. We have the following lemma, whose proof proceeds just as in \cite[Lemma 3.6, 3.7]{Gabdurakhmanov2025}, wherein the key ingredients are the real-analyticity of the Green's kernels and the fact that $x_1$ is not in $\bar{U}$, as observed above. 

    \begin{lemma}\label{graph functions are real-anal}
        The maps $H_i : \sO \to \tilde{\cE}_i$ are real-analytic in a neighbourhood of $\pi(u)$ in $\sV$. Moreover, the maps $F_i : \sO \to \sV^{\bot}$ coincide in a neighbourhood of $\Pi_{\sV}(u)$ in $\sV$.
    \end{lemma}

    In particular, it follows that there is a neighbourhood of $v_{x_1}$ in $\tilde{\cE}_1$ on which the map $H_2 \circ H_1^{-1}$ is real-analytic, while the equality of $F_1$ and $F_2$ implies that the images $j_{\cE}(\sR_1)$ and $\sR_2$ coincide near the point $u$. So far, all this is near the point $u = j_{\tilde{\cE}} \circ \sG_1(v_{x_1}) = \sG_2(w_{x_2})$. To complete the proof of Theorem \ref{theorem: constructing Phi from the image of embeddings}, we quote one more lemma \cite[Lemma 3.8]{Gabdurakhmanov2025}, whose proof is once again identical when applied to our setting. The main ingredients here are the conical structures of $j_{\cE}(\sR_1)$ and $\sR_2$, and the fact that if $\sG_i((\tilde{\cE}_i)_x)$ has non-trivial intersection with $\sG_i((\tilde{\cE}_i)_p) \oplus \sG_i((\tilde{\cE}_i)_q)$ in $H^s(U, \tilde{\cE}|_U)$, then it must follow that either $x = p$ or $x = q$; this last statement follows easily from the real-analyticity of the Green's kernel and the asymptotics \eqref{asymptotics of Greens kernel}.

    \begin{lemma}\label{lemma 3.8 gab}
        There is a neighbourhood $W_1$ of $x_1 \in \tilde{M}_1$ such that for all $x \in W_1$, there is $y \in \tilde{M}_2$ such that the images $j_{\cE} \circ \sG_1((\tilde{\cE}_1)_x)$ and $\sG_2((\tilde{\cE}_2)_y)$ coincide.
    \end{lemma}

    In particular, Lemma \ref{lemma 3.8 gab} implies that $\Phi_x := \sG_2^{-1} \circ j_{\cE} \circ \sG_1$ is a well-defined map from $(\tilde{\cE}_1)_x$ to $(\tilde{\cE}_2)_y$ for all $x \in W_1$, and by Lemma \ref{graph functions are real-anal}, $\Phi_x$ must coincide with $H_2 \circ H^{-1}$ on a neighbourhood $\sN$ of a given $v_{x_1} \in (\tilde{\cE}_1)_{x_1}$. Since the map $H_2 \circ H_1^{-1}$ is real-analytic and an isometry on $\sN \cap \tilde{\pi}_1^{-1}(\Omega_1)$, where $\tilde{\pi}_1 : \tilde{\cE}_1 \to \tilde{M}_1$ is the bundle projection, it follows by analytic continuation that $H_2 \circ H_1^{-1}$ is an isometry on $\sN$. As this holds for all points $v_{x_1}$ in the unit sphere of $(\tilde{\cE}_1)_{x_1}$, it follows that $\Phi_x$ is an isometry for al $x \in W_1$. We have thus showed that $W_1 \subseteq \Omega_1$. But this is a contradiction since we assumed that the interior point $x_1$ was in $\d \Omega_1$. Therefore, we must have that $\Omega_1$ is all of $\tilde{M}_1$.

    The rest of Theorem \ref{theorem: constructing Phi from the image of embeddings} now follows easily. Indeed, we have that $\Phi = \sG_2^{-1} \circ j_{\cE} \circ \sG_1$ is a well-defined fiber-preserving map $\tilde{\cE}_1 \to \tilde{\cE}_2$. By Lemma \ref{graph functions are real-anal}, $\Phi$ locally agrees with the maps $H_2 \circ H_1^{-1}$, and so is real-analytic. As $\Phi$ is a unitary isomorphism on each fiber, it is a real-analytic unitary isomorphism of bundles. In particular, it projects to a real-analytic map $\vp : \tilde{M}_1 \to \tilde{M}_2$. We have by construction that $\Phi|_U = j_{\cE}$ and $\vp|_U = j$. Since we have $\vp^* \tilde{g}_2 = \tilde{g}_1$ on $U$, it follows by analytic continuation that $\vp^* \tilde{g}_2 = \tilde{g}_1$ everywhere on $\tilde{M}_1$, and so $\vp$ is an isometry. Similarly, since $\Phi$ intertwines the twisted spin connections on $U$, it follows again by analytic continuation that $\Phi^*\nabla^{A_2} = \nabla^{A_1}$ everywhere. This completes the proof of Theorem \ref{theorem: constructing Phi from the image of embeddings}.
\end{proof}

We now complete the proof of Theorem \ref{main theorem: reconstruction, intro}. First, let us pull back all bundles over $\tilde{M}_2$ by the isometry $\vp$ so that $\Phi : \tilde{\cE}_1 \to \tilde{\cE}_2$ becomes a vertical bundle isomorphism projecting to the identity on $\tilde{M}_1$. The same is true for the isomorphisms $j_S : \tilde{\bS}_{1}|_U \to \tilde{\bS}_2|_{U}$ and $j_E : \tilde{E}_1|_{U} \to \tilde{E}_2|_{U}$. Note that the spinor bundle $\tilde{\bS}_2$ becomes a bundle of spinors over $\tilde{M}_1$, possibly associated to a different spin structure on $\tilde{M}_1$ than the one inducing $\tilde{\bS}_1$. In particular, there exists a complex line bundle $L$ on $\tilde{M}_1$, which is the complexification of a real line bundle, and an isomorphism $\sJ : \tilde{\bS}_1 \to \tilde{\bS}_2 \otimes L$ of Clifford module bundles; see \cite[\S 2.5]{Friedrich2000} for example. It is easy to see that $\sJ$ must be real-analytic. Composing $\Phi$ with $\sJ^{-1} \otimes \id_{E_1}$, we obtain a unitary isomorphism $\Phi' : \tilde{\bS}_2 \otimes L \otimes \tilde{E}_1 \to \tilde{\bS}_2 \otimes \tilde{E}_2$. We then define $\Phi_E : \tilde{E}_1 \otimes L \to \tilde{E}_2$ as the partial trace of $\Phi'$ over $\tilde{\bS}_2$,
\begin{equation}
    \Phi_E := \frac{1}{N_S} \Tr_{\bS_2} \Phi',
\end{equation}
where $N_S$ is the rank of $\tilde{\bS}_2$. Note that given the identification $\sJ : \tilde{\bS}_1 \to \tilde{\bS}_2 \otimes L$, the isomorphism $j_S : \tilde{\bS}_1|_U \to \tilde{\bS}_2|_U$ over $U$ induces an isomorphism $j_L : L|_U \to \C$. One can then check that
\begin{equation}
    \Phi_E|_U = j_E \otimes j_L,
\end{equation}
identifying $\tilde{E}_2$ with $\tilde{E}_2 \otimes \C$. Since $\Phi_E$ is a unitary isomorphism over $U$, the same is true over $\tilde{M}_1$ by real-analyticity. We now claim that there is a flat connection $\nabla^L$ on $L$ such that the pullback of $\nabla^{A_2}$ by $\Phi_E$ is the product connection $\nabla^{A_1} \otimes \nabla^{L}$. This shall then complete the proof of Theorem \ref{main theorem: reconstruction, intro}. Note that since $L$ is the complexification of a real line bundle, one can take local trivializations with locally constant transition functions. In particular, we can construct a real-analytic, even flat, unitary connection $\nabla^{L'}$ on $L$, and define $a \in \Omega^1(\tilde{M}_1, \End(\tilde{E}_1 \otimes L))$ by
\begin{equation}
    a := \nabla^{A_1} \otimes \nabla^{L'} - \Phi_E^*\nabla^{A_2}.
\end{equation}
Define $b \in \Omega^1(\tilde{M}_1, \C)$ as the partial trace of $a$ over $\tilde{E}_1$,
\begin{equation}
    b := \frac{1}{N_E} \Tr_{\tilde{E}_1} a.
\end{equation}
It is easy to see that $b \in \Omega^1(\tilde{M}_1, i\R)$. Then, using the fact that $j_E$ is a gauge equivalence of $\nabla^{A_1}$ and $\nabla^{A_2}$ over $U$, one can check that $b$ coincides with the connection $1$-form of $\nabla^{L'}$ in the trivialization $j_L$. It follows that the connection $\nabla^L$ on $L$ defined by $\nabla^L := \nabla^{L'} - b$ satisfies
\begin{equation}\label{pullback is product connection on U}
    \Phi_E^* \nabla^{A_2} \big|_U = \nabla^{A_1} \otimes \nabla^L \big|_U,
\end{equation}
which by analytic continuation continues to hold over all of $\tilde{M}_1$. By computing the curvatures of the connections on both sides of \eqref{pullback is product connection on U} and using the local gauge equivalence of $\nabla^{A_1}$ and $\nabla^{A_2}$ on $U$, one finds that the curvature of $\nabla^L$ vanishes on $U$. By analyticity, $\nabla^L$ is flat everywhere. This completes the proof of our main uniqueness result, Theorem \ref{main theorem: reconstruction, intro}.
\qed

\vspace{10pt}

As remarked in Section \ref{subsec: main results}, an immediate corollary of Theorem \ref{main theorem: boundary determination, intro} is that the connections $A_1$ and $A_2$ are locally gauge equivalent everywhere in $\tilde{M}_1$, and that the only obstruction to global gauge equivalence is this flat $2$-torsion line bundle $L$, which arises from the possibility that the spinor bundles $\bS_1$ and $\bS_2$ do not necessarily correspond to isomorphic spin structures. The main issue here is that it is not generally possible to globally `untwist' the twisted spinor bundles to obtain the two factors $\bS$ and $E$ independently, as the following example shows. 
\vspace{5pt}
\begin{example}
    Let $M \sub \R^{2d}$ be a bounded open domain with $H^1(M, \Z) = \Z_2$ and $H^2(M, \Z) = \Z_2$. For example, one can embed a connected CW complex $X$ with such cohomology groups into $\R^{2d}$ for large enough $d$, and let $M$ be a tubular neighbourhood of $X$. One can verify that the Bockstein homomorphism $\b : H^1(M,\Z_2) \to H^2(M, \Z)$ is non-zero. There exists on such an $M$ a non-trivial complex line bundle $L$, which is the complexification of a non-trivial real line bundle (see \cite[\S 15]{MilnorStasheff1974} for example). Now, let $\bS_1$ be a trivial spinor bundle over $M$, and let $E_1$ be the trivial $\C$-line bundle with the flat connection. Let $\bS_2$ be the complex spinor bundle associated to another spin structure such that $\bS_2 \cong \bS_1 \otimes L$, and let $E_2 = L^*$, where $L$ is also given a flat connection. It is easy to see that $\bS_1 \otimes E_1 \cong \bS_2 \otimes E_2$ as Clifford module bundles, and moreover, the twisted Dirac operators are isomorphic. By letting $\Pi$ be the standard chirality operator in even dimensions, one can check that the boundary conjugation maps are equivalent over an open set $\sU \sub \d M$.
\end{example}
\vspace{5pt}

Of course, if we know a priori that the line bundle $L$ must be trivial, then it follows that the map $\Phi_E$ constructed above is a unitary isomorphism from $E_1$ to $E_2$ and a global gauge equivalence of the connections. This is of course true if $\bS_1 \cong \bS_2$ or if $E_1 \cong E_2$. Moreover, if $E_1 \cong E_2$, then we can define a map $\Phi_S : \tilde{\bS}_1 \to \tilde{\bS}_2$ as a partial trace of $\Phi : \tilde{\bS}_1 \otimes \tilde{E}_1 \to \tilde{\bS}_2 \otimes \tilde{E}_1$, 
\begin{equation}
    \Phi_S := \frac{1}{N_E} \Tr_{E_1} \Phi.
\end{equation}
One checks that $\Phi_S|_U = j_S$, and so by analyticity, $\Phi$ yields a unitary isomorphism of the complex spinor bundles. Moreover, since $j_S$ intertwines the Clifford multiplications of $\tilde{\bS}_1$ and $\tilde{\bS}_2$ over $U$, it follows by analyticity that $\Phi_S$ does so over $\tilde{M}_1$. Therefore, $\Phi_S : \tilde{\bS}_1 \to \tilde{\bS}_2$ is indeed an isomorphism of Clifford modules, and Corollary \ref{corollary: global gauge equiv, intro} follows.

\vspace{5pt}

\begin{remark}[Recovering the $\Spin^c$ structure]\label{remark: recover spin c}
    It is known that there is an equivalence between $\Spin^c$ structures on a manifold and bundles of irreducible complex Clifford modules \cite{Plymen1986}. Thus, if the complex spinor bundles above are the canonical spinor bundles, then under the hypotheses of Corollary \ref{corollary: global gauge equiv, intro} we have that $\bS_1$ and $\bS_2$ are isomorphic as complex Clifford modules and are irreducible, and thus determine the same $\Spin^c$ structure. Note, however, that it is not generally possible to recover the underlying spin structure, as two spin structures may be isomorphic as $\Spin^c$-structures \cite[\S 2.5]{Friedrich2000}. It is interesting to note that in the real case, the equivalence between spin structures and bundles of irreducible real Clifford modules exists only in dimensions congruent to $6$, $7$, or $0$ mod $8$; in other dimensions the equivalence does not hold \cite{Lawson2023, Chen2017thesis}.
\end{remark}


\section*{Acknowledgements}

This work has received funding from the European Research Council (ERC) under the European Union's Horizon 2020 research and innovation programme through the grant agreement 862342. The author is also partially supported by the grants CEX2023-001347-S, RED2022-134301-T and PID2022-136795NB-I00.

\end{document}